\documentclass[10pt]{article}

\usepackage{amssymb,amsmath,amsthm}
\usepackage{MnSymbol}
\usepackage[mathscr]{euscript}
\usepackage{tikz-cd}
\usepackage{hyperref}
\hypersetup{
    colorlinks=true,
    linkcolor=[HTML]{0D47A1},		
    citecolor=[HTML]{0D47A1},		
    urlcolor=[HTML]{0D47A1}		    
}
\usepackage{fullpage}
\usepackage{rotating}
\usepackage[capitalize,nameinlink]{cleveref}
\usepackage{enumitem}
\usepackage{bbm}

\usepackage[shortcuts]{extdash}

\usepackage{tabularx}

\definecolor{mat}{HTML}{FFD6AD}
\definecolor{eric}{HTML}{ffadad}
\definecolor{andre}{HTML}{a9f5ae}
\definecolor{georg}{HTML}{faffad}

\newtheorem{thm}{Theorem}[subsection]
\newtheorem{theorem}[thm]{Theorem}
\newtheorem{lemma}[thm]{Lemma}
\newtheorem{proposition}[thm]{Proposition}
\newtheorem{corollary}[thm]{Corollary}

\theoremstyle{definition}

\newtheorem{definition}[thm]{Definition}

\newtheorem{examples}[thm]{Examples}
\newtheorem{remark}[thm]{Remark}

\newlist{examplenum}{enumerate}{1}
\setlist[examplenum]{label=\alph*), ref=\theproposition~(\alph*)}
\crefalias{examplenumi}{example}
\newlist{exmpenum}{enumerate}{1}
\setlist[exmpenum]{label=\alph*), ref=\theproposition~(\alph*)}
\crefalias{exmpenumi}{exmp}

\newtheorem{thmintro}{}
\newenvironment{thm-intro}[1]
  {\renewcommand\thethmintro{#1}\thmintro\itshape}
  {\endthmintro}
\newenvironment{thm-introref}[2]
  {\renewcommand\thethmintro{#1}\thmintro(#2)\itshape}
  {\endthmintro}
\newtheorem{defnintro}{}
\newenvironment{defn-intro}[1]
  {\renewcommand\thedefnintro{#1}\defnintro}
  {\enddefnintro}

\newcommand{\tto}{{\begin{tikzcd}[ampersand replacement=\&]{}\ar[r]\&{}\end{tikzcd}}}
\newcommand{\mto}{{\begin{tikzcd}[ampersand replacement=\&]{}\ar[r,mapsto]\&{}\end{tikzcd}}}
\newcommand{\stto}{{\begin{tikzcd}[ampersand replacement=\&, sep=small]{}\ar[r]\&{}\end{tikzcd}}}
\newcommand{\nstto}[1]{{\begin{tikzcd}[ampersand replacement=\&, sep=small]{}\ar[r,"{#1}"]\&{}\end{tikzcd}}}

\newcommand\pbmark{\ar[dr, phantom, "\ulcorner" very near start, shift right=1ex]}

\newcommand\pbmarkk{\ar[drr, phantom, "\ulcorner" very near start, shift right=1ex]}

\newcommand\fperp{\upModels}

\newcommand\lorth[1]{{}^\perp{#1}}
\newcommand\rorth[1]{{#1}^\perp}

\newcommand\Map[2]{\mathsf{Map}\left(#1,#2\right)}
\newcommand\Loc[2]{\mathsf{Loc}\left(#1,#2\right)}

\newcommand\PSh[1]{\mathsf{PSh}\left(#1\right)}
\newcommand\Sh[2]{\mathsf{Sh}\left(#1,#2\right)}
\newcommand\HSh[2]{\mathsf{HSh}\left(#1,#2\right)}

\newcommand\cA{\mathscr{A}}
\newcommand\cB{\mathscr{B}}
\newcommand\cC{\mathscr{C}}
\newcommand\cD{\mathscr{D}}
\newcommand\cE{\mathscr{E}}
\newcommand\cF{\mathscr{F}}
\newcommand\cG{\mathscr{G}}

\newcommand\cJ{\mathscr{J}}
\newcommand\cK{\mathscr{K}}
\newcommand\cL{\mathscr{L}}
\newcommand\cM{\mathscr{M}}

\newcommand\cO{\mathscr{O}}

\newcommand\cR{\mathscr{R}}
\newcommand\cS{\mathscr{S}}
\newcommand\cT{\mathscr{T}}

\newcommand\cV{\mathscr{V}}

\renewcommand\powerset[1]{\mathcal{P}\left(#1\right)}
\newcommand\calP{\mathcal{P}}

\newcommand\xto\xrightarrow
\newcommand\ot\leftarrow

\newcommand\Set{\mathsf{Set}}

\newcommand\Fin{\mathsf{Fin}}

\newcommand\CAT{\mathsf{CAT}}
\newcommand\op{^\mathrm{op}}
\newcommand\sfop{^\mathsf{op}}
\newcommand\slice[1]{_{/#1}}
\newcommand\ho[1]{\mathsf{ho}(#1)}
\newcommand\fun[2]{\left[#1,#2\right]}
\newcommand\LOC[2]{#1[#2^{-1}]}
\newcommand\LOCcc[2]{#1[#2^{-1}]_{\mathrm{cc}}}

\newcommand\Topos{\mathsf{Topos_{geom}}}
\newcommand\Toposop{\mathsf{Topos_{geom}^{op}}}
\newcommand\Toposalg{\mathsf{Topos_{alg}}}
\newcommand\Toposalghc{\mathsf{Topos_{alg}^{hc}}}

\newcommand\lex{^\mathrm{lex}}
\newcommand\rex{^\mathrm{rex}}
\newcommand\cclex{_\mathrm{cclex}}
\newcommand\alg{_\mathsf{alg}}
\newcommand\geom{_\mathsf{geom}}
\newcommand\oo{$\infty$\=/}
\newcommand\ooo{$(\infty,1)$\=/}

\newcommand\arr{^\to}

\newcommand\ac{^\mathsf{a}}
\renewcommand\cong{^\mathsf{c}}
\newcommand\gtop{^\mathsf{G}}

\newcommand\cov{^\mathsf{cov}}
\newcommand\hcov{^\mathsf{hcov}}
\newcommand\diag{^\Delta}
\newcommand\bc{^\mathsf{bc}}
\newcommand\topo{^\mathsf{top}}
\newcommand\cotop{^\mathsf{cotop}}

\newcommand\loc{^\mathsf{loc}}
\newcommand\cons{^\mathsf{cons}}
\newcommand\hyper{^\mathsf{hc}}

\newcommand\Iso{\mathsf{Iso}}
\newcommand\All{\mathsf{All}}
\newcommand\Mono{\mathsf{Mono}}
\newcommand\Surj{\mathsf{Surj}}
\newcommand\Trunc[1]{\mathsf{Trunc}_{#1}}
\newcommand\Conn[1]{\mathsf{Conn}_{#1}}
\renewcommand\Im{\mathrm{Im}}
\newcommand\im[1]{\mathsf{im}\left(#1\right)}
\newcommand\coim[1]{\mathsf{coim}\left(#1\right)}
\newcommand\surj[1]{\mathsf{coim}\left(#1\right)}

\newcommand\truncated[1]{^{\leq #1}}
\newcommand\connected[1]{_{> #1}}

\newcommand\Class{\mathsf{ClassMaps}}
\newcommand\ClassMono{\mathsf{ClassMono}}
\newcommand\LClass{\mathsf{LCMaps}}
\newcommand\LClassMono{\mathsf{LCMono}}
\newcommand\Acyclic{\mathsf{Acy}}
\newcommand\MAcyclic{\mathsf{MAcy}}
\newcommand\EAcyclic{\mathsf{EAcy}}
\newcommand\Cong{\mathsf{Cong}}
\newcommand\sg{_\mathsf{sg}}
\newcommand\TCong{\mathsf{TCong}}

\newcommand\HCong{\mathsf{HCong}}
\newcommand\GTop{\mathsf{GTop}}
\newcommand\LTTop{\mathsf{LTTop}}

\newcommand\CTop{\mathsf{CTop}}

\newcommand\TLoclex{\mathsf{TLoc}}

\newcommand\Cover[1]{#1^\mathsf{cov}}
\newcommand\HCover[1]{#1^\mathsf{hcov}}

\newcommand\dec[1]{\mathsf{D}\left(#1\right)}
\newcommand\decn[2]{\mathsf{D}^{#1}\left(#2\right)}
\newcommand\decinfty[1]{\mathsf{D}^\infty\left(#1\right)}

\renewcommand\S[1]{\mathscr S\left[#1\right]}

\newcommand\Dense[1]{\mathsf{Dns}(#1)}
\newcommand\Close[1]{\mathsf{Cls}(#1)}

\newcommand\forcing{:}
\newcommand\Forcing[2]{\left\lsem #1 \forcing #2\right\rsem}

\title{
Left-exact Localizations of \texorpdfstring{$\infty$-Topoi}{Infinity-Topoi} II:\\
Grothendieck Topologies
}

\author{
Mathieu Anel%
\footnote{Department of Philosophy, Carnegie Mellon University, mathieu.anel@protonmail.com} 
\and Georg Biedermann%
\footnote{Departamento de Matem\'{a}ticas y Estad\'{i}stica, Universidad del Norte, gbm@posteo.de}
\and Eric Finster%
\footnote{University of Birmingham, e.l.finster@bham.ac.uk}
\and Andr\'{e} Joyal%
\footnote{CIRGET, UQ\`AM. joyal.andre@uqam.ca} 
}
\date{}

\begin{document}

\maketitle

\begin{abstract}
We revisit the work of To\"en--Vezzosi and Lurie on Grothendieck topologies, using the new tools of acyclic classes and congruences.
We introduce a notion of extended Grothendieck topology on any \oo topos, 
and prove that the poset of extended Grothendieck topologies is isomorphic to that of topological localizations, hypercomplete localizations, Lawvere--Tierney topologies, and covering topologies (a variation on the notion of pretopology).
It follows that these posets are small and have the structure of a frame.
We revisit also the topological--cotopological factorization by introducing the notion of a cotopological morphism.
And we revisit the notions of hypercompletion, hyperdescent, hypercoverings and hypersheaves associated to an extended Grothendieck topology.

We also introduce the notion of forcing, which is a tool to compute with localizations of \oo topoi.
We use this in particular to show that the topological part of a left-exact localization of an \oo topos is universally forcing the generators of this localization to be \oo connected instead of inverting them.
\end{abstract}

\setcounter{tocdepth}{2}
\tableofcontents

\section{Introduction}

This paper continues the study of left-exact localizations of \oo topoi started in \cite{ABFJ:HS}.
The focus there was to provide tools to work with general localizations.
The focus here is to study the special case of localizations controlled by Grothendieck topologies by revisiting the work of To\"en--Vezzosi on hypercomplete localizations \cite{TV:hag1} and the work of Lurie on topological localizations \cite{Lurie:HTT}.

\medskip

Any \oo topos can be presented as a left-exact localization of an \oo category $\PSh C = \fun { C\op} \cS$ of presheaves of spaces over a small \oo category $ C$.
To\"en and Vezzosi introduce the notion of a Grothendieck topology on a small \oo category $ C$ as an ordinary Grothendieck topology on the homotopy 1-category $\ho C$.
They prove in \cite[Theorem 3.8.3]{TV:hag1}, that Grothendieck topologies on $ C$ are in bijective correspondence with the left-exact localizations of $\PSh C$ which are {\it t-complete}, that is {\it hypercomplete} in the sense of \cite[Section~6.2.5]{Lurie:HTT}
(a hypercomplete \oo topos is one in which the Whitehead theorem holds: every map inducing an isomorphism on homotopy sheaves is invertible).
On the other hand, Lurie proves in \cite[Proposition 6.2.2.17]{Lurie:HTT} that Grothendieck topologies on $ C$ are in bijective correspondence with {\it topological} localizations of $\PSh C$ which are those left-exact localizations that can be generated by inverting monomorphisms (rather than arbitrary maps).

A topological localization may not be hypercomplete and vice-versa.
But together, To\"en--Vezzosi's and Lurie's correspondences provide a non-trivial bijection between  hypercomplete localizations and topological localizations.

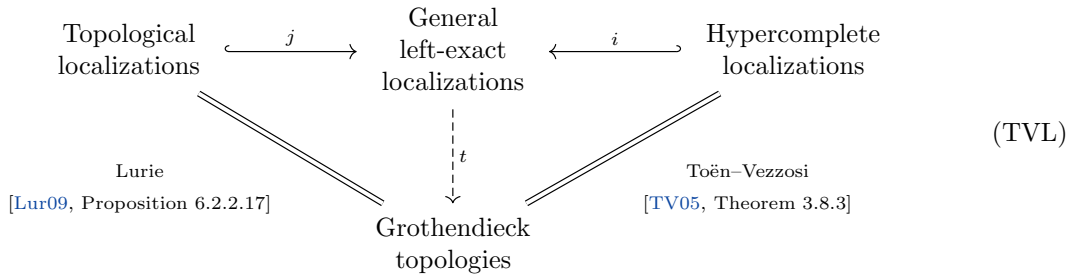
\begin{equation}
\label{eq:mainbij}
\tag{TVL}
\begin{tikzcd}
\begin{array}{c}\text{Topological}\\ \text{localizations}\end{array}
\ar[rrdd,equal,"\begin{array}{c}\text{Lurie}\\ \text{\cite[Proposition 6.2.2.17]{Lurie:HTT}}\end{array}"']
\ar[rr,hook,"j"]&&
\begin{array}{c}\text{General}\\\text{left-exact}\\ \text{localizations}\end{array}\ar[from=rr,hook',"i"'] \ar[dd,dashed,"t"]
&&
\begin{array}{c}\text{Hypercomplete}\\ \text{localizations}\end{array}
\\
\\
&&\begin{array}{c}\text{Grothendieck}\\ \text{topologies}\end{array}
\ar[rruu,equal,"\begin{array}{c}\text{To\"en--Vezzosi}\\ \text{\cite[Theorem 3.8.3]{TV:hag1}}\end{array}"']
\end{tikzcd}
\end{equation}
\noindent The purpose of this paper is to investigate further these correspondences.
We shall do so in the context of left-exact localizations of an {\it arbitrary} \oo topos $\cE$ and not necessarily a presheaf \oo topos.
We will introduce an extended notion of Grothendieck topology on an arbitrary \oo topos (and not only a presheaf topos, see \cref{sec:Grothendieck}) and define the map $t$ of the diagram, extracting an extended Grothendieck topology from any left-exact localization.
One of our main results 
will be to prove that the inclusions $i$ and $j$ are (up to the bijections of To\"en--Vezzosi and Lurie) respectively right and left adjoint of the map $t$ (\cref{thm:tripleadj}).
The composite $it$ takes a left-exact localization to its hypercompletion,
and the composite $jt$ takes a left-exact localization to its topological part in the sense of \cite[Proposition 6.5.2.19]{Lurie:HTT}.

\begin{center}
*    
\end{center}

We now describe our results in detail.
Our starting point is to define extended Grothendieck topologies on an arbitrary \oo topos $\cE$.
Let us say here that throughout the paper, we are working in the category of
\oo topoi and {\it algebraic morphisms} which are cocontinuous left-exact functors (see \cref{sec:topos}).

\begin{defn-intro}{\cref{definitionGrothtop}}
A class of monomorphisms $\cG$ in an \oo topos $\cE$ is an {\it extended Grothendieck topology} if     
\begin{enumerate}[label=\roman*)]
\item $\cG$ contains the isomorphisms and is closed under composition; 
\item $\cG$ is closed under base change and is a local class (\cref{locclass}); 
\item if the composite of two monomorphisms $u:A\to B$ and $v:B\to C$ belongs to $\cG$, then $v\in \cG$. 
\end{enumerate}
\end{defn-intro}

If $\phi:\cE\to \cF$ is an algebraic morphism of \oo topoi, then the class $\cG$ of monomorphisms in $\cE$ inverted by $\phi$ is an extended Grothendieck topology in $\cE$.
We prove in \cref{Gtop=Gtop} that an extended Grothendieck topology on the presheaf \oo topos $\PSh C$ is equivalent to a Grothendieck topology on the \oo category $ C$ in the sense of \cite[Definition 6.2.2.1]{Lurie:HTT}.
Moreover, the definition of an extended Grothendieck topology make sense in a 1-topos and not only in \oo topos.
Our first important result is the following characterization of extended Grothendieck topologies.

\begin{thm-intro}{\cref{thm:equivalences}}
There are canonical isomorphisms between
\begin{enumerate}
\item the poset $\GTop(\cE)$ of extended Grothendieck topologies on an \oo topos $\cE$,
\item the poset of Lawvere--Tierney topologies on the Lawvere object $\Omega$ of $\cE$ (\cref{LawvereTierney}),
\item the poset of extended Grothendieck topologies on the 1-topos $\cE\truncated 0$ of discrete objects of $\cE$, and
\item the poset of covering topologies on the \oo topos $\cE$ (\cref{def:covering-top}).
\end{enumerate}
The posets are small and have the structure of a frame.
\end{thm-intro}

Let us elaborate on the content of this theorem.
We define a {\it Lawvere--Tierney topology} by importing the classical definition from 1-topos theory: it is a closure operator $j:\Omega\to \Omega$ on the Lawvere object (aka the subobject classifier) of the \oo topos.
Such a topology provides a factorization system on the class of monomorphisms of the \oo topos $\cE$ which is stable under base change along arbitrary maps.
The bijection of Lawvere--Tierney topologies with extended Grothendieck topologies is done in \cref{LT-topoGroth} by showing that every extended Grothendieck topology is the left class of such a factorization system.

Since the Lawvere object $\Omega$ is discrete, the notion of Lawvere--Tierney topology on an \oo topos (and the notion of extended Grothendieck topologies since they are equivalent) depends only on the 1-topos of discrete objects $\cE\truncated 0\subseteq \cE$ (where it recovers the classical notion of a Lawvere--Tierney topology).
This remark generalizes the fact that a Grothendieck topology in the sense of \cite{Lurie:HTT, TV:hag1} on an \oo category $ C$ is a Grothendieck topology in the ordinary sense on the homotopy category $\ho C$ (see \cref{cor:TVGtop}).

The notion of {\it covering topology} seems new.
It is a variation on the notion of pretopology.
In the same way that an extended Grothendieck topology is a class of monomorphisms meant to be inverted, a covering topology is a class of maps meant to be surjective.
If $\phi:\cE \to \cF$ is an algebraic morphism of \oo topoi, then the inverse image by $\phi$ of the class of surjective maps in $\cF$ is a covering topology in $\cE$.
If $\cG$ is an extended Grothendieck topology, and $f = \im f \circ \coim f$ is the image factorization of a map $f$ (into a surjection followed by a monomorphism), then $f$ is a $\cG$-covering if $\im f\in \cG$.
The class $\Cover\cG$ of all $\cG$-coverings is a covering topology.
The topology can be recovered as the {\it covering sieves}, that is $\cG = \Cover\cG\cap \Mono$ and this is essentially the proof of the bijection between the two notions (\cref{thm:covering-top}).

\medskip

The motivation to introduce covering topologies is that it is sometimes more convenient to describe a localization by means of forcing some maps to be surjective than inverting some monomorphisms (for example, in logic, this corresponds to forcing existential axioms).
To handle the passage between the two kinds of conditions, we introduce in \cref{sec:forcing} the notion of {\it forcing}, which is a general theory for imposing universally conditions on maps (like becoming surjective, connected, truncated...).
All our examples of forcing will be equivalent to actual localizations, but presenting them in this more general setting provides efficient tools to navigate between the equivalent presentations of a localization (see \cref{thm:forcing}).

For example, if, for a class of maps $\Sigma$ in an \oo topos $\cE$, we define $\im\Sigma :=\{\im f\,|\, f\in \Sigma\}$, then, an example of forcing rewriting is 
$
\Forcing \Sigma\Surj = \Forcing {\im\Sigma}\Iso
$
which says that forcing a class $\Sigma$ to be surjective (by a left-exact localization of topoi) is equivalent to forcing the class $\im\Sigma$ to be invertible.
We shall particularly be interested with the forcing condition $\Forcing\Sigma{\Conn\infty}$ which means that we want to force the maps in the class $\Sigma$ to be \oo connected (e.g. in \cref{thm:meaningtopcotopfacto}).

If $\Theta$ is a class of maps existing uniformly in every \oo topos (like isomorphism, surjections, \oo connected maps...) and if $\Sigma$ is a class of maps in an \oo topos $\cE$, the forcing condition $\Forcing\Sigma\Theta$ may or may not be representable in the category of topoi.
If it is, we denote the corresponding \oo topos by $\cE\Forcing\Sigma\Theta$. 
In the previous examples, this would give $\cE\Forcing\Sigma\Surj$, $\cE\Forcing{\im \Sigma}\Iso$, or $\cE\Forcing\Sigma{\Conn\infty}$.
In particular, we shall use throughout the whole paper the notation $\cE\Forcing\Sigma\Iso$ for the left-exact localization generated by inverting a class of maps $\Sigma$ in $\cE$ (instead of the more classical notations $\cE[\Sigma^{-1}]$, $\Sigma^{-1}\cE$, $L(\cE,\Sigma)$...).

\medskip
An extended Grothendieck topology $\cG$ comes with a notion of sheaf, which is simply a local object for the class $\cG$.
It comes also with a notion of {\it hypercovering} (\cref{def:hypercovering}) and {\it hypersheaf} (\cref{def:GHsheaf}).
A hypercovering is a map $f$ for which all its iterated diagonals $\Delta^nf$ are $\cG$-coverings as above.
A hypersheaf is then a local object for the class $\HCover\cG$ of $\cG$-hypercovers.
Any hypersheaf is a sheaf but the converse is not necessarily true.
The subcategories $\HSh\cE\cG\subseteq \Sh\cE\cG\subseteq \cE$ of hypersheaves and sheaves are reflective and enjoy the following universal properties.

\begin{thm-intro}{\cref{prop:Gsheaf,prop:GHsheaf}}
The reflection $\cE\to \Sh\cE\cG$ is left-exact and universal for the following forcing conditions:
\[
\Sh\cE\cG
\ =\ 
\cE\Forcing \cG\Iso
\ =\ 
\cE\Forcing {\Cover\cG}\Surj
\ =\ 
\cE\Forcing {\HCover\cG}{\Conn\infty}\,.
\]
The reflection $\cE\to \HSh\cE\cG$ is left-exact and universally inverts all $\cG$-hypercoverings:
\[
\HSh\cE\cG
\ =\ 
\cE\Forcing {\HCover\cG}{\Iso}\,.
\]
In particular, both $\Sh\cE\cG$ and $\HSh\cE\cG$ are \oo topoi.
Moreover $\HSh\cE\cG$ is a hypercomplete topos and the reflection $\Sh\cE\cG \to \HSh\cE\cG$ is the hypercompletion of $\Sh\cE\cG$
in the sense of \cite[Proposition 6.5.2.13]{Lurie:HTT}.
\end{thm-intro}

\medskip
To build the connection between extended Grothendieck topologies and left-exact localizations, we describe the latter in terms of the classes of maps they invert, which we called {\it congruences} in \cite{ABFJ:HS}.
A congruence on an \oo topos $\cE$ is a class of maps $\cK$ 
which contains all isomorphisms, is closed by composition, and
is stable by colimits and finite limits in the arrow category $\cE \arr$.
If $\phi:\cE\to \cF$ is a cocontinuous left-exact functor between \oo topoi, the class $\cK_\phi=\phi^{-1}(\Iso(\cF))$ is a congruence on $\cE$.
If $\Sigma$ is a class of maps in an \oo topos $\cE$, it is contained in a smallest congruence denoted $\Sigma\cong$.
We proved in \cite[Proposition 4.2.3]{ABFJ:HS} that congruences are the same thing as the strongly saturated classes closed under base change introduced in \cite[Section~6.2.1]{Lurie:HTT}, but the definition of congruence is handier since it does not involve the 3-for-2 condition.
If $\phi:\cE\to \cF$ is a left-exact cocontinuous functor between topoi, then the class $\cK_\phi$ of maps of $\cE$ inverted by $\phi$ is a congruence on $\cE$.
Then we can deduce from results of \cite{Lurie:HTT} (see \cref{thm:bij-congruence-lexloc}), that, for an \oo topos $\cE$, the function $\phi\mapsto \cK_\phi$ defines an isomorphism
\[
\mathsf{LexLoc}(\cE)
\ =\ 
\Cong(\cE)
\]
between the poset of left-exact localizations of $\cE$
and the poset of congruences in $\cE$ ordered by inclusion.
(In this introduction we skip all accessibility questions, and refer to 
\cref{sec:left-exact-localization} for more details on this issue.)

We use this isomorphism to formulate the isomorphisms \eqref{eq:mainbij} in terms of congruences.
A congruence $\cK$ is said to be {\it topological}  if $\cK=\Sigma\cong$ for $\Sigma$ a class of monomorphisms, that is if the corresponding localization $\cE\to \cE\Forcing\cK\Iso$ is topological (see \cref{deftopcong} and \cite[Definition 6.2.1.4]{Lurie:HTT}).
We define a congruence $\cK$ to be {\it hypercomplete} if the corresponding localization $\cE\Forcing\cK\Iso$ is a hypercomplete \oo topos (\cref{def:hypercongruence}).
We denote by $\TCong(\cE)$ and $\HCong(\cE)$ the subposets of $\Cong(\cE)$ spanned by topological  and hypercomplete congruences.

If $\cK$ is a congruence, the intersection $\cK\cap \Mono$ is immediately an extended Grothendieck topology (and this fact was our motivation for introducing the notion).
This defines the morphism of posets $t:\Cong(\cE)\to \GTop(\cE)$ mentioned before.
Let us also introduce the poset $\CTop(\cE)$ of covering topologies, and the map $t:\Cong(\cE)\to \CTop(\cE)$ sending a congruence $\cK$ to its class of coverings $\Cover\cK = \{f\,|\,\im f \in \cK\}$ (i.e. the class of maps that become surjective in the localization by $\cK$).
Then, the isomorphisms of \eqref{eq:mainbij} enter the more complete diagram
\[
\begin{tikzcd}
\TCong(\cE)
\ar[rr, "j", hook]
&&
\Cong(\cE)
\ar[from=rr, "i"', hook']
\ar[ddl,"t = -\cap \Mono" description]
\ar[ddr,"\Cover{(-)}" description]
&&
\HCong(\cE)
\\
\\
&
\GTop(\cE)
\ar[rr,equal,"\text{\cref{thm:covering-top}}"']
\ar[luu,equal,"\text{\cref{thm:equivTcongGtop}}"]
&&\CTop(\cE)
\ar[ruu,equal,"\text{\cref{thm:equivPtopHcong}}"']
\end{tikzcd}
\]
The isomorphism of To\"en--Vezzosi between Grothendieck topologies and hypercomplete localizations is generalized to any \oo topos in \cref{generalTV} by composing the two equivalences of \cref{thm:covering-top} and \cref{thm:equivPtopHcong}.
And the isomorphism of Lurie between Grothendieck topologies and topological localizations is generalized in \cref{generalL}.

\medskip

The interpretation of the Diagram~\eqref{eq:mainbij} can now be stated properly.
\begin{thm-intro}{\cref{thm:tripleadj}}
The morphism of posets $t:=\Mono\cap -:\Cong(\cE) \to \GTop(\cE)$ admits 
\begin{enumerate}
\item a fully faithful left adjoint $j$ whose image is the subposet $\TCong(\cE)$ of topological congruences, and
\item a fully faithful right adjoint $i$ whose image is the subposet $\HCong(\cE)$ of hypercomplete congruences.
\end{enumerate}
\[
\begin{tikzcd}
\Cong(\cE) \ar[rr,"t" description] \ar[from=rr, shift left = 3,"i",hook'] \ar[from=rr, shift right = 3,"j"', hook']
&& \GTop(\cE)
\end{tikzcd}
\]
\end{thm-intro}
This triple adjunction defines a coreflection and a reflection.
For a congruence $\cK$, we have two other congruences
\[
\cK\topo := jt(\cK) \quad\subseteq\quad \cK
\quad\subseteq\quad
it(\cK) =: \HCover\cK
\]
which we call the {\it topological part}  of $\cK$ and the {\it hypercompletion} of $\cK$.
The congruence $\cK$ is topological if and only if $\cK=\cK\topo$ and hypercomplete if and only if $\cK=\HCover\cK$.
The corresponding localizations fit in a diagram
\[
\begin{tikzcd}
&\cE
\ar[ld,"\text{topological part}"']
\ar[d]
\ar[rd,"\text{hypercompletion}"]
\\
\cE\Forcing{\cK\topo}\Iso \ar[r]
&\cE\Forcing{\cK}\Iso  \ar[r]
&\cE\Forcing{\HCover\cK}\Iso
\end{tikzcd}
\]
where
$\cE\Forcing{\cK\topo}\Iso$ is the topological part of the localization $\cE\Forcing{\cK}\Iso$ 
in the sense of \cite[Proposition 6.5.2.19]{Lurie:HTT}, and
$\cE\Forcing{\HCover\cK}\Iso$ is the hypercompletion of $\cE\Forcing{\cK}\Iso$ 
in the sense of \cite[Section~6.5.2]{Lurie:HTT}.
In particular, the congruence $\HCover\cK$ can be understood as the class of maps in $\cE$ inverted in the hypercompletion $\cE\Forcing{\cK}\Iso$, that is the class of maps in $\cE$ that become \oo connected in $\cE\Forcing{\cK}\Iso$.
In \cref{sec:hypercoverings}, we characterize $\HCover\cK$ as the class of all hypercoverings for the topology $\cK\cap \Mono$ .

\medskip

From such a triple adjunction, we can deduce formally that the fixed points of $(-)\topo$ and $\HCover{(-)}$ are are isomorphic large posets, and we get the following explicit correspondence between topological and hypercomplete localizations.

\begin{thm-intro}{\cref{thm:adjTCongHCong}}
The following adjunction is an isomorphism of posets
\[
\begin{tikzcd}
\TCong(\cE) \ar[rr,shift left = 1.6, "\HCover{(-)}"] \ar[from=rr,shift left = 1.6, "(-)\topo"]
&&\HCong(\cE)\,.
\end{tikzcd}
\]
\end{thm-intro}
\noindent (In the logic of the paper though, we proceed the other way. 
We start by defining explicitly $\cK\topo$ and $\HCover\cK$, then we prove \cref{thm:adjTCongHCong} as a step toward \cref{thm:tripleadj}.)

\bigskip

As an application of this setting, we revisit the topological--cotopological factorization of left-exact localization introduced in \cite[Proposition 6.5.2.19]{Lurie:HTT}.
First, we use our notion of forcing to get the following interpretation of the topological part of a localization.

\begin{thm-intro}{\cref{thm:meaningtopcotopfacto}}
Let $\Sigma$ be a set of maps in an \oo topos $\cE$.
Then the topological part of the localization 
$\cE\to \cE\Forcing\Sigma\Iso$
is the localization $\cE\to \cE\Forcing\Sigma{\Conn\infty}$
universally forcing the maps in $\Sigma$ to be \oo connected.     
\end{thm-intro}

\noindent Recall from \cite[Section~6.2.5]{Lurie:HTT} that a localization is called {\it cotopological} if it inverts only \oo connected maps (see \cref{prop:caractopcong} for other equivalent characterizations).
Then, the topological--cotopological factorization of a localization $\cE\Forcing\cK\Iso$ corresponds to the factorization
\[
\begin{tikzcd}
\cE \ar[rr,"\phi"] \ar[rd,"\phi\topo"']&&\cE\Forcing\cK\Iso \\
& \cE\Forcing\cK{\Conn\infty} \ar[ru,"\phi\cotop"']
\end{tikzcd}
\]
In other words, the topological part forces the map in $\cK$ to be \oo connected and the cotopological part inverts these \oo connected maps, thus fully inverting the maps in $\cK$.

For an application, we consider $\S X = \fun \Fin \cS$ (where $\Fin$ is the \oo category of finite spaces).
The topos $\S X$ is freely generated by a single object $X$ in the sense that an algebraic morphism $\phi:\S X\to \cE$ to another topos $\cE$ is entirely determined by its value $\phi(X)$ in $\cE$ (see \cref{def:freetopos}).
The universal object $X$ is the canonical inclusion $X:\Fin\to \cS$.
If $\cK= \{X\to 1\}\cong$ is the congruence generated by the object $X$, the corresponding localization is the functor $\S X\to \cS$ sending $F:\Fin \to \cS$ to $F(1)$.
The topological part of this localization is the morphism $\S X \to \S {X\connected \infty}$, where $\S {X\connected \infty}$ is the \oo topos freely generated by an \oo connected object (see \cref{ex:toppart:ooconn}).

\smallskip

We then generalize the topological--cotopological factorization to arbitrary morphisms of topoi.
By considering the factorization of an algebraic morphism of \oo topoi $\phi:\cE\to \cF$ into a left-exact localization followed by a conservative morphism (\cref{factloccons}), we can use Lurie's factorization on the localization part to get a triple factorization $\phi = \phi\cons\circ \phi^\mathsf{cotop.loc} \circ \phi \topo$.
\begin{equation}
\begin{tikzcd}
\cE
\ar[rr,"\phi"] \ar[dd,"\phi\topo"']
\ar[rrdd,"\phi\loc"', near start]
&&\cF
\\
\\
\cE\Forcing{\cK\topo}\Iso
\ar[rr,"(\phi\topo)\cotop"']
\ar[rruu,"\phi\cotop"', crossing over, near end]
&& \cE\Forcing\cK\Iso \ar[uu,"{\phi\cons}"']\,.
\end{tikzcd}
\end{equation}
This suggest the introduction of the notion of a {\it cotopological morphism} as the composite $\phi\cons\circ (\phi\topo)\cotop$ of a cotopological localization followed by a conservative morphism.
In particular any conservative morphism is cotopological.
We define cotopological morphisms in \cref{def:cotopcong} and prove in \cref{morphismvslemmacotop} that they can be characterized as the morphisms reflecting \oo connected maps (making them a weaker version of conservative functors).
Then we prove in \cref{facttopcotop} that the pair of classes (topological localizations, cotopological morphisms) form a factorization system on the category of topoi.
Geometrically, this triple factorization system is a refinement of the image factorization (see \cref{rem:image}).

\medskip
Finally, let us say a word on the main technical device used throughout the paper, which is the notion of an {\it acyclic class} introduced in \cite{ABFJ:HS}, inspired by the notion of modality of Homotopy Type Theory \cite{RSS}.
A class of maps $\cA$ in an \oo topos $\cE$ is acyclic if it 
contains all isomorphisms, is closed by composition, and
is stable by colimits in the arrow category $\cE \arr$.
Acyclic classes abound in \oo topos theory and homotopy theory.
Any congruence is acyclic. 
The class of surjections and $n$-connected maps in $\cE$ are all acyclic.
Extended Grothendieck topologies are not acyclic classes, but the associated covering topologies are.
And, most importantly in this paper, the acyclic class generated by a class of monomorphisms is always a congruence (\cref{sigmaac=sigmacong}).
We recall their definition and develop a number of new results in \cref{sec:acyclic,sec:acvcong}.

\bigskip

\noindent {\bf Acknowledgments:} 
The authors would like to thank 
Reid Barton,
Dan Christensen, 
Simon Henry,
Egbert Rijke, and 
Mike Shulman,
for useful discussions on the material of this paper.
We would also like to thank the reviewer for the high quality of his report.
The first author gratefully acknowledges the support of the 
Air Force Office of Scientific Research through grant FA9550-20-1-0305.
The last author acknowledges the support of the 
Natural Sciences and Engineering Research Council of 
Canada through grant 371436.

\section{Preliminaries}
\label{sec:preliminaries}

\subsection{Conventions, notations and miscellaneous}
\label{sec:convents}

Throughout the paper, we use the language of higher category theory.
We will simplify the vocabulary and drop the prefix ``$\infty$'' when referring to higher categories and their associated notions.
The word \emph{category} refers to \ooo category, and all constructions are assumed to be homotopy invariant.
When necessary, we shall refer to an ordinary category as a \emph{1-category} and to an ordinary Grothendieck topos as a \emph{1-topos}.
Furthermore, we work in a model independent style, which is to say, we do not choose an explicit combinatorial model for \ooo categories such as quasicategories, but rather give arguments which we feel are robust enough to hold in any model. 
We will refer to the work of Lurie \cite{Lurie:HTT} for the general theory of \oo categories and \oo topoi.
Other references are \cite{Cisinski} and \cite{Riehl-Verity:EICT}.

\medskip

We use the word \emph{space} to refer generically to a homotopy type or \oo groupoid.
We denote the category of spaces by $\cS$.
We shall say that a map between two spaces $f:X\to Y$ is an {\it isomorphism} if it is a homotopy equivalence.
We say an object is \emph{unique} if the space it inhabits is contractible.
For example, the inverse of an isomorphism is unique in this sense.

\medskip

We shall denote by $\cC(A,B)$ or by $\mathsf{Map}_{\cC}(A,B)$ the space of maps between two objects $A$ and $B$ of a category $\cC$ and write $f:A\to B$ to indicate that $f \in \cC(A,B)$.
We write $A \in\cC$ to indicate that $A$ is an object of $\cC$. 
 The \emph{opposite} of a category $\cC$ is denoted $\cC\op$ and defined by the fact that $\cC\op(B,A):=\cC(A,B)$ with its category structure inherited from $\cC$.
We write $\cC\slice{A}$ for the slice category of $\cC$ over an object $A$. 
If $f : X \to A$ is a morphism of $\cC$, we often write $(X,f) \in \cC\slice{A}$, as it is frequently convenient to have both the object and structure map visible when working in a slice category.
If a category $\cC$ has a terminal object, we denote it by $1$.
Every category $\cC$ has a {\it homotopy category} $\ho\cC$ which is a 1-category with the same objects as $\cC$, but where $\ho\cC(A,B)=\pi_0\cC(A,B)$.
We shall say that a morphism $f:A\to B$ in $\cC$ is {\it invertible}, or that it is an {\it isomorphism}, if the morphism is invertible in the homotopy category $\ho\cC$. 
We make a small exception to this terminology with regard to equivalence of categories: we continue to employ the more traditional term {\it equivalence}.
We shall say that a functor $F:\cC\to \cD$ is {\it essentially surjective} if for every object $X\in\cD$ there exists an object $A\in\cC$ together with an isomorphism $X\simeq FA$.
We shall say that $F$ is {\it fully faithful} if the induced map $\cC(A,B)\to \cD(FA,FB)$ is invertible for every pair of objects $A,B\in\cC$.
We shall say that $F$ is an {\it equivalence} (of categories) if it is fully faithful and essentially surjective.
We assume that all subcategories and classes of maps in a category are defined by properties which are invariant under isomorphism, and consequently {\it we adopt the convention that all subcategories are replete}.
We denote the category of functors from $\cC$ to $\cD$ alternatively by $\fun \cC \cD$ or $\cD^{\cC}$ as seems appropriate from the context.
For a small category $ C$, we will write $\PSh C := \fun { C\op} \cS$ for the category for presheaves on $ C$.
Recall that the {\it Yoneda functor} $Y: C\to \PSh C$ is defined by putting $Y(A)(B):=\Map B A$ for objects $A,B\in  C$.

\medskip

All limits and colimits are homotopy limits and colimits; in particular, all pullback squares are homotopy pullbacks and all pushout squares are homotopy pushouts. 
A category $\cE$ is {\it complete} if any small diagram $I\to \cE$ has a limit; a functor is {\it continuous} if it preserves all small limits.
A category $\cE$ is {\it finitely complete}, or {\it lex}, if it has a terminal object and all pullbacks;
a functor between lex categories is {\it left-exact}, or {\it lex}, if it preserves terminal objects and pullbacks.
Dually, a category $\cE$ is {\it cocomplete} if any small diagram $I\to \cE$ has a colimit;
a functor is {\it cocontinuous} if it preserves all small colimits. 
A category $\cE$ is {\it finitely cocomplete}, or {\it rex}, if it has an initial object and all pushouts;
a functor between rex categories is {\it right-exact}, or {\it rex}, if it preserves initial objects and pushouts.
The category of presheaves $\PSh C=\fun { C\op} \cS$ on a small category $ C$ is cocomplete and the Yoneda functor $Y: C\to \PSh C$ exhibits the {\it free cocompletion} of $ C$ \cite[Theorem 5.1.5.6]{Lurie:HTT}.

\medskip

We shall say that a space $X\in \cS$ is {\it finite}
if it has the homotopy type of a $CW$-complex
with a finite number of cells.
We shall denote the category of finite spaces 
by $\Fin$; it is the smallest full subcategory
of $\cS$ which is closed under
finite colimits (=which is closed under pushout
and contains the initial object) 
and which contains the space $1\in \cS$.
The rex category $\Fin$ is actually
freely generated by the object $1\in  \cS$.
More precisely, for every object $A$ in a rex category $\cE$
there exists a unique rex functor $\phi_A:\Fin\to \cE$
such that $\phi_A(1)=A$. 
By construction, $\phi_A(F)=\bigsqcup_F A$ is the colimit of the constant diagram $c(A):F\to \cE$ with values $A$. 
Every small category $\cC$
generates freely a rex category $\cC\rex$.
By construction, $\cC\rex$
is the smallest full rex subcategory of $\PSh C$
which contains representable functors.
The Yoneda functor 
$Y:\cC\to \PSh C$
induces the functor $y:\cC\to \cC\rex$
which exhibits the 
{\it free rex completion} of $\cC$.
(This is a special case of \cite[Theorem~5.3.6.2]{Lurie:HTT}.)

\smallskip
Dually, the lex category $\Fin\op$ is 
freely generated by the object $1\op$.
More precisely, for every object $A$ in a lex category $\cE$ there exists a unique lex functor $\phi_A:\Fin\op\to \cE$ such that $\phi_A(1\op)=A$. 
By construction, $\phi_A(F\op)=\prod_F A$ is the limit of the constant diagram $c(A):F\to \cE$ with values $A$. 
Every small category $\cC$ generates freely a lex category $\cC\lex$.
By construction, $(\cC\lex)\op=(\cC\op)\rex$.

\medskip  
  
For an object $A$ of a category $\cC$ with finite limits, we will write $\Delta(A)=(1_A,1_A):A\to A\times A$ for the canonical map, which we refer to as the \emph{diagonal of $A$}.
More generally, the \emph{diagonal} of a map $u:A\to B$ is defined to be the canonical map $\Delta(u)=(1_A,1_A):A\to A\times_B A$
\[
  \begin{tikzcd}
    A \ar[dr, "\Delta(u)"] \ar[drr, "1_A", bend left] \ar[ddr, "1_A"', bend right] && \\
    & A\times_B A\ar[r, "p_2"] \ar[d, "p_1"'] \pbmark & A\ar[d, "u"] \\
    & A \ar[r, "u"']  & B
  \end{tikzcd}
\]
induced by the universal property of the pullback.
This construction can be iterated, and we use the notation $\Delta^n(u)$ for the $n$-th iterated diagonal of a map,
starting with $\Delta^0(u)=u$.
The $n$-th iterated diagonal $\Delta^n(A)$
of an object $A$ is defined to be $\Delta^n(A\to 1)$.
The map $\Delta^n(A):A\to A^{S^{n-1}}$
is induced to the map $S^{n-1}\to 1$, where $S^{n-1}$ is the $(n-1)$-sphere.

\medskip
When a functor $F:\cC\to \cD$ is left adjoint to a functor $G:\cD\to \cC$, we shall write $F\dashv G$.
When representing adjoint functors horizontally, our convention will be that the functor on top is left adjoint to the one below.
For example, for three adjoint functors $F\dashv G\dashv H$, we shall write 
\[
\begin{tikzcd}
\cC \ar[rr,"G" description] \ar[from=rr, shift right = 3,"F"'] \ar[from=rr, shift left = 3,"H"]
&& \cD
\end{tikzcd}
\]
Beware that, with this convention, the left adjoint functors are not always oriented from the left to the right (and vice-versa for right adjoints)

\subsection{Topoi, congruences and acyclic classes}

\subsubsection{Localizations}
\label{sec:localizations}
A functor $F:\cE\to \cF$ is said to {\it invert} a map $f\in \cE$ if the map $\phi(f)\in \cF$ is invertible, and to invert a class of maps $\Sigma \subseteq \cE$ if it inverts all maps in $\Sigma$.
The functor $F$ is said to be a {\it $\Sigma$-localization}, or to invert $\Sigma$ {\it universally}, if it is initial in the category of functors which invert $\Sigma$.
We shall say that $F$ is a {\it localization} if it is a $\Sigma$-localization with respect to some class of maps $\Sigma\subseteq \cE$ (equivalently, if it is a localization with respect to the class of all maps inverted by $F$).
If $\Sigma$ is a class of maps in a category $\cE$, then the codomain $\cF$ of any $\Sigma$-localization $\cE\to \cF$ is unique up to equivalence of categories, and we denote the codomain $\cF$ generically by $\LOC\cE\Sigma$.

If $\cE$ and $\cF$ are cocomplete categories and $\Sigma \subseteq \cE$ is a class of maps in $\cE$, then a cocontinuous functor $F : \cE \to \cF$ is said to be a {\it cocontinuous $\Sigma$-localization}, or to invert $\Sigma$ {\it universally among cocontinuous functors} if it is initial in the category of functors which invert $\Sigma$.
More precisely, this means that if a cocontinuous functor $G:\cE\to \cG$ (with values in a cocomplete category) inverts $\Sigma$, then there exists a unique pair $(G',\alpha)$ where $G':\cF\to \cG$ is a cocontinuous functor and $\alpha$ is an isomorphism $G\simeq G'\circ F$.
We denote generically by $\LOCcc\cE\Sigma$ the codomain of the cocontinuous localization with respect to $\Sigma$.

\medskip

We shall say that a functor $\rho:\cE\to \cF$ is a {\it reflector}, or a {\it reflection} if it has a fully faithful right adjoint $\iota: \cF\to \cE$.
A full subcategory $\cE'$ of $\cE$ is called {\it reflective} if the inclusion functor $\iota:\cE'\hookrightarrow \cE$ has a left adjoint $\rho:\cE\to \cE'$.
Beware that \cite[Definition 5.2.7.2]{Lurie:HTT} defines a \emph{localization} to be what we have here called a \emph{reflection}.
Any reflection is a localization, and conversely a localization is a reflection as soon as it has a right adjoint \cite[Proposition 2.2.1]{ABFJ:HS}.

Recall from \cite[Definition 5.5.4.1]{Lurie:HTT} that an object $X$ in a category $\cE$ is said to be \emph{local} with respect to a map $u:A\to B$ in $\cE$ if the map 
\[
\Map u X : \Map B X \to \Map A X
\]
is invertible.
The object $X$ is said to be local with respect to a class of maps $\Sigma\subseteq \cE$ if it is local with respect to every map in $\Sigma$.
We shall denote by $\Loc \cE \Sigma \subseteq \cE$ the full subcategory spanned by the $\Sigma$-local objects.

\medskip

\begin{definition}[Accessibility, Presentability]
\label{def:accessible}
We shall say that a reflector $\phi:\cE\to \cF$ is {\it accessible} if $\cF = \Loc \cE \Sigma$ for a set of maps $\Sigma$ in $\cE$.
We shall say that such a category $\cF$ is an accessible reflection of $\cE$.
A category $\cE$ is said to be \emph{presentable} if it is an accessible reflection of a presheaf category $\PSh C$ for a small category $ C$.
In particular, $\PSh C$ is presentable.

When $\cE$ is a presentable category, the notion of accessible reflector is equivalent to the notion of accessible localization defined in \cite{Lurie:HTT}.     
\end{definition}

\begin{proposition}[{\cite[Propositions 5.5.4.15 and 5.5.4.20]{Lurie:HTT}}]
\label{localizationofpresentable}
If $\Sigma$ is a set of maps in a presentable category $\cE$, then the full subcategory $\Loc \cE \Sigma \subseteq \cE$ of $\Sigma$-local objects is presentable, reflective, and the reflector $\cE\to \Loc \cE \Sigma$ is a cocontinuous localization $\cE\to \LOCcc \cE \Sigma$.
\end{proposition}

\subsubsection{Topoi and algebraic morphisms}
\label{sec:topos}

If $\Sigma$ is a set of maps in a presheaf category $\PSh C$ we shall say that the reflector
$\rho:\PSh C\to \Loc {\PSh C} \Sigma$ is a {\it left-exact reflection} if it preserves finite limits.

\begin{definition}[Topos {\cite[Definitions 6.1.0.4 and 6.3.1.1]{Lurie:HTT}}]
A category $\cE$ is a {\it topos} if it is an accessible left-exact reflection of the category $\PSh C$ of presheaves over a small category $ C$.
A {\it geometric morphism of topoi } $\cF\to\cE$ is a functor $\phi_*:\cF\to \cE$ admitting a left adjoint $\phi^*$ which is a left-exact functor.
We denote by $\fun\cF\cE\geom$ the category of geometric morphisms $\cF\to \cE$,
and by $\Topos$ the category of topoi and geometric morphisms.
\end{definition}

The theory of topoi has two sides, a geometric side and an algebraic side \cite{Anel-Joyal:topo-logie}, \cite[Remark 6.1.1.3]{Lurie:HTT}.
Since our work is focused on the algebraic side, we will work with the ``algebraic'' category of topoi.

\begin{definition}(Algebraic morphisms)
\label{def:algebraicmorphism}
We define an {\it algebraic morphism of topoi} $\phi:\cE\to \cF$ as a cocontinous and left-exact functor.
We denote by $\fun \cE\cF\alg$ the category of algebraic morphisms $\cE\to \cF$.
We denote the category of topoi and algebraic morphisms by $\Toposalg$.
\end{definition}

Since topoi are presentable categories, any algebraic morphism $\phi:\cE\to \cF$ has a right adjoint $\phi_*$ which defines a geometric morphism $\cF\to \cE$.
This provides an equivalence $\fun \cE\cF\alg = \fun \cF\cE\geom\sfop$.
This also provides an identification $\Toposalg=\Toposop$.
In the notations of \cite{Lurie:HTT}, we have 
\begin{align*}
\fun \cF\cE\geom &= \mathrm{Fun}_*(\cF,\cE) &\Topos &= \mathrm{\cR\cT op}  \\
\fun \cE\cF\alg &= \mathrm{Fun}^*(\cE,\cF)  &\Toposalg &= \mathrm{\cL\cT op}\ .
\end{align*}

The algebraic category of topoi has the advantage to have a nice forgetful functor to the category of large categories $\Toposalg\to \CAT$.
This functor has a left adjoint defined on small categories.
If $C$ is a small category, recall that we denote $C\lex$ the completion of $C$ for finite limits.
This category is still small and the presheaf category $\PSh{C\lex}$ is a topos.
The following result is \cite[Proposition 6.1.5.2]{Lurie:HTT} and \cite[Proposition 2.3.2]{Anel-Lejay:topos-exp}.

\begin{theorem}[Algebraically free topos]
Let $\cE$ be a topos and $C$ a small category
The restriction along the composite functor $C\to C\lex \to \PSh{C\lex}$ induces an equivalence of categories
\[
\fun {\PSh{C\lex}}\cE\alg
\ =\ 
\fun C \cE \,.
\]
\end{theorem}

\begin{definition}[Algebraically free topos]
\label{def:freetopos}
Following \cite{Anel-Lejay:topos-exp}, we shall denote the topos $\PSh{C\lex}$ by $\S C$ and call it the {\it algebraically free topos on $C$}.
\end{definition}

When $C=1$ is the terminal category, we denote the algebraically free topos on 1 by $\S X$.
This topos is known as the ``object classifier'' since its universal property says
\[
\fun {\S X}\cE\alg
\ =\ 
\cE\,.
\]
In other words, an algebraic morphism $\S X\to \cE$ is the same thing as an object of $\cE$.    
The ``universal object'' $X$ is the functor represented by the terminal object $1\in 1\lex = \Fin$, that is the canonical inclusion $\Fin\to \cS$.

Any left-exact localization of $\S X$ corresponds to a property that can be enforced universally on $X$ and that is preserved by algebraic morphisms.
We refer to \cite[Section~5]{ABFJ:HS} for a detailed study of the localizations forcing $X$ to become $n$-truncated or $n$-connected.
We shall use these as examples throughout the paper.

\subsubsection{Surjections and connected maps}
\label{sec:surjection}
\label{sec:connected}
\label{sec:truncated}

\begin{definition}[Monomorphisms and surjections]
\label{def:surjection}
A map $f:X\to Y$ in a topos $\cE$ is a {\it monomorphism} if the square
\[
\begin{tikzcd}
X \ar[r, equal] \ar[d, equal] & X \ar[d, "f"] \\
X\ar[r, "f"'] \ar[r] & Y
\end{tikzcd}
\]
is a pullback.
We denote by $\Mono(\cE)$ the class of monomorphisms in $\cE$.
We shall say that a map $f: X \to Y$ in a topos $\cE$ is \emph{surjective}, or that $f$ is a \emph{surjection}, or a {\it cover}, if it is left orthogonal to $\Mono(\cE)$ (surjective maps are called \emph{effective epimorphisms} in \cite{Lurie:HTT}).
We denote by $\Surj(\cE)$ the class of surjections in $\cE$.
A family of maps $f_i:X_i\to Y$ is said to be a {\it surjective family} if the corresponding map $\coprod_iX_i\to Y$ is surjective.
\end{definition}

If $f:A\to B$ is a map in a topos $\cE$.
We define the {\it nerve} of $f$ to be the simplicial diagram in $N(f):\Delta\op \to \cE\slice B$
sending $[n]$ to $(A,f)^{\times n+1}$, the $(n+1)$-iterated product of $(A,f)$ in $\cE\slice B$ (i.e. the iterated fiber product over $B$ in $\cE$).
The colimit of $N(f)$ is denoted $\im f$ and called the {\it image} of $f$.
Its domain is denoted $\Im(f)$. 
The canonical map $A\to \Im(f)$ will be denoted $\coim f$ and called the {\it coimage} of $f$.
It can be proven that the image of $f$ is a monomorphism \cite[Proposition 6.2.3.4]{Lurie:HTT} and that the coimage is surjective:

\begin{proposition}[{\cite[6.2.3]{Lurie:HTT}, \cite[Lecture 4]{Rezk:Leeds}}]
\label{prop:charac-surjection}
If $f:A\to B$ is a map in a topos $\cE$. 
The following conditions are equivalent:
\begin{enumerate}
\item $f$ is surjective;
\item the colimit of $N(f)$ is terminal in $\cE\slice B$.
\end{enumerate}
\end{proposition}

By definition of the image $\im f$, the map $f$ factors into
\[
\begin{tikzcd}
A \ar[rr,"f"] \ar[rd,"\coim f "']&& B \\
& \Im(f) \ar[ru,"\im f"']
\end{tikzcd}
\]
Using \cref{prop:charac-surjection} one can prove that the map $\coim f$ is a surjection (called the {\it surjective part} of $f$).
We shall refer to this factorization as the {\it image factorization} of $f$.
The pair $(\Surj(\cE),\Mono(\cE))$ is an example of a modality, which is a factorization system on $\cE$ stable under base change.
We refer to \cref{sec:modality} for a quick reminder on factorization systems and modalities. 
For more details we refer to \cite[Section~5.2.8]{Lurie:HTT} or \cite[Section~3.1]{ABFJ:HS} for factorization systems, and to \cite[Section~3.2]{ABFJ:HS} for modalities.
In particular, a map which is both a surjection and a monomorphism is an isomorphism.
Any algebraic morphism preserves colimits and finite limits, thus it preserves monomorphisms and also surjections by \cref{prop:charac-surjection}.
Thus, any algebraic morphism preserves the image factorization.

\bigskip
\begin{definition}[Truncated and connected maps]
\label{def:truncated}
\label{def:connected}
For any $-1\leq n<\infty$, an object $X$ in a topos $\cE$ is said to be {\it $n$-truncated} if the diagonal map 
$X\to X^{S^{n+1}}$ is invertible.
A map $f$ is said to be {\it $n$-truncated} if the map 
$\Delta^{n+2}f$ is invertible.
An object  $X$ is said to be  {\it $n$-connected}
if the diagonal $X\to X^{S^{k}}$ is surjective
for every $-1\leq k\leq n$  (with the convention that $S^{-1} =\emptyset$). 
A map $f$ is said to be {\it $n$-connected} if the map 
$\Delta^k(f)$ is surjective for every $0\leq k\leq n+1$.
\end{definition}

Beware that an $n$-connected map in our sense is $(n+1)$-connected in the conventional topological indexing and is called $(n+1)$-connective in \cite{Lurie:HTT}.
We refer to \cite[Section~6.5.1]{Lurie:HTT} and \cite[Section~3.3]{ABFJ:GBM} for a study of properties of truncated and connected maps in a topos.

\medskip
If $\Trunc n$ (resp. $\Conn n$) denotes the class of $n$-truncated maps (resp. $n$-connected maps), then the pair $(\Conn n,\Trunc n)$ is another example of modality \cite[Example 3.4.2(2)]{ABFJ:GBM}.
In particular, a map which is both $n$-connected and $n$-truncated is an isomorphism.
Notice that $\Conn {-1} = \Surj$ and $\Trunc {-1} = \Mono$.
Any algebraic morphism preserves diagonals and surjective maps, therefore it preserves the two classes of $n$-connected and $n$-truncated maps, and the $n$-connected--$n$-truncated factorization.

\begin{definition}[\oo connected map, hypercomplete topos]
\label{def:oo-connected}
A map $f:X\to Y$ is {\it \oo connected}
if it is $n$-connected for all $n$.
We denote by $\Conn\infty(\cE)$ the class of \oo connected maps in $\cE$.
An object of $\cE$ is {\it hypercomplete} if it is local with respect to $\Conn\infty(\cE)$.
The subcategory of hypercomplete objects is denoted $\cE\hyper:=\Loc \cE {\Conn\infty}$.
A topos $\cE$ is {\it hypercomplete} if $\cE=\cE\hyper$ if and only if all \oo connected maps are invertible ($\Conn\infty(\cE) = \Iso(\cE)$).
We denote $\Toposalghc\subseteq\Toposalg$ the full subcategory spanned by hypercomplete topoi.
The next result says that it is a reflective subcategory.
\end{definition}

\begin{proposition}[{\cite[Lemmas 6.5.2.10, 6.5.2.12, and Proposition 6.5.2.13]{Lurie:HTT}}]
\label{prop:luriehypercompletion}
For a topos $\cE$, the full subcategory $\cE\hyper \subseteq \cE$ is reflective, the reflection is left-exact, and  the class of maps inverted by the reflector $\rho:\cE\to \cE\hyper$ is exactly $\Conn\infty$.
In particular $\cE\hyper$ is a topos and it is hypercomplete.
Moreover, $\rho$ is the reflection of $\cE$ in hypercomplete topoi.
\end{proposition}

The topos $\cE\hyper$ is called the {\it hypercompletion} of $\cE$.

\subsubsection{Congruences and left-exact localizations}
\label{sec:left-exact-localization}

The class of maps $f\in \cE$ inverted by an algebraic morphism $\phi:\cE\to \cF$ is a {\it congruence} $\cK_\phi \subseteq \cE$ in the following sense.

\begin{definition}[Congruence {\cite[Definition 4.2.1]{ABFJ:HS}}]
\label{defcongruenceclass2}
We say that a class of maps $\cK$ in a topos $\cE$ is a {\it congruence} if the following conditions hold:
\begin{enumerate}[label=\roman*)]
\item $\cK$ contains the isomorphisms and is closed under composition;
\item ${\cK}$ is closed under colimits and finite limits (in the arrow category of $\cE$).
\end{enumerate}
A class of maps $\cK\subseteq \cE$ is a congruence if and only if it is closed under base changes
and strongly saturated in Lurie's sense \cite[Proposition 4.2.3]{ABFJ:HS}.
In particular, any congruence satisfies the 3-for-2 property.
\end{definition}

\begin{examples}
\label{ex:congruence} 
Let $\cE$ be a topos.
\begin{examplenum}
\item\label{exmpcongruence1} 
The classes $\Iso$ and $\All$ of isomorphism and all maps in $\cE$ are respectively the smallest and the largest congruences (for the inclusion relation).

\item\label{exmpcongruence4} 
Let $\phi:\cE\to \cF$ be an algebraic morphism of topoi.
Recall that algebraic morphisms preserve isomorphisms, compositions, colimits and finite limits.
Then, for any congruence $\cK$ in $\cF$, the class $\phi^{-1}(\cK) =\{f\in \cE\ |\ \phi(f)\in \cK\}$ is a congruence on $\cE$.
In particular, the class $\cK_\phi:= \phi^{-1}(\Iso)$ of maps inverted by $\phi$ is a congruence.
We shall refer to $\cK_\phi$ as {\it the congruence of $\phi$}.

\item\label{exmpcongruence5} 
Let $\phi:\cE\to \cF$ be a left-exact localization.
Then, for any congruence $\cK$ in $\cE$ such that $\cK_\phi\subseteq \cK$, its image $\phi(\cK)$ is a congruence on $\cF$.
Moreover, we have $\cK = \phi^{-1}(\phi(\cK))$.

\item\label{exmpcongruence3}
The class $\Conn \infty$ of \oo connected maps is a congruence (see \cite[Proposition 6.5.2.8]{Lurie:HTT} and also \cite[Example 4.2.5.d]{ABFJ:HS}).
By \cref{prop:luriehypercompletion}, $\Conn \infty$ is the congruence of $\phi:\cE\to \cE\hyper$, the hypercompletion of $\cE$ (\cref{def:oo-connected}).

\item\label{exmpcongruence6}
Any intersection of congruences is a congruence.

\end{examplenum}
\end{examples}

Any class of maps $\Sigma$ in a topos $\cE$ is contained in a smallest congruence $\Sigma\cong \subseteq \cE$.
We say that ${\Sigma}\cong$ is the congruence \emph{generated} by the class of maps $\Sigma\subseteq \cE$.
An algebraic morphism $\phi:\cE\to \cF$ inverts $\Sigma$ if and only if it inverts the whole congruence $\Sigma\cong$ ($\Sigma\subseteq \cK_\phi \Leftrightarrow \Sigma\cong\subseteq \cK_\phi$).

\medskip

If $\Sigma$ is a class of maps in a topos $\cE$, an algebraic morphism $\phi:\cE\to \cF$ is said to be a {\it $\Sigma$-localization in algebraic morphisms}, or to invert $\Sigma$ {\it universally among algebraic morphisms}, or to be the {\it left-exact localization generated by $\Sigma$},
if it is initial in the category of algebraic morphisms inverting $\Sigma$.
More precisely, if we denote by $\fun\cE\cG\alg^\Sigma$, the category of algebraic morphisms inverting $\Sigma$, we will say that $\phi:\cE\to \cF$ is a {\it $\Sigma$-localization in algebraic morphisms}, or that it inverts $\Sigma$ {\it universally} in algebraic morphisms, if it inverts $\Sigma$ and if the induced functor
\[
(-)\circ \phi: \fun\cF\cG\alg
\stto
\fun\cE\cG\alg^\Sigma
\]
is an equivalence of categories for every topos $\cG$.
More prosaically, this means that if an algebraic morphism $\gamma:\cE\to \cG$ inverts every map in $\Sigma$, then there exists a unique pair $(\gamma',\alpha)$ where $\gamma':\cF\to \cG$ is an algebraic morphism and $\alpha$ is an isomorphism $\gamma\simeq \gamma'\circ \phi$.
We shall say that an algebraic morphism $\phi:\cE\to \cF$ is a {\it left-exact localization} if it is the left-exact localization generated by some class $\Sigma\subseteq \cE$ (which we can always take to be the class of all maps inverted by $\phi$).
Recall from \cref{sec:localizations} that any reflection is a localization in cocontinuous functors.
The universality property shows immediately that any left-exact reflection is a left-exact localization.
Since any algebraic morphism has automatically a right adjoint, any left-exact localization is a reflection by \cite[Proposition 2.2.1]{ABFJ:HS}.
This reflection is always left-exact and this proves that, conversely, any left-exact localization is a left-exact reflection.
If $\Sigma$ is a class of maps in a topos $\cE$, then the codomain of any left-exact $\Sigma$-localization $\cE\to \cF$ is unique up to equivalence of categories and we denote this codomain generically by $\cE\Forcing\Sigma\Iso$ (this non-classical notation will be justified in \cref{sec:forcing}).

\medskip
For an arbitrary class $\Sigma$, the topos $\cE\Forcing\Sigma\Iso$ may not exist, but it does when $\Sigma$ is a set.
We say that a congruence $\cK\subseteq \cE$ is of {\it small generation} if $\cK={\Sigma}\cong$ for a set of maps $\Sigma\subseteq \cE$.
The congruence $\cK_\phi$ of an algebraic morphism $\phi$ is of small generation by \cite[Lemma 4.2.7]{ABFJ:HS}.
A left-exact localization is said to be {\it accessible} if it is accessible in the sense of \cref{def:accessible}.
The following theorem is not explicitly stated in \cite{Lurie:HTT}, but it is an easy consequence of Propositions 5.5.4.15, 6.2.1.1, and 6.2.1.2 therein combined.

\begin{theorem}[{\cite{Lurie:HTT}}]
\label{Luriethm1}
\label{Luriethmloc}
If $\Sigma$ is a set of maps in a topos $\cE$, then the subcategory $\Loc\cE{\Sigma\cong}$
is reflective, it is a topos, the reflector 
$\rho:\cE\to \Loc\cE{\Sigma\cong}$ is accessible, left-exact, and universal for inverting $\Sigma$ in algebraic morphisms.
Moreover, $\Sigma\cong$ is the class of maps inverted by $\rho$.
Symbolically,
\[
\cE\Forcing\Sigma\Iso
\ =\ 
\Loc\cE{\Sigma\cong}\,.
\]
\end{theorem}

\begin{theorem}[{\cite[Proposition 5.5.4.16]{Lurie:HTT}}]
\label{Luriethm2} 
If $\phi:\cE\to \cF$ is an algebraic morphism of topoi, then the congruence $\cK_\phi$ is of small generation.
\end{theorem}

The following theorem from \cite{Lurie:HTT} shows that the notion of congruence plays a role similar to that of Grothendieck topologies in controlling left-exact localizations.
Contrary to the case of 1-topoi, it is not known whether all left-exact localizations of topoi are accessible.
Therefore, a condition of small generation must be imposed.
For $\cE$ a fixed topos, we consider
the poset $\Cong(\cE)$ of all congruences in $\cE$ (ordered by inclusion), 
the subposet $\Cong\sg(\cE)\subseteq \Cong(\cE)$ of congruences of small generation in $\cE$,
and the poset $\mathsf{LexLoc_{acc}}(\cE)$ of (isomorphism classes of) accessible left-exact localizations of $\cE$ and algebraic morphisms between them.
The map $\phi\mapsto \cK_\phi$ defines a morphism of posets $\mathsf{LexLoc_{acc}}(\cE) \to \Cong\sg(\cE)$.
Conversely, if $\cK=\Sigma\cong$ is a congruence of small generation, then the localization $\phi_\cK:\cE\to \cE\Forcing\cK\Iso$ exist and is accessible by \cref{Luriethmloc} and we get a function
$\Cong\sg(\cE) \to \mathsf{LexLoc_{acc}}(\cE)$.

\begin{theorem}[{\cite[Propositions 5.5.4.2 and 6.2.1.1 together]{Lurie:HTT}}]
\label{thm:bij-congruence-lexloc}
The functions $\phi\mapsto \cK_\phi$ and $\cK \mapsto \phi_\cK$ define inverse isomorphisms of posets
\[
\mathsf{LexLoc_{acc}}(\cE)
\ \simeq\ 
\mathsf{Cong_{sg}}(\cE)\, .
\]
\end{theorem}

\begin{remark}
In the rest of this paper, we will work with arbitrary congruences, not only those of small generation.
We shall mention explicitly when the hypothesis of small generation is needed.    
\end{remark}

The following lemma will be useful.
Let $\alpha$ be the inaccessible cardinal bounding the size of small objects.
Let $\beta>\alpha$ be the inaccessible cardinal bounding the size of large objects.
A topos $\cE$ is a $\beta$-small category and the poset of all congruences in $\Cong(\cE)$ is always $\beta$-small.
The arities of the operations used in the definition of congruences are $\alpha$-small categories, hence the following result (recall that we denote by $\Class(\cE)$ the poset of all classes of maps in a topos $\cE$).

\begin{lemma}
\label{lem:filtcong}
The inclusion $\Cong(\cE)\to \Class(\cE)$ commutes with $\alpha$-filtered unions.
\end{lemma}

\subsubsection{Modalities and fiberwise orthogonality}
\label{sec:modality}

In this short section, we recall the definitions of modalities and fiberwise orthogonality from \cite{ABFJ:GC,ABFJ:HS}.
They will be needed in some statements and proofs.
We refer to \cite[Section~5.2.8]{Lurie:HTT} or \cite[Section~3.1]{ABFJ:HS} for more details on factorization systems, and to \cite[Section~3.2]{ABFJ:HS} for modalities.

\medskip

Recall that a map $u:A\to B$ in a category
$\cE$ is said to be (left) \emph{orthogonal} 
to a map $f:X\to Y$ and we write $u\perp f$ (and $f$ is said to be
right \emph{orthogonal} to $u$)
if the following commutative square in the category of spaces $\cS$ is cartesian.
\begin{equation}\label{eq:def-orthogonality2}
\begin{tikzcd}
\Map B X \ar[rr,"{\Map u X}"] \ar[d,"{\Map B f}"']  && 
\Map A X \ar[d,"{\Map A f}"] \\
\Map B Y \ar[rr,"{\Map u Y}"]
&& \Map A Y
\end{tikzcd} \,
\end{equation}

If $\cA$ and $\cB$ are two classes of maps in a category $\cE$, we
shall write $\cA\perp \cB$ to mean that we have $u\perp v$ for
every $u\in \cA$ and $v\in \cB$.  We shall denote by $\rorth\cA$
(resp.  $\lorth \cA$ ) the class of maps in $\cE$ that are
right orthogonal (resp. left orthogonal) to every map in $\cA$.  We
have
\[
\cA\subseteq \lorth\cB 
\quad \Leftrightarrow \quad \cA\perp \cB 
\quad \Leftrightarrow  \quad \rorth\cA \supseteq \cB\ .
\]

Recall that a class $\cO$ of objects in a category $\cE$
is said to be {\it replete} if every object isomorphic to
an object in $\cO$ belongs to $\cO$.
We shall say that a class of maps $\cM$ in a category $\cE$
is {\it replete} if it is replete as a class of objects of
the arrow category $\cE\arr $.
Most classes of maps considered in this paper
are replete.

\medskip

A pair $(\cL, \cR)$ of classes of maps in a category $\cE$ is said to be a \emph{factorization system} if the following three conditions hold:
\begin{enumerate}[label=\roman*)]
\item\label{defFactSystem:1} the classes $\cL$ and $\cR$ are replete;
\item\label{defFactSystem:2} $\cL\perp \cR$;
\item\label{defFactSystem:3} every map $f:X\to Y$ in $\cE$ admits a factorization $f=pu:X\to E\to Y$ with $u\in \cL$ and $p\in \cR$.
\end{enumerate}

If $(\cL, \cR)$ is a factorization system, then $\cR=\rorth \cL$ and $\cL=\lorth \cR$.
The class $\cL$  is said to be the \emph{left class} of the factorization system and the class $\cR$ to be the \emph{right class}. 
We shall say that the factorization in \ref{defFactSystem:3} is a {\it $(\cL,\cR)$-factorization} of the map $f:X\to Y$. 
The $(\cL, \cR)$-factorization of a map is unique (up to unique isomorphism).

\begin{definition}[Modality]
\label{defmodality} 
We shall say that a factorization system $(\cL, \cR)$ in 
a topos $\cE$ is a \emph{modality} if its left class $\cL$ is closed under base change. 
\end{definition}

We refer to \cite{ABFJ:GC,ABFJ:GBM,ABFJ:HS,RSS} for more on modalities.

\begin{examples}
\label{examplemodality}
Let $\cE$ be a topos.
\begin{examplenum}
\item\label{examplemodality:1}
    The factorization system $(\Surj,\Mono)$ of surjections and monomorphisms is a modality \cite[Example 3.2.12~(b)]{ABFJ:HS}.

\item\label{examplemodality:2} 
    The factorization system system $(\Conn n,\Trunc n)$ of $n$-connected maps and $n$-truncated maps is a modality for every $n \geq -1$ \cite[Example 3.4.2]{ABFJ:GBM}.
    Notice that $(\Conn {-1},\Trunc {-1})=(\Surj,\Mono)$.

\item\label{examplemodality:3} 
    If $\phi:\cE\to \cF$ is an algebraic morphism, the pair $(\cK_\phi,\cK_\phi^\perp)$ is a modality in $\cE$.
\end{examplenum}
\end{examples}

\medskip

\begin{definition}[Fiberwise orthogonality]
\label{fperp} 
Let $\mathcal{E}$ be a category with finite limits.
We shall say that a map $u:A\to B$ in $\mathcal{E}$ is \emph{fiberwise left orthogonal} to a map $f:X\to Y$, and write $u\fperp f$ (and say that $f$ is \emph{fiberwise right orthogonal} to $u$) if every base change $u'$ of $u$ is left orthogonal to $f$.
\end{definition}

If $\cA$ and $\cB$ are two classes of maps in a category with finite limits $\cE$, we shall write $\cA\fperp \cB$ to mean that we have $u\fperp f$ for every $u\in \cA$ and $f\in \cB$.
We shall denote by $\cA^\fperp$ (resp. ${}^\fperp \cA$) the class of maps in $\cE$ that are fiberwise right orthogonal (resp.  fiberwise left orthogonal) to every map in $\cA$.
We have
\[
\cA\subseteq {}^\fperp\cB 
\quad \Leftrightarrow \quad
\cA\fperp \cB
\quad \Leftrightarrow
\quad \cA^\fperp \supseteq \cB
\]

Let $\cE$ be a category with finite limits.
Then a factorization system $(\cL, \cR)$ in $\cE$ is a modality if and only if $\cL \fperp \cR$, in which case $\cR=\cL^\fperp$ and $\cL={}^\fperp \cR$.

\medskip
For any set of maps $\Sigma$ in a topos $\cE$, the pair $({}^\fperp(\Sigma^\fperp),\Sigma^\fperp)$ is a modality \cite[Theorem 3.2.20]{ABFJ:HS}.
Let us see that the modality $(\Conn n, \Trunc n)$ admits such generator.
Let $S^n$ be the $n$-sphere in $\cS$.
We shall abuse notation and denote again $S^n$ the image of $S^n$ by the unique algebraic morphism $\cS\to \cE$.
We denote by $s^n$ the unique map $S^n\to 1$ in $\cE$.
Recall the class $\Trunc n$ of $n$-truncated maps and the class $\Conn n$ of $n$-connected maps from \cref{def:truncated}.

\begin{lemma}
\label{lem:gen-conn-trunc}
For $-1\leq n <\infty$, the modality $(\Conn n, \Trunc n)$ is generated by the single map $s^{n+1}$, that is we have
\[
\{s^{n+1}\}^\fperp = \Trunc n
\qquad\textrm{and}\qquad
{}^\fperp\big(\{s^{n+1}\}^\fperp\big) = \Conn n \,.
\]
\end{lemma}
\begin{proof}
Since $(\Conn n, \Trunc n)$ is a modality, we have $\Conn n = {}^\fperp\Trunc n$.
So it is enough to show that $\{s^{n+1}\}^\fperp = \Trunc n$.

Any base change of the map $s^n:S^n\to 1$ is of the type $S^n\times A\to A$ for some $A\in \cE$.
Thus, for a map $f:X\to Y$, the orthogonality relation $s^{n+1} \fperp f$ is true if and only if all squares
\[
\begin{tikzcd}
\Map A X \ar[r]\ar[d] \pbmark & \Map A Y \ar[d]\\
\Map A {X^{S^{n+1}}} \ar[r] & \Map A {Y^{S^{n+1}}}
\end{tikzcd}
\]
are cartesian in $\cS$, if and only if the single square
\[
\begin{tikzcd}
X \ar[r,"f"]\ar[d] \pbmark & Y \ar[d]\\
{X^{S^{n+1}}} \ar[r] & {Y^{S^{n+1}}}
\end{tikzcd}
\]
is cartesian in $\cE$.
The cartesian gap map of this last square is $\Delta^{n+2} f$, so we have that
$s^{n+1} \fperp f$
if and only if $\Delta^{n+2} f$ is invertible
if and only if $f$ is $n$-truncated.
\end{proof}

\subsubsection{Acyclic classes}
\label{sec:acyclic}

In this section, we recall the notion of acyclic class from \cite{ABFJ:HS} and develop a number of new results about them.

\begin{definition}[Acyclic class {\cite[Definition 3.2.8]{ABFJ:HS}}]
\label{defacyclicclass2}
We say that a class of maps $\cA$ in a topos $\cE$ is {\it acyclic} if the following conditions hold:
\begin{enumerate}[label=\roman*)]
\item\label{defacyclicclass2:1} the class $\cA$ contains the isomorphisms and is closed under composition;
\item\label{defacyclicclass2:2} the class ${\cA}$ is closed under colimits (in the arrow category of $\cE$);
\item\label{defacyclicclass2:3} the class ${\cA}$ is closed under base change.
\end{enumerate} 
\end{definition}

\begin{examples}
\label{ex:acyclic} 
Acyclic classes abound in topos theory and homotopy theory.
\begin{examplenum}
\item\label{ex:acyclic:1} 
In a topos $\cE$, the classes $\Iso$ of isomorphisms and $\All$ of all maps are respectively the smallest and the largest acyclic classes (for the inclusion relation).

\item\label{ex:acyclic:2}
Any congruence is an acyclic class, in particular the class $\cK_\phi$ of maps inverted by an algebraic morphism $\phi:\cE\to \cF$ is acyclic.

\item\label{ex:acyclic:3bis}
The left class of a modality is always acyclic.
In fact, by \cite[Proposition 3.2.14]{ABFJ:HS}, the class  ${}^\fperp \cM$ is acyclic for any class of maps $\cM$ in a topos $\cE$.

\item\label{ex:acyclic:3}
In particular, the class $\Surj$ and the classes $\Conn n$ (for $-1\leq n\leq \infty$) are acyclic since they are the left classes of some modalities \cite[Examples 3.2.12~(b) and (c)]{ABFJ:HS}.

\item\label{ex:acyclic:3ter} 
The class of acyclic maps in a topos in the sense of \cite{Raptis:acyclic,Hoyois:acyclic} is an acyclic class since it is the left class of a modality.

\item\label{ex:acyclic:4} 
Any algebraic morphism of topoi $\phi:\cE\to \cF$ preserves isomorphisms, composition, colimits and pullbacks, therefore the class $\phi^{-1}(\cA) =\{f\in \cE\ |\ \phi(f)\in \cA\}$ is an acyclic class of $\cE$, for any acyclic class $\cA\subseteq \cF$.
In particular, the class $\phi^{-1}(\Surj)$ of maps sent to surjections by $\phi$ is acyclic. 
More generally, $\phi^{-1}(\Conn n)$ is acyclic for $-1\leq n\leq \infty$.

\item\label{ex:acyclic:5} 
Any intersection of acyclic classes is acyclic.

\end{examplenum}
\end{examples}

Every class of maps $\Sigma$ in a topos $ \cE$ is contained in a smallest acyclic class $\Sigma\ac$ called the acyclic class {\it generated} by $\Sigma$.
We shall say that an acyclic class $\cA$ is of small generation if $\cA=\Sigma\ac$ for a set of maps $\Sigma\subseteq \cA$.
We refer to \cite[Corollary 3.2.19]{ABFJ:HS} for a description of $\Sigma\ac$ in terms of saturated classes.
We shall need only the following lemma.
Recall the notion of fiberwise orthogonality of \cref{fperp}.

\begin{lemma}[{\cite[Lemma~3.2.15]{ABFJ:HS}}]
\label{lem:acyclicsat}
For any class of maps $\Sigma$, we have $\Sigma\ac = {}^\fperp(\Sigma^\fperp)$.
\end{lemma}

\begin{remark}[Acyclic classes and modalities]
\label{rem:acyclic2modality}
\Cref{lem:acyclicsat} implies in particular that any acyclic class of small generation $\cA=\Sigma\ac$ is the left class of the modality $(\Sigma\ac,\Sigma^\fperp)$ (see \cite[Theorem 3.2.20]{ABFJ:HS}).
\end{remark}

Acyclic classes and congruences are intimately related.
We recall some results from \cite{ABFJ:HS}.
For a class of maps $\Sigma\subseteq \cE$, we define
\[
\Delta(\Sigma)
\ :=\ 
\big\{\,\Delta u \ | \ u\in \Sigma\,\big\}\ ,
\]
\[
\Delta^0(\Sigma)
\ :=\ 
\Sigma
\ ,\qquad
\Delta^{n+1}(\Sigma)
\ :=\ 
\Delta(\Delta^n(\Sigma))\,,
\]
\[
\Delta^{\leq n}(\Sigma)
\ :=\ 
\bigcup_{k=0}^n \Delta^k(\Sigma)
\ =\ 
\big\{ \Delta^i u \, | \, u \in \Sigma,\, 0\leq k\leq n \big\}
\ ,
\]
\[
{\rm and} \qquad
\Sigma\diag
\ :=\ 
\bigcup_{k=0}^\infty \Delta^k(\Sigma)
\ =\ 
\big\{ \Delta^k u \, | \, u \in \Sigma,\, k\geq 0 \big\}
\ .
\]

\begin{proposition}[Recognition of congruences {\cite[Theorem~4.1.8(3)]{ABFJ:HS}}]
\label{lem:caraccong}
An acyclic class $\cA$ is a congruence if and only if $\Delta(\cA)\subseteq \cA$ if and only if $\cA\diag = \cA$. 
\end{proposition}

\begin{theorem}[Generation of congruences {\cite[Theorem 4.2.12 and Proposition 4.3.6]{ABFJ:HS}}]
\label{congruencegenerated17}
\label{sigmaac=sigmacong}
If $\Sigma$ is a class of maps in a topos $\cE$, then $\Sigma\cong=(\Sigma\diag)\ac$.
Moreover, if $\Sigma$ is a class of monomorphisms, then $\Sigma\cong=\Sigma\ac$.
\end{theorem}

Let $\Class(\cE)$ be the poset of classes of maps in a topos $\cE$ (i.e. the poset of full and replete subcategories of the arrow category $\cE\arr$).
Let $\Acyclic(\cE)$ be the subposet of acyclic classes and $\Cong(\cE)$ the subposet of congruences.
All the inclusions have left adjoints.
\[
\begin{tikzcd}
\Class(\cE)
\ar[rr,shift left = 1.6, "(-)\ac"]
\ar[from=rr,shift left = 1.6, hook']
\ar[rrrr,bend left = 30, "(-)\cong = \left((-)\diag\right)\ac"]
&
&
\Acyclic(\cE)
\ar[rr,shift left = 1.6, "(-)\cong"]
\ar[from=rr,shift left = 1.6, hook']
&
&
\Cong(\cE)
\end{tikzcd}
\]   
We shall see in \cref{thm:adjCongAcyclic} that the inclusion $\Cong(\cE)\to \Acyclic(\cE)$ has also a right adjoint.

\medskip

If $\Sigma$ is a class of maps in a topos $\cE$, we shall denote by $\Sigma\bc$ the smallest class of maps which contains $\Sigma$ and is closed under base change. 

\begin{definition}[{\cite[Definition 4.3.1]{ABFJ:HS}}]
\label{def:ABFJmain}
If $\Sigma$ is a class of maps in a topos $\cE$, 
we shall say that an object $X\in \cE$
is a $\Sigma$-{\it sheaf} if it
is local with respect to the class $(\Sigma\diag)\bc$.
We write  $\Sh \cE \Sigma:=\Loc {(\Sigma\diag)\bc} \cE$ for the full subcategory of $\Sigma$-sheaves.
\end{definition}

The following result is easy consequence 
of \cref{congruencegenerated17,Luriethmloc}.

\begin{theorem}[{\cite[Theorem 4.3.3]{ABFJ:HS}}]
\label{thm:ABFJmain}
Let $\Sigma$ be a set of maps in a topos $\cE$.
Then $\Sh \cE \Sigma=\Loc\cE{\Sigma\cong}=\cE\Forcing\Sigma\Iso$.
\end{theorem}

\medskip

\begin{lemma}[{\cite[Lemma 3.1.5]{ABFJ:HS}}]
\label{acyclic-cancell} 
Any acyclic class is right cancellable ($vu, u\in \cA \Rightarrow v\in \cA$).
\end{lemma}

\begin{proposition}[Transport]
\label{prop:transport-acyclic}
\label{prop:transport-cong}
Let $\phi:\cE\to \cF$ be a left-exact localization of topoi, with associated congruence $\cK_\phi$.
\begin{enumerate}
\item\label{prop:transport-acyclic:0}
For any acyclic class $\cA$ in $\cE$ such that $\cK_\phi\subseteq \cA$, its image $\phi(\cA)$ is an acyclic class in $\cF$.

\item\label{prop:transport-acyclic:1}
    The inverse image $\phi^{-1}$ induces an isomorphism between the poset $\Acyclic(\cF)$ of acyclic classes in $\cF$ and the poset $\Acyclic(\cE)\upslice {\cK_\phi}$ of acyclic classes in $\cE$ containing $\cK_\phi$.

\item\label{prop:transport-acyclic:2}
    The previous isomorphism restricts to an isomorphism between the poset $\Cong(\cF)$ of congruences in $\cF$ and the poset $\Cong(\cE)\upslice {\cK_\phi}$ of congruences in $\cE$ containing $\cK_\phi$.
\end{enumerate}
\end{proposition}
\begin{proof}
\eqref{prop:transport-acyclic:0}
Since any left-exact localization is a reflection, we can assume that $\cF$ is a full subcategory of $\cE$.
With this convention, let us show that $\phi(\cA) = \cA\cap \cF$.
We show first that $\phi(\cA)\subseteq \cA$.
For any map $f:A\to B$ in $\cE$, we have a commutative square
\begin{equation}
\label{eq:square}    
\begin{tikzcd}
A\ar[r]\ar[d,"f"'] & \phi(A)\ar[d,"\phi(f)"]\\
B\ar[r] & \phi(B)
\end{tikzcd}
\end{equation}
where the horizontal maps are in $\cK_\phi$, thus in $\cA$ by hypothesis.
When $f\in \cA$, the map $\phi(f)$ is in $\cA$ by right cancellation (see \cref{acyclic-cancell}).
This proves $\phi(\cA) \subseteq \cA\cap \cF$.
Conversely, for any map $f\in \cA\cap \cF$, we have $\phi(f) = f$.
This proves that $\cA\cap \cF \subseteq \phi(\cA)$, and therfore $\cA\cap \cF = \phi(\cA)$.

Now, let us show that $\cA\cap \cF$ is acyclic.
Conditions \ref{defacyclicclass2:1} and \ref{defacyclicclass2:3} of \cref{defacyclicclass2} are trivial since $\cF$ is stable by limits in $\cE$.
We are left to prove condition \ref{defacyclicclass2:2}.
Let $f_i$ be a diagram of maps in $\cA\cap \cF$, and let $f$ be its colimit computed in $\cE$.
It is in $\cA$ since $\cA$ is acyclic.
The colimit of the diagram computed in $\cF$ is $\phi(f)$, which is then in $\phi(\cA) =\cA\cap \cF$.
Hence $\cA\cap \cF$ is closed under colimits and $\phi(\cA)$ is acyclic.

\smallskip
\noindent\eqref{prop:transport-acyclic:1}
Since every acyclic class $\cB$ in $\cF$ contains the class $\Iso$, we have always $\cK_\phi =\phi^{-1}(\Iso) \subseteq \phi^{-1}(\cA)$. 
This proves the restriction of $\phi^{-1}$ is well defined.
We use the convention introduced in the proof of \eqref{prop:transport-acyclic:0} that $\cF$ is a full subcategory of $\cE$.
Let us see that $\phi(\phi^{-1}(\cB)) = \cB$ for any acyclic class $\cB\subseteq \cF$.
We have always $\phi(\phi^{-1}(\cB))\subseteq \cB$.
Conversely, for any $g\in \cB\subseteq \cF$, we have $\phi(g)=g$.
This proves $\cB \subseteq \phi^{-1}(\cB)\cap \cF = \phi(\phi^{-1}(\cB))$ and the equality.

Let us see now that $\phi^{-1}(\phi(\cA)) = \cA$ for any acyclic class in $\cE$ such that $\cK_\phi\subseteq \cA$.
We have always $\cA \subseteq \phi^{-1}(\phi(\cA))$.
Conversely, let $f$ be a map such that $\phi(f)\in \phi(\cA)=\cA\cap \cF$.
We consider the cartesian gap map $g:A\to B\times_{\phi(B)}\phi(A)$ of the square \eqref{eq:square}.
By construction, the horizontal maps are in $\cK_\phi$.
Thus, the projection $p_2:B\times_{\phi(B)}\phi(A)\to \phi(A)$ is also in $\cK_\phi$ since congruences are closed under base change.
Then the gap map $g:A\to B\times_{\phi(B)}\phi(A)$ is in $\cK_\phi$ by the 3-for-2 property satisfied by congruences.
The projection $p_1:B\times_{\phi(B)}\phi(A)\to B$ is in $\cA$ since $\phi(f)\in \cA$ and $\cA$ is closed by base change.
By hypothesis $\cK_\phi\subseteq\cA$, so both maps $g$ and $p_1$ are in $\cA$.
Then so is $f= p_1\circ g$ by closure under composition.

\smallskip
\noindent\eqref{prop:transport-acyclic:2}
If $\cK\subseteq \cF$ is a congruence, then $\phi^{-1}(\cK)$ is a congruence on $\cE$.
Conversely, for a congruence $\cK\subseteq \cE$, we know that $\phi(\cK)$ is acyclic by \eqref{prop:transport-acyclic:0}. 
To see that it is a congruence we use the fact that $\phi$ preserves diagonals and \cref{lem:caraccong}.
This proves the isomorphism of \eqref{prop:transport-acyclic:1} restricts to congruences.
\end{proof}

\smallskip

We shall say that a colimit is a coproduct if it is indexed by a set (and not by a general small groupoid).

\begin{definition}[Local class {\cite[Proposition 6.2.3.14]{Lurie:HTT}}]
\label{locclass}
\label{localclasscoproduct}  
Let $\cM$ be a class of maps closed under base change in a topos $\cE$. 
We say that $\cM$ is {\it local} if the following two conditions holds:
\begin{enumerate}[label=\roman*)]
\item\label{locclass:1} $\cM$ is closed under coproducts;
\item\label{locclass:2} If the base change of a map $f:A\to B$ along a surjection $v:B'\to B$ belongs to $\cM$, then the map $f$ belongs to $\cM$.
\end{enumerate}
\end{definition}

\begin{examples}
\label{ex:local-class}
Both the left and the right class of a modality are local by \cite[Proposition 3.2.7]{ABFJ:HS}.
In particular, the classes $\Surj$, $\Mono$, $\Conn n$, $\Trunc n$, $\Conn \infty$ are all local.
\end{examples}

\begin{proposition}
\label{acyclic-is-local} 
Let $\cA$ be an acyclic class in a topos $\cE$. Then $\cA$ is a local class.
\end{proposition}
\begin{proof} 
By \cref{ex:local-class}, the result is true in the case where the acyclic class $\cA$ is the left class of a modality.
It follows that the acyclic class $\Sigma\ac\subseteq \cA$ generated by any set of maps $\Sigma\subseteq \cE$ is local, since it is the left class of a modality $(\Sigma\ac,\Sigma^\fperp)$ by \cite[Theorem 3.2.20]{ABFJ:HS}.
Let us now show condition \cref{locclass}.\ref{locclass:1}.
For a family $f_i:A_i\to B_i$ in $\cA$, we want to show that the coproduct is in $\cA$.
If $\Sigma:=\{f_i :i\in I \}$, then $\Sigma\ac \subseteq \cA$ and the class $\Sigma\ac$ is local by the argument above.
Thus, $f\in \Sigma\ac$, since $f_i \in \Sigma\ac$ for every $i\in I$. 
This proves that $f\in \cA$ and hence that the class $\cA$ is local.
The argument is similar for condition \cref{locclass}.\ref{locclass:2}.
\end{proof}

\medskip

Recall from \cref{sec:surjection} that every map $f:A\to B$ in a topos $\cE$ admits a unique factorization $f={\coim f}\circ {\im f}:A\to \Im(f)\to B$
where ${\coim f}:A\to \Im(f)$ is a surjection and ${\im f}:\Im(f)\to B$ a monomorphism.

\begin{lemma}
\label{lem:acyclic-image}
Let $\cA$ be an acyclic class in a topos $\cE$. 
Then a map $f:A\to B$ in $\cE$ belongs to $\cA$ if and only if both maps $\im f$ and $\coim f$ belong to $\cA$.
\end{lemma}
\begin{proof}
The acyclic class is closed under composition by definition.
Thus, if both $\im f$ and $\coim f$ are in $\cA$, so is $f$.
Conversely, the following commutative square is a pullback, since the map  $\im f$ is a monomorphism.
\[
\begin{tikzcd}
A \ar[rr,equal]\ar[d,"\surj f"'] \pbmarkk && A\ar[d,"f"]\\
\Im(f)\ar[rr,"\im f"'] && B
\end{tikzcd}
\]
Thus, $f\in \cA$ implies $\coim f\in \cA$, since an acyclic class is closed under base change.
Moreover, $\im f\in \cA$, since $\im f \coim f =f\in \cA$
and the class $\cA$ has the right cancellation
property by \cref{acyclic-cancell}.
\end{proof}

\medskip

For a class of maps $\Sigma$, we define 
\[
\im\Sigma
\ :=\ 
\big\{\im f \,|\, f\in \Sigma\big\}
\qquad\text{and}\qquad
\coim\Sigma
\ :=\ 
\big\{\coim f \,|\, f\in \Sigma\big\}\,.
\]

\begin{lemma}
\label{lem:imagepretopo}
For any acyclic class $\cA$, we have
\[
\im\cA = \cA\cap \Mono
\qquad\text{and}\qquad
\coim\cA = \cA\cap \Surj\,.
\]
\end{lemma}
\begin{proof}
Since $\im m = m$ when $m$ is a monomorphism, we have always $\cA\cap \Mono \subseteq \im \cA$. 
Conversely, if $f$ is a map in $\cA$, then so is its image $\im f$ by \cref{lem:acyclic-image}.
Hence $\im \cA\subseteq \cA\cap \Mono$.
The proof is similar for the $\coim\cA$.
\end{proof}

The class $\coim \cA= \cA\cap \Surj$ is always acyclic since it is an intersection of two acyclic classes.
The class $\im \cA = \cA \cap \Mono$ however is not acyclic in general.
We will see in \cref{fromacyclictoGroth} that is what we call an extended Grothendieck topology.

\begin{lemma}
\label{lem:nogo}
Let $\cA$ be an acyclic class, then
\[
\cA\subseteq \Mono
\quad\iff\quad 
\coim\cA = \Iso
\quad\iff\quad 
\im\cA=\cA
\quad\iff\quad 
\cA=\Iso\,.
\]
\end{lemma}
\begin{proof}
The equivalence between the first three conditions follow from \cref{lem:acyclic-image}.
Let us prove that $\cA\subseteq \Mono \Leftrightarrow \cA=\Iso$.
Recall that the codiagonal of a map $f:A\to B$, is the map $\nabla f:B\cup_AB\to B$.
This map can be seen as a pushout $1_B\ot f\to 1_B$ in the arrow category of the ambient topos $\cE$.
An acyclic class $\cA$ is by definition stable by such pushouts, hence $f\in \cA$ implies $\nabla f \in \cA$.
The map $\nabla f$ is always surjective, so if $\cA\subseteq \Mono$, $\nabla f$ must be an isomorphism.
But that shows that $f$ itself is an isomorphism (this can be deduced using \cite[Proposition 2.2.6]{ABFJ:GBM} on $1_B\ot f\to 1_B$).
This proves $\cA\subseteq \Mono$ implies $\cA=\Iso$.
The converse is trivial.
\end{proof}

\begin{definition}[Image, coimage, monogenic, epigenic]
\label{def:monogenic}
\label{def:epigenic}
For an acyclic class $\cA$ in a topos $\cE$, we define
its {\it monogenic part}, to be the acyclic class $\im\cA\ac = (\cA\cap \Mono)\ac$ 
and its {\it epigenic part}, to be the acyclic class $\coim \cA = \cA\cap \Surj$.
An acyclic class is called {\it monogenic} if there exists a class $\Sigma$ of monomorphism such that $\cA = \Sigma\ac$.
An acyclic class is called {\it epigenic} if there exists a class $\Sigma$ of surjections such that $\cA = \Sigma\ac$.
\end{definition}

\begin{proposition}[Monogenic--epigenic decomposition]
\label{prop:imcoimdecomp}
The following relation holds in the poset $\Acyclic(\cE)$
\[
\cA
\ =\ 
\im \cA\ac
\,\vee\,
\coim \cA
\]
where $\vee$ is the supremum in $\Acyclic(\cE)$.
\end{proposition}
\begin{proof}
By \cref{lem:imagepretopo}, we have $\im \cA\subseteq \cA$ and $\coim \cA\subseteq \cA$.
Since $\cA$ is acyclic, we have $\im \cA\ac\subseteq \cA$.
Hence $\im \cA\ac \cup \coim \cA\subseteq \cA$ and 
$\im \cA\ac\ \vee\ \coim \cA = \big(\im \cA\ac \cup \coim \cA\big)\ac\subseteq \cA$.
Conversely, by \cref{lem:acyclic-image}, any acyclic class containing $\im\cA$ and $\coim\cA$ must contain $\cA$.
\end{proof}

By \Cref{sigmaac=sigmacong}, we have $\im\cA\ac=\im\cA\cong$ and therefore $\im\cA\ac$ is always a congruence.

\begin{lemma}
\label{lem:monogenic}
An acyclic class $\cA$ is monogenic if and only if $\cA = \im \cA\ac$.
\end{lemma}
\begin{proof}
By definition of a monogenic acyclic class, there exists $\Sigma \subseteq \cA\cap \Mono$ such that $\Sigma\ac = \cA$.
This implies that $\im\cA\ac = (\cA\cap \Mono)\ac = \cA$.
The converse is obvious.    
\end{proof}

\begin{lemma}[Characterization of epigeneration]
\label{lem:epigenic}
The following conditions are equivalent:
\begin{enumerate}
\item\label{lem:epigenic:1}
    $\cA$ is epigenic;
\item\label{lem:epigenic:2}
    $\cA = \coim \cA$;
\item\label{lem:epigenic:3}
    $\cA\subseteq \Surj$;
\item\label{lem:epigenic:4}
    $\im \cA= \Iso$;
\item\label{lem:epigenic:4bis}
    $\im \cA\ac= \Iso$.
\end{enumerate}
Moreover, when $\cA=\cK$ is a congruence, the previous conditions are equivalent to 
\begin{enumerate}
\setcounter{enumi}{5}
\item\label{lem:epigenic:5}
    $\cK\subseteq \Conn\infty$.
\end{enumerate}
\end{lemma}
\begin{proof}
\eqref{lem:epigenic:1}$\Rightarrow$\eqref{lem:epigenic:2} 
By definition, $\cA$ is epigenic if there exists $\Sigma\subseteq \cA\cap \Surj$ such that $\Sigma\ac = \cA$.
Since $\coim\cA = \cA\cap \Surj$ is acyclic, we have $\Sigma\ac\subseteq \cA\cap \Surj \subseteq \cA$ and \eqref{lem:epigenic:2} follows from $\Sigma\ac=\cA$.
\eqref{lem:epigenic:2}$\Rightarrow$\eqref{lem:epigenic:1} by definition.
The equivalences \eqref{lem:epigenic:2}$\Leftrightarrow$\eqref{lem:epigenic:3} and
\eqref{lem:epigenic:3}$\Leftrightarrow$\eqref{lem:epigenic:4} are trivial.
\eqref{lem:epigenic:4}$\Rightarrow$\eqref{lem:epigenic:4bis} because $\Iso$ is an acyclic class.
Conversely, since $\cA$ is acyclic, it contains all isomorphisms, hence $\Iso\subseteq \im\cA\subseteq \im\cA\ac$.
Then it is clear that \eqref{lem:epigenic:4bis}$\Rightarrow$\eqref{lem:epigenic:4}.
Finally, let us see that \eqref{lem:epigenic:3}$\Leftrightarrow$\eqref{lem:epigenic:5}
Recall from \cref{def:oo-connected} that a map is \oo connected if and only if all its diagonals are surjections.
From there, for a class $\Sigma$, we have $\Sigma\subseteq \Conn\infty \Leftrightarrow \Sigma\diag \subseteq \Surj$.
Then, the result follows from the relation $\cK\diag = \cK$ of \cref{lem:caraccong}.
\end{proof}

The following is a direct application of \eqref{lem:epigenic:3}$\Leftrightarrow$\eqref{lem:epigenic:4} of \cref{lem:epigenic}.

\begin{lemma}
\label{lem:imcoim}
For any acyclic class $\cA$, $\im{\coim\cA} = \Iso$.
\end{lemma}

\begin{remark}
The dual relation does not hold in general: from \cref{lem:nogo} we have $\coim{\im\cA\ac} \not= \Iso$ if $\im\cA\ac\not=\Iso$.    
\end{remark}

Let $\MAcyclic(\cE)$ and $\EAcyclic(\cE)$ be the subposets of monogenic and epigenic acyclic classes in $\cE$.

\begin{proposition}$\quad$
\label{prop:coreflection-epi-mono}
\begin{enumerate}
\item\label{prop:coreflection-epi-mono:1}
    The map $\cA\mapsto \im\cA\ac$ is right adjoint to the inclusion $\MAcyclic(\cE)\to \Acyclic(\cE)$.
\item\label{prop:coreflection-epi-mono:2}
    The map $\cA\mapsto \coim\cA$ is right adjoint to the inclusion $\EAcyclic(\cE)\to \Acyclic(\cE)$.
\end{enumerate}
\end{proposition}
\begin{proof}
\eqref{prop:coreflection-epi-mono:1}
Let $\cA$ be an acyclic class and $\cB$ a monogenic acyclic class.
We need to prove that $\cB\subseteq \cA \Leftrightarrow \cB\subseteq \im\cA\ac$.
If $\cB\subseteq \cA$, then $\im\cB\subseteq\im\cA$ and $\cB=\im\cB\ac\subseteq\im\cA\ac$.
Conversely, if $\cB\subseteq\im\cA\ac$, then $\cB\subseteq\cA$ since $\im\cA\ac\subseteq\cA$.

\smallskip
\noindent\eqref{prop:coreflection-epi-mono:2}
Let $\cA$ be an acyclic class and $\cB$ an epigenic acyclic class.
We need to prove that $\cB\subseteq \cA \Leftrightarrow \cB\subseteq \coim\cA$.
If $\cB\subseteq \cA$, then $\cB\cap \Surj\subseteq\cA\cap \Surj$ and $\cB=\coim\cB\subseteq\coim\cA$.
Conversely, if $\cB\subseteq\coim\cA$, then $\cB\subseteq\cA$ since $\coim\cA\subseteq\cA$.
\end{proof}

\subsubsection{The inclusion of congruences in acyclic classes}
\label{sec:acvcong}

The goal of this section is to prove \cref{thm:adjCongAcyclic}, which states that the inclusion of congruences in acyclic classes has both a left and a right adjoint.
We shall need the right adjoint in the proof of \cref{thm:equivPtopHcong}.

\medskip
For a class of maps $\cA$, we define the {\it decalage} of $\cA$:
\[
\dec \cA
\ :=\ 
\{\, f\in \cE \ |\  f\in \cA,\, \Delta f \in \cA \,\}
\]
and, by induction
\[
\decn 0 \cA
\ :=\ 
\cA 
\qquad\text{and}\qquad
\decn {n+1} \cA
\ :=\ 
\dec {\decn n \cA}
\ =\ 
\{\, f\in \cE \ |\  \forall 0\leq k\leq n+1,\, \Delta^k f \in \cA\,\}\,.
\]
We put also
\[
\decinfty \cA
\ :=\ 
\bigcap_n \decn n \cA
\ =\ 
\{\, f\in \cE\ |\  \forall n\geq 0\,, \Delta^n f \in \cA\,\}\,.
\]

The following justifies the importance of the construction $\dec \cA$.
\begin{lemma}
\label{lem:decalage}
\label{lem:racine-acyclic}
Let $\cA$ be an acyclic class.
\begin{enumerate}
\item\label{lem:decalage:0} The class $\cA$ is a congruence if and only if $\dec \cA = \cA$.
\item\label{lem:decalage:1} The class $\dec \cA$ is acyclic.
\item\label{lem:decalage:2} The class $\decinfty \cA$ is a congruence.
\end{enumerate}
\end{lemma}
\begin{proof}
\noindent\eqref{lem:decalage:0}
This is a reformulation of the recognition criteria of \cref{lem:caraccong}.

\smallskip
\noindent\eqref{lem:decalage:1} 
When $\cA$ is of small generation, it is the left class of a modality by \cite[Theorem 3.2.20]{ABFJ:HS}. 
Then \eqref{lem:decalage:1} is \cite[Remark 3.3.10]{ABFJ:HS}.
For the general case, we can filter $\cA$ by acyclic classes $\Sigma\ac$ for $\Sigma$ a set of maps in $\cA$.
Since the definition of $\dec-$ involves only finitary constructions, we have $\dec\cA = \bigcup_\Sigma\dec{\Sigma\cong}$ and the result follows from the analogue of \cref{lem:filtcong} for acyclic classes.

\smallskip
\noindent\eqref{lem:decalage:2}
The class $\decinfty \cA$ is acyclic as an intersection of acyclic classes by \eqref{lem:decalage:1}.
Then, by construction, it is clear that $\dec {\decinfty \cA} = \decinfty \cA$,
and the result follows from \eqref{lem:decalage:0}.
\end{proof}

\begin{examples} 
\label{ex:lexcore}
We have the following examples of the constructions $\dec-$:
\begin{examplenum}
\item\label{ex:lexcore:1} $\dec \Surj = \Conn 0$;
\item\label{ex:lexcore:2} $\dec {\Conn n} = \Conn {n+1}$ and $\decn {n+1} \Surj = \Conn n$;
\item\label{ex:lexcore:3} $\decinfty \Surj = \decinfty {\Conn n} = \Conn\infty$.
\item\label{ex:lexcore:4} Let $\phi:\cE\to \cF$ be a left-exact localization of topoi. 
If $\cA$ an acyclic class containing the congruence $\cK_\phi$, we have $\cK_\phi =\dec {\cK_\phi} \subseteq \dec \cA$. 
This means that $\dec -$ induces an endomorphism of the poset $\Acyclic(\cE)\upslice {\cK_\phi}$.
Moreover, since $\phi$ preserves diagonals, we have $\phi^{-1}(\dec \cA) = \dec {\phi^{-1}(\cA)}$, for any class of maps $\cA$ in $\cF$.
This shows that the transport bijections of \cref{prop:transport-acyclic} commute with $\dec-$.
In particular, we have 
$\phi^{-1}(\Conn n) = \decn {n+1} {\phi^{-1}(\Surj)} $
and
$\phi^{-1}(\Conn \infty) = \decinfty {\phi^{-1}(\Surj)}$.

\end{examplenum}
\end{examples}

\begin{remark}
\label{rem:decalage-v-suspension}
Let $(\cA,\cB)$ be a modality generated by a set of maps $\Sigma$.
Recall from \cite{CORS} that the separation of a modality $(\cA,\cB)$ is the modality $(\cA',\cB')$ generated by the set $\nabla\Sigma = \{\nabla f\,|\,f\in \Sigma\}$ of codiagonals of maps in $\Sigma$.
The decalage $\dec \cA$ is closely related to $\cA'$, but different.
We shall prove in a subsequent work \cite{ABFJ:box} that 
$\cA' = \dec \cA \cap \Surj$ and $\dec \cA = \cA'\vee \im\cA\ac$.
The two notions coincide if and only if $\cA\subseteq \Surj$.
\end{remark}

\begin{lemma}
\label{lem:congracine}
An acyclic class $\cA$ is a congruence if and only if $\decinfty\cA = \cA$.
\end{lemma}
\begin{proof}
Since $\decinfty\cA\subseteq \dec\cA\subseteq \cA$, we have 
$\decinfty\cA = \cA$ if and only if $\dec\cA = \cA$.
Then the result follows from \cref{lem:decalage}\eqref{lem:decalage:0}.
\end{proof}

\begin{proposition}
\label{prop:largestcong=racine}
For any acyclic class $\cA$, the class $\decinfty\cA$ is the largest congruence contained in $\cA$.
\end{proposition}
\begin{proof}
We prove first that $\decinfty\cA$ is a congruence.
The class $\decinfty\cA$ is acyclic by \cref{lem:racine-acyclic}.
To see that it is a congruence, we use \cite[Proposition 4.2.3]{ABFJ:HS}:
an acyclic class $\cL$ is a congruence if and only if $\Delta(\cL)\subseteq \cL$.
By definition, we have
\[
\decinfty\cA
\ =\ 
\big\{f\in \cE \ |\ \forall n\geq 0,\, \Delta^n f \in \cA\big\}\ .
\]
Hence it is clear that $\Delta\left(\decinfty\cA\right)\subseteq \decinfty\cA$.
Let us see now the maximality property.
Let $\cK$ be a congruence included in $\cA$.
From the inclusion $\cK\subseteq \cA$, we get an inclusion 
$\decinfty\cK \subseteq \decinfty\cA$.
Then the result follows from \cref{lem:congracine}.
\end{proof}

The following theorem compares the notion of acyclic classes and congruences.

\begin{theorem}
\label{thm:adjCongAcyclic}
Let $\cE$ be a topos.
The inclusion of congruences in acyclic classes admits both a left and a right adjoint.
\[
\begin{tikzcd}
\Cong(\cE) \ar[rr, hook]
\ar[from=rr, shift left = 3,"{\decinfty-}"]
\ar[from=rr, shift right = 3,"{(-)\cong}"']
&& \Acyclic(\cE)
\end{tikzcd}
\]
The left adjoint is given by the congruence completion $\cA\mapsto \cA\cong$.
The right adjoint is given by the map $\cA\mapsto \decinfty\cA$.
\end{theorem}
\begin{proof}
The left adjoint part is essentially by definition of the congruence completion.
The right adjoint part is \cref{prop:largestcong=racine}.
\end{proof}

\begin{remark}
\Cref{thm:adjCongAcyclic} says that every acyclic class $\cA$ sits between two congruences
\[
\decinfty \cA
\ \subseteq\ 
\cA
\ \subseteq\ 
\cA\cong
\,.
\]
This implies a reverse order on the category of local objects
\[
\Loc\cE{\cA\cong}
\ \subseteq\ 
\Loc\cE\cA
\ \subseteq\ 
\Loc\cE{\decinfty\cA}
\,.
\]
If the congruences $\decinfty\cA$ and $\cA\cong$ are of small generation, the categories $\Loc\cE{\cA\cong}$ and $\Loc\cE{\decinfty\cA}$ are categories of sheaves for the corresponding left-exact localizations \cref{Luriethm1}.
Hence, every sheaf for the congruence $\cA\cong$ is $\cA$-local, and 
every $\cA$-local object is a sheaf for the congruence $\decinfty\cA$.
\end{remark}

We shall need the previous theorem and the following proposition in the proof of \cref{thm:equivPtopHcong}.

\begin{proposition}
\label{prop:racinetopo}
We have $\decinfty\cA\cap \Mono = \cA\cap \Mono$.
\end{proposition}
\begin{proof}
A map $f$ is in $\decinfty\cA$ if and only if all its diagonals $\Delta^nf$ are in $\cA$.
When $f=m$ is a monomorphism the collection of diagonals reduces to $m$ itself and some isomorphisms.
Since an acyclic class contains all isomorphisms, the condition $m\in \decinfty\cA$ reduces to $m\in \cA$. 
Hence the result.
\end{proof}

\subsection{Forcing}
\label{sec:forcing}

A localization of categories is forcing universally a class of maps $\Sigma$ to be invertible.
In this section, we introduce a variation of the notion of localization, where the condition to be invertible is replaced by another one, typically, to be surjective or to be \oo connected.
We shall only be interested with forcing conditions that happen to be equivalent to localizations, but it is sometimes more convenient, or more meaningful, to present a localization by saying that it forces some maps to be surjective, than to present it as actually inverting some maps.
For example, from a logical point of view, forcing a map to be surjective is quite natural since this corresponds to imposing an existential axiom.
The theory of forcing presented here is a rudiment of a potential ``higher geometric logic'' for topoi.

\medskip
Formally, we are going to replace the class $\Iso$ of isomorphisms by a chosen class $\Theta$ of maps and consider the problem of forcing the inclusion of $\Sigma$ in $\Theta$.
Our main examples will be forcing the maps in $\Sigma$
\begin{enumerate}[label=\alph*)]
\item to be invertible ($\Theta = \Iso$, the class of isomorphisms),
\item to be surjective ($\Theta = \Surj$, the class of surjective maps),
\item to be $n$-connected ($\Theta = \Conn n$, the class of $n$-connected maps, for $-2\leq n\leq \infty$),
\item to be $n$-truncated ($\Theta = \Trunc n$, the class of $n$-truncated maps, for $-2\leq n<\infty$).
\end{enumerate}
For any topos $\cF$, these examples of $\Theta$ define full subcategories of maps $\Theta(\cF)\subseteq \cF\arr$ which are natural in $\cF$, in the sense that, for any algebraic morphism of topoi $\phi:\cF\to \cF'$, we have 
\[
\phi\left(\Theta(\cF)\right)
\ \subseteq\ 
\Theta(\cF')\ .
\]
We shall call such a class $\Theta$ a {\it uniform class of maps}.

\medskip
Given a topos $\cE$, a class of maps $\Sigma$ in $\cE$ and a uniform class of maps $\Theta$, we shall say that an algebraic morphism $\phi:\cE\to \cF$ 
{\it forces the inclusion of $\Sigma$ in $\Theta$} if $\phi(\Sigma)\subseteq \Theta(\cF)$.
The functor $\phi$ is said to {\it force the inclusion of $\Sigma$ in $\Theta$  universally} if it is initial in the category of functors forcing the inclusion of $\Sigma$ in $\Theta$.
More precisely, for any topos $\cG$, let us denote by
\[
\fun\cE\cG\alg^{\Sigma\forcing\Theta}
\ \subseteq\ 
\fun\cE\cG\alg
\]
the full subcategory spanned by algebraic morphisms $\phi:\cE\to \cF$ forcing the inclusion of $\Sigma$ in $\Theta$.
This defines a functor
\begin{align*}
\Toposalg &\tto \CAT    \\
\cG &\mto \fun\cE\cG\alg^{\Sigma\forcing\Theta}\ .
\end{align*}
which is a subfunctor of the representable functor $\fun \cE-\alg$.
If $\phi:\cE\to \cF$ forces the inclusion of $\Sigma$ in $\Theta$, then the composition 
$(-)\circ \phi: \fun \cF\cG\alg \to \fun\cE\cG\alg$ 
induces a
natural transformation of functors in $\cG$
\[
(-)\circ \phi: \fun\cF\cG\alg \stto \fun\cE\cG\alg^{\Sigma\forcing\Theta}\,.
\]
The functor $\phi$
forces the inclusion of $\Sigma$ in $\Theta$
universally if the induced functor is an equivalence of categories for every topos $\cG$, 
that is if the functor $\fun\cE-\alg^{\Sigma\forcing\Theta}$ is representable.
If such a map $\phi:\cE\to \cF$ exist, it is unique and we denote the codomain $\cF$ generically by $\cE\Forcing\Sigma\Theta$.

\medskip

We shall call the data of $(\cE,\Sigma)$ and $\Theta$ a {\it forcing condition} and denote it 
\[
\Forcing\Sigma\Theta
\]
leaving $\cE$ implicit to lighten the notation.
A forcing condition is said to be {\it representable}, or {\it efficient}, if it can be forced universally.
Two forcing conditions $\Forcing\Sigma\Theta$ and $\Forcing{\Sigma'}{\Theta'}$ are said to be {\it equivalent} 
if the underlying topos $\cE$ is the same and if the corresponding subfunctors of the functor $\fun \cE-\alg$ are identical.
This implies that $\Forcing\Sigma\Theta$ and $\Forcing{\Sigma'}{\Theta'}$ are representable by the same topos (if one of them is representable).
We shall denote equivalent forcing conditions by an equality symbol
\[
\Forcing\Sigma\Theta
\ =\ 
\Forcing{\Sigma'}{\Theta'}\ .
\]

\begin{remark}
\label{rem:localization}
When $\Theta=\Iso$ is the class of isomorphism, this definition gives back the notion of the left-exact cocontinuous localization of $\cE$ generated by $\Sigma$ of \cref{sec:left-exact-localization}. 
In the notation of \cite{ABFJ:HS}, we have
\[
\cE\Forcing\Sigma\Iso
\ =\ 
\LOC \cE \Sigma\cclex\ .
\]
All the examples of forcing we are going to be concerned with will be equivalent to actual localizations.    
\end{remark}

\begin{lemma}
\label{lem:forcing:trivial}
Let $\Theta$ and $\Theta'$ be two uniform classes of maps, and $\Sigma$ and $\Sigma'$ be two classes of maps in a topos $\cE$.
We have the following equivalences of forcing conditions (where the concatenation corresponds to the intersection of the corresponding subfunctors).
\begin{enumerate}
\item\label{lem:forcing:trivial:1} $\Forcing\Sigma{\Theta\cap \Theta'} = \Forcing\Sigma\Theta \Forcing\Sigma{\Theta'}  = \Forcing\Sigma{\Theta'} \Forcing\Sigma\Theta$
\item\label{lem:forcing:trivial:2} $\Forcing{\Sigma\cup\Sigma'}\Theta = \Forcing\Sigma\Theta \Forcing{\Sigma'}\Theta  = \Forcing{\Sigma'}\Theta \Forcing\Sigma\Theta$.
\item\label{lem:forcing:trivial:3} Moreover, if $\Forcing{\Sigma}\Theta$ and $\Forcing{\Sigma'}\Theta$ are representable 
then $\cE\Forcing{\Sigma\cup\Sigma'}\Theta$ is representable, and we have
\[
\cE\Forcing{\Sigma\cup\Sigma'}\Theta
\ =\ 
\big(\cE\Forcing{\Sigma}\Theta\big)\Forcing{\phi(\Sigma')}\Theta
\]
where $\phi:\cE\to \cE\Forcing{\Sigma}\Theta$ is the canonical morphism.
\end{enumerate}

\end{lemma}
\begin{proof}
Direct computation.    
\end{proof}

\medskip

\begin{lemma}
\label{lem:forcing}
We have the following equivalences of forcing conditions:
\begin{enumerate}
\item\label{lem:forcing:1} If the class $\Theta(\cF)$ is acyclic for every topos $\cF$, then $\Forcing\Sigma\Theta = \Forcing{\Sigma\ac}\Theta$.
\item\label{lem:forcing:2} If the class $\Theta(\cF)$ is a congruence for every topos $\cF$, then $\Forcing\Sigma\Theta = \Forcing{\Sigma\ac}\Theta = \Forcing{\Sigma\cong}\Theta$.
\end{enumerate}
\end{lemma}
\begin{proof}
Let $\phi:\cE\to \cF$ be an algebraic morphism of topoi.
Then, if $\Theta(\cF)$ is acyclic, the condition $\phi(\Sigma)\subseteq\Theta(\cF)$ is equivalent to  
$\Sigma\subseteq \phi^{-1}(\Theta(\cF))$ 
and to 
$\Sigma\ac\subseteq \phi^{-1}(\Theta(\cF))$ 
since the class $\phi^{-1}(\Theta(\cF))$ is acyclic by \cref{prop:transport-acyclic}.
This proves $\phi(\Sigma\ac)\subseteq\Theta(\cF)$.
The proof is similar for the congruences.
\end{proof}

To provide the translation between forcing and localizations, we need some notation.
Recall that every map $u:A\to B$ in a topos admits a unique factorization 
$u=\im u\circ \coim u$ as a surjective map $\coim u$ followed by a monomorphism $\im u$. 
For a class of maps $\Sigma\subseteq \cE$, recall from \cref{sec:acyclic} the notation
\[
\im\Sigma
\ :=\ 
\big\{\im u \ | \ u\in \Sigma\big\}\,,
\qquad\qquad
\Delta(\Sigma)
\ :=\ 
\big\{\Delta u \ | \ u\in \Sigma\big\}\ ,
\]
\[
\Delta^{\leq n}(\Sigma)
\ =\ 
\big\{ \Delta^i u \, | \, u \in \Sigma,\, 0\leq k\leq n \big\}
\ ,
\qquad\text{and}\qquad
\Sigma\diag
\ =\ 
\big\{ \Delta^k u \, | \, u \in \Sigma,\, k\geq 0 \big\}
\ .
\]

\begin{theorem}[Forcing equivalences]
\label{thm:forcing}
For any topos $\cE$ and any class of maps $\Sigma$ in $\cE$, we have the following equivalences of forcing conditions.
\begin{enumerate}
\item\label{thm:forcing:1}
    $\Forcing\Sigma\Iso = \Forcing{\Sigma\ac}\Iso = \Forcing{\Sigma\cong}\Iso\,$
\item\label{thm:forcing:2}
    $\Forcing\Sigma\Surj = \Forcing{\Sigma\ac}\Surj = \Forcing{\im\Sigma}\Iso\,$
\item\label{thm:forcing:3}
    $\Forcing\Sigma{\Conn n} = \Forcing{\Sigma\ac}{\Conn n} = \Forcing{\Delta^{\leq n+1}(\Sigma)}\Surj = \Forcing{\im{\Delta^{\leq n+1}(\Sigma)}}\Iso\,$
\item\label{thm:forcing:4}
    $\Forcing\Sigma{\Conn\infty} = \Forcing{\Sigma\ac}{\Conn\infty} = \Forcing{\Sigma\cong}{\Conn\infty} = \Forcing{\Sigma\diag}\Surj = \Forcing{\im{\Sigma\diag}}\Iso\,$
\item\label{thm:forcing:5bis}
    $\Forcing\Sigma{\Mono} = \Forcing{\Delta(\Sigma)}\Iso\,$
\item\label{thm:forcing:5}
    $\Forcing\Sigma{\Trunc n} = \Forcing{\Delta^{n+2}(\Sigma)}\Iso\,$
\end{enumerate}
\end{theorem}
\begin{proof}
To prove an equivalence $\Forcing\Sigma\Theta = \Forcing{\Sigma'}{\Theta'}$, we need to show that, for an algebraic morphism of topoi $\phi:\cE\to \cF$, $\phi$ forces $\Sigma$ to be in $\Theta$ if and only if $\phi$ forces $\Sigma'$ to be in $\Theta'$.

\smallskip
\noindent\eqref{thm:forcing:1}
The class $\Iso(\cF)$ is a congruence for all $\cF$ and the result follows from \cref{lem:forcing}~\eqref{lem:forcing:2}.

\smallskip
\noindent\eqref{thm:forcing:2}
The class $\Surj(\cF)$ is acyclic for all $\cF$ and the first equality follows from \cref{lem:forcing}~\eqref{lem:forcing:1}.
Any algebraic morphism of topoi $\phi:\cE\to \cF$ preserves the image factorization of maps.
Hence $\phi$ forces the maps of $\Sigma$ to be surjections if and only if $\phi$ inverts all their images. 
This proves $\Forcing\Sigma\Surj = \Forcing{\im\Sigma}\Iso$.

\smallskip
\noindent\eqref{thm:forcing:3}
The class $\Conn n(\cF)$ is acyclic for all $\cF$ and the first equality follows from \cref{lem:forcing}~\eqref{lem:forcing:1}.
A morphism $f$ is $n$-connected if and only if its diagonals $\Delta^kf$ ($0\leq k\leq n-1$) are all surjective, and diagonals and surjections are preserved by algebraic morphisms of topoi.
Then $\Forcing\Sigma{\Conn n} = \Forcing{\Delta^{\leq n+1}(\Sigma)}\Surj$
is deduced from \eqref{thm:forcing:2} using \cref{lem:forcing:trivial}.
The last equivalence is an application of \eqref{thm:forcing:2}.

\smallskip
\noindent\eqref{thm:forcing:4}
The class $\Conn\infty(\cF)$ is a congruence for all $\cF$ and the first two equalities follow from \cref{lem:forcing}~\eqref{lem:forcing:2}.
The other ones are just the limit case of \eqref{thm:forcing:3} when $n=\infty$.

\smallskip
\noindent\eqref{thm:forcing:5bis}
A map $f$ is a monomorphism if and only if its diagonal $\Delta f$ is invertible.

\smallskip
\noindent\eqref{thm:forcing:5}
A map $f$ is $n$-truncated if and only if its diagonal $\Delta^{n+2}f$ is invertible.
\end{proof}

\begin{theorem}
\label{cor:forcing}
All forcing conditions of \cref{thm:forcing} are representable if $\Sigma$ is a set of maps.
\end{theorem}
\begin{proof}
Recall that a localization $\Forcing\Sigma\Iso$ is representable when $\Sigma$ is a set by \cref{Luriethm1}.
Then, the result follows from the fact that the classes 
$\im\Sigma$, $\im{\Delta^{\leq n+1}(\Sigma)}$, $\im{\Sigma\diag}$, and $\Delta^{n+2}(\Sigma)$ are all sets if $\Sigma$ is a set.
\end{proof}

We shall see in \cref{cor:topcongexists} that the forcing conditions \eqref{thm:forcing:2}, \eqref{thm:forcing:3}, and \eqref{thm:forcing:4} of \cref{thm:forcing} are representable for any class $\Sigma$.

\section{Topologies}
\label{sec:Grothendieck}

In this section, we introduce the trilogy of extended Grothendieck topology, covering topologies, and Lawvere--Tierney topologies on an arbitrary topos.

The classical theory of Grothendieck topologies limit their definition to presheaf topoi $\PSh C$ as a structure on the category $ C$.
Our definition does not use the choice of a generating category: we define extended Grothendieck topologies on a topos $\cE$ as a class of monomorphisms satisfying some stability conditions (\cref{definitionGrothtop}).
When $\cE=\PSh C$, we prove in \cref{Gtop=Gtop} that this notion is equivalent to the notion introduced in \cite[Definition 6.2.2.1]{Lurie:HTT}.

In the second section, we prove that every extended Grothendieck topology is of small generation (\cref{Gtoparesmall})
and that the poset of extended Grothendieck topologies is small (\cref{cor:GTop-is-small}).
The proof rely on the existence of a subobject classifier in any Grothendieck topos (\cref{mono-classifier}) and on the notion of univalent family of monomorphisms (\cref{def:univ-mono}).

We then introduce the notion of covering topology, which is an acyclic class containing the class of surjections.
Covering topologies can be thought of as a version of pretopologies which do not depend on the choice of a generating category. 
We shall see in \cref{cor:cov=invsurj} that they are exactly the classes of maps that become surjective in some localization.
In \cref{thm:covering-top}, we prove that covering topologies are in bijection with extended Grothendieck topologies.
We will see across the paper that it is sometimes more convenient to present a localization in terms of the maps that become surjective instead of the maps that become invertible.
\Cref{thm:tripleadj-acyclic} shows that covering topologies are a natural notion in the articulation of acyclic classes and extended Grothendieck topologies.

We define also the notion of a Lawvere--Tierney topology in \cref{LawvereTierney}, simply importing the definition from the theory of 1-topoi.
Any Lawvere--Tierney topology naturally defines a factorization system on monomorphisms (see  \cref{LTtop}) which is the restriction to monomorphism of a modality on the topos (\cref{cor:ptopmodality}).
The main result of this section is \cref{LT-topoGroth}, establishing a bijective correspondence between extended Grothendieck topologies and Lawvere--Tierney topologies.

Finally, the section explains how a topology on a topos $\cE$ is entirely determined by its restriction to the associated 1-topos $\cE\truncated 0$ of 0-truncated objects (\Cref{prop:bijLTLT,cor:bijGTGT}).

\bigskip
For purpose of reference, we have assembled below all the equivalences results proved in this section.
Given that all the notions of topology on a topos are equivalent, we can simply talk of a topology without ambiguity.
Topologies come also with a notion of sheaf.
This will be the matter of \cref{sec:GTsheaf}.

\begin{thm}
\label{thm:equivalences}
There are canonical isomorphisms between
\begin{enumerate}
\item\label{thm:equivalences:1} the poset $\GTop(\cE)$ of extended Grothendieck topologies on an \oo topos $\cE$,
\item\label{thm:equivalences:2} the poset of covering topologies on the \oo topos $\cE$,
\item\label{thm:equivalences:3} the poset of Lawvere--Tierney topologies on $\cE$, and
\item\label{thm:equivalences:4} the poset of extended Grothendieck topologies on the 1-topos $\cE\truncated 0$ of discrete objects of $\cE$.
\end{enumerate}
The posets are small and have the structure of a frame.
\end{thm}
The equivalence $\eqref{thm:equivalences:1}\Leftrightarrow\eqref{thm:equivalences:2}$ is the content of \cref{thm:covering-top}.
The equivalence $\eqref{thm:equivalences:1}\Leftrightarrow\eqref{thm:equivalences:3}$ is the content of \cref{LT-topoGroth}.
The equivalence $\eqref{thm:equivalences:3}\Leftrightarrow\eqref{thm:equivalences:4}$ is the content of \cref{cor:bijGTGT}.

\subsection{Extended Grothendieck topologies}

\begin{lemma}
\label{monoarelocal}
The class of monomorphisms in a topos $\cE$ is local.
\end{lemma}

\begin{proof}
The result can be proved directly. 
We prefer to deduce it from a general result about modalities (\cref{defmodality}). 
Both the left and the right classes of a modality are local by \cite[Proposition 3.2.7]{ABFJ:HS}.
The class of monomorphisms is the right class of a modality by
\cref{examplemodality}. 
\end{proof}

We shall see in \cref{fromacyclictoGroth} that if $\cK$ is a congruence on a topos, then the intersection $\cK\cap \Mono$ is an extended Grothendieck topology in the following sense:

\begin{definition}[Extended Grothendieck topology]
\label{definitionGrothtop}
We shall say that a class of monomorphisms $\cG$ in a topos $\cE$ is an {\it extended Grothendieck topology on $\cE$} if the following three conditions hold
\begin{enumerate}[label=\roman*)]
\item \label{definitionGrothtop:1}
    $\cG$ contains the isomorphisms and is closed under composition; 
\item \label{definitionGrothtop:2}
    $\cG$ is closed under base change and is a local class (\cref{locclass}); 
\item \label{definitionGrothtop:3}
    if the composite of two monomorphisms $u:A\to B$ and $v:B\to C$ belongs to $\cG$, then $v\in \cG$. 
\end{enumerate}
We denote by $\GTop(\cE)$ the poset of extended Grothendieck topologies on $\cE$ (ordered by inclusion).
\end{definition}

\begin{examples}
\label{ex:GTop}
We give some examples of extended Grothendieck topologies.
\begin{examplenum}
\item\label{ex:GTop:0}
    The class $\Iso$ and $\Mono$ are respectively the smallest and the largest extended Grothendieck topologies.

\item\label{ex:GTop:1}
    If $\phi:\cE\to \cF$ is an algebraic morphism of topoi, the class $\cG_\phi$ of monomorphisms inverted by $\phi$ is an extended Grothendieck topology.
    We shall see in \cref{cor:GTop=invmono} that all extended Grothendieck topologies can be produced in this way.

\item\label{ex:GTop:2}
    As a particular case, let $\S X = \fun \Fin \cS$ be the algebraically free topos on one generator of \cref{def:freetopos}. 
    The evaluation at $1\in \Fin$ provides an algebraic morphism $\S X \to \cS$.
    The class of all monomorphisms $F\to G$ such that $F(1)\simeq G(1)$ is an extended Grothendieck topology.
    The reader can look forward to \cref{ex:toploc:ooconn} for an interpretation.

\item\label{ex:GTop:3}
    More generally, if $\phi:\cE\to \cF$ is an algebraic morphism of topoi and $\cG$ is an extended Grothendieck topology on $\cF$, the class $\phi^{-1}(\cG)\cap \Mono$ is an extended Grothendieck topology on $\cE$.

\item\label{ex:GTop:4}
    Any intersection of extended Grothendieck topologies is an extended Grothendieck topology.

\end{examplenum}
\end{examples}

\begin{remark}[Terminology]
\label{rem:terminology}
Classically, Grothendieck topologies are defined on a {\it category} $ C$, as a means to define a left-exact localization of the {\it topos} $\PSh C$.
We shall see in \cref{Gtop=Gtop} that, when $\cE=\PSh C$ any Grothendieck topology on $ C$ can be {\it extended} over the whole of $\PSh C$, and that, conversely, any extended Grothendieck topology on $\PSh C$ can be {\it restricted} into a Grothendieck topology on $ C$.
\end{remark}

\medskip

Any class of monomorphisms $\Sigma\subseteq \cE$ is contained in a smallest extended Grothendieck topology $\Sigma\gtop$.
Let $\ClassMono(\cE)$ be the poset of classes of monomorphisms in $\cE$. We have an adjunction
\[
\begin{tikzcd}
\ClassMono(\cE)
\ar[rr,shift left = 1.6, "(-)\gtop"]
\ar[from=rr,shift left = 1.6, hook']
&&
\GTop(\cE)\,.
\end{tikzcd}
\]   
We shall see in \cref{cor:freeGtop} that $\Sigma\gtop = \Sigma\ac\cap \Mono = \im {\Sigma\ac}$ where $\Sigma\ac$ is the acyclic closure of $\Sigma$ (see \cref{sec:acyclic}).

\smallskip
Let $\cJ$ be a Grothendieck topology on a category $ C$ in the sense of \cite[Definition 6.2.2.1]{Lurie:HTT}.
Let $\cJ\loc$ be the smallest local class of maps in $\PSh C$ containing $\cJ$.
It is the class of monomorphisms $Y\to X$ in $\PSh C$ such that for any $R\in  C$ and any map $R \to X$ in $\PSh C$ , the pullback $R\times_YX\to R$ is in $\cJ$ (we have left implicit the notation of the Yoneda embedding).

\begin{lemma}
\label{lem:explicitGtop}
Let $\cJ$ be a Grothendieck topology on the category $ C$, then $\cJ\loc$ is the smallest extended Grothendieck topology on the topos $\PSh C$ containing $\cJ$, that is $\cJ\loc=\cJ\gtop$.
\end{lemma}
\begin{proof}
By definition, $\cJ\gtop$ is the smallest extended Grothendieck topology on $\PSh C$ containing $\cJ$.
Since $\cJ\gtop$ is local, we have always $\cJ\loc\subseteq \cJ\gtop$.
The converse will be true if we show that $\cJ\loc$ is an extended Grothendieck topology on $\PSh C$.
We leave the details of the proof to the reader:
Axioms \ref{definitionGrothtop:1} and \ref{definitionGrothtop:2} are straightforward from the  definition of $\cJ\loc$,
and Axiom \ref{definitionGrothtop:3} follows from Axiom (3) in \cite[Definition 6.2.2.1]{Lurie:HTT}.
\end{proof}

We shall say that $\cJ\loc$ is the {\it extension} of $\cJ$ to $\PSh C$.
Conversely, of $\cG$ is an extended Grothendieck topology on $\PSh C$, we define its {\it restriction} to $ C$ as the class $\cG_\mathsf{rep}\subseteq\cG$ spanned by the maps whose codomain are representable presheaves.

\begin{proposition}
\label{Gtop=Gtop}
The functions $\cJ\mapsto \cJ\gtop$ and $\cG\mapsto \cG_\mathsf{rep}$ defines inverse bijections
between the Grothendieck topologies on the category $ C$ and the extended Grothendieck topologies on $\PSh C$.
\end{proposition}
\begin{proof}
It is easy to see that $\cG_\mathsf{rep}$ is a Grothendieck topology on $ C$:
Axioms (1) and (2) of \cite[Definition 6.2.2.1]{Lurie:HTT} are easily implied by \ref{definitionGrothtop:1} and \ref{definitionGrothtop:2} of \cref{definitionGrothtop},
and Axioms (3), once formulated in terms of monomorphisms in $\PSh C$ instead of sieves, is exactly \cref{definitionGrothtop}~\ref{definitionGrothtop:3}.
Let us see that $\cG = (\cG_\mathsf{rep})\gtop$.
Certainly, we have $(\cG_\mathsf{rep})\gtop \subseteq \cG$.
Conversely, if $A\to B$ is a map in $\cG$, we consider a cover of $B$ by a family of representables functors $C_i\to B$.
All maps $C_i\times_BA\to C_i$ are in $\cG_\mathsf{rep}$.
Because $(\cG_\mathsf{rep})\gtop$ is a local class, this implies that $A\to B$ is in $(\cG_\mathsf{rep})\gtop$.
This proves $\cG \subseteq (\cG_\mathsf{rep})\gtop$ and therefore $\cG = (\cG_\mathsf{rep})\gtop$.
Let us see now that $\cJ = (\cJ\gtop)_\mathsf{rep}$.
We have always $\cJ \subseteq (\cJ\gtop)_\mathsf{rep}$.
Let $C$ be an object of $ C$ (viewed as an object of $\PSh C$) and let 
$A\to C$ be a monomorphism in $(\cJ\gtop)_\mathsf{rep}$.
The explicit description of $\cJ\gtop$ of \cref{lem:explicitGtop} says that $C$ can be covered by maps $C_i\to C$ in $ C$ such that $C_i\times_CA\to C_i$ are in $\cJ$.
In a presheaf category, the map $\coprod_i C_i\to C$ has a section since $C$ is representable.
This says that the map $A\to C$ is a base change of one of the maps $C_i\times_CA\to C_i$, hence in $\cJ$.
This finishes the proof that $(\cJ\gtop)_\mathsf{rep}=\cJ$ and of the proposition.
\end{proof}

\begin{lemma} \label{remGroth}
If a class of monomorphisms $\cM$ in a topos is
closed under base change, then 
the implication $vu\in \cM\Rightarrow u\in \cM$
holds
for any pair of monomorphisms 
$u:A\to B$ and $v:B\to C$.
\end{lemma}

\begin{proof}
The following square is cartesian, since $v$ is a monomorphism.
\[
\begin{tikzcd}
A \ar[d,"u"']   \ar[r,"{1_A}", equals] & A  \ar[d, "{vu}"]    \\
B  \ar[r, "v"] & C
\end{tikzcd}
\]
Thus, $vu\in \cM\Rightarrow u\in \cM$, since
$\cM$ is closed under base change.
   \end{proof}

\begin{lemma}[Extended Grothendieck topology associated to an acyclic class]
\label{fromacyclictoGroth}
If $\cA$ is an acyclic class (for example a congruence) in a topos $\cE$, then the class $\im \cA = \Mono \cap \cA$ of monomorphisms in $\cA$ (see \cref{lem:imagepretopo}) is an extended Grothendieck topology.
\end{lemma}

\begin{proof}
Let $\cA$ be an acyclic class, we put $\cG:=\cA\cap\Mono$.
The class $\cG$ contains the isomorphisms and it is closed under composition and base changes, since this is true of the classes $\cA$ and $\Mono$.
The class $\cA$ is local, since an acyclic class is local by \cref{acyclic-is-local}.
The class of monomorphisms $\Mono$ is also local by \cref{monoarelocal}. 
Hence their intersection $\cG$ is a local class.
If $u:A\to B$ and $v:B\to C$ are monomorphisms and $vu\in \cG$, let us show that $v\in \cG$.
Since it is a local class, $\cG$ is closed under base changes and $u\in \cG$ by \cref{remGroth}.
Thus, $v\in \cA$, since an acyclic class has the right cancellation property by \cref{acyclic-cancell}.
But then, $v\in \cG = \cA\cap\Mono$ and this completes the proof that $\cG$ is an extended Grothendieck topology.
\end{proof}

\begin{remark}
\label{rem:diagram}
Let $\LClass(\cE)$ be the poset of local classes of maps, 
$\LClass_\Delta(\cE)$ be the subposet of local classes of maps closed by diagonals, 
and $\LClassMono(\cE)$, the poset of local classes of monomorphisms.
Any local class of monomorphisms is closed under diagonals, since it contains all isomorphisms.
So we have inclusions $\LClassMono(\cE)\subseteq \LClass_\Delta(\cE)\subseteq \LClass(\cE)$.
The following diagram, in which all the squares are cartesian, summarizes the inclusion relations between various classes of maps.
\[
\begin{tikzcd}
\GTop(\cE) \ar[r, hook]\ar[d,hook'] \pbmark
&\Cong(\cE)  \ar[d,hook'] \ar[r, hook] \pbmark
& \Acyclic(\cE) \ar[d,hook']\\
\LClassMono(\cE) \ar[r]
&\LClass_\Delta(\cE) \ar[r,hook]
& \Class(\cE)
\end{tikzcd}
\]
The right hand square is cartesian because of \cref{lem:caraccong} 
and the outer square is cartesian because of \cref{fromacyclictoGroth}.
\end{remark}

\begin{proposition}
\label{prop:adjGtopAc}
The morphism of posets $\im - : \Acyclic(\cE)\to \GTop(\cE)$ of \cref{fromacyclictoGroth} has a fully faithful left adjoint $\cG \mapsto \cG\ac$.

\[
\begin{tikzcd}
\GTop(\cE)
\ar[from=rr,shift left = 1.6, "\im-"]
\ar[rr,shift left = 1.6, "(-)\ac", hook]
&&\Acyclic(\cE)
\end{tikzcd}
\]
The image of the left adjoint is the poset $\MAcyclic(\cE)$ of monogenic acyclic classes and the adjunction restricts to an equivalence 
\[
\begin{tikzcd}
\GTop(\cE)
\ar[from=rr,shift left = 1.6, "\im-", "\simeq"']
\ar[rr,shift left = 1.6, "(-)\ac", hook]
&&\MAcyclic(\cE)
\end{tikzcd}
\]
\end{proposition}
\begin{proof}
Let $\cG$ be an extended Grothendieck topology and $\cA$ an acyclic class.
We need to prove that $\cG\ac \subseteq \cA \iff \cG\subseteq \cA\cap \Mono$.
Because $\cA$ is acyclic, we have always $(\cA\cap \Mono)\ac \subseteq\cA$.
Then, if $\cG\subseteq \cA\cap \Mono$, we have $\cG\ac\subseteq (\cA\cap \Mono)\ac \subseteq \cA$.
Conversely, if $\cG\ac \subseteq \cA$, then $\cG\subseteq \cA$ and $\cG\subseteq \cA\cap \Mono$ since $\cG$ is a class of monomorphisms.

Let us show that the image of the left adjoint is the poset of monogenic acyclic classes.
For any extended Grothendieck topology $\cG$, $\cG\ac$ is a monogenic acyclic class.
Conversely, if $\cA$ is monogenic, we have $\cA = (\cA\cap \Mono)\ac$ by definition.
Since $\cA\cap \Mono$ is an extended Grothendieck topology by \cref{fromacyclictoGroth} any monogenic $\cA$ is in the image of $\cG\mapsto \cG\ac$.

The proof that the right adjoint is fully faithful is \cref{cor:GacMono=G} which will be proved independently below.
The last statement about the equivalence $\MAcyclic(\cE)\simeq \GTop(\cE)$ follows.
\end{proof}

\begin{lemma}
\label{lem:sg-gtop-to-sg-ac}
For $\Sigma$ a class of monomorphisms, we have
\[
(\Sigma \gtop)\ac
\ =\ 
\Sigma \ac\,.
\]
In particular, for an extended Grothendieck topology $\cG$, the acyclic class $\cG\ac$ is of small generation if $\cG$ is of small generation as an extended Grothendieck topology.
\end{lemma}
\begin{proof}
We have always $\Sigma \subseteq \Sigma \gtop$ and thus $\Sigma \ac \subseteq (\Sigma \gtop)\ac$.
Conversely, we have $\Sigma \subseteq \Sigma \ac \cap \Mono$ and we know from \cref{fromacyclictoGroth} that $\Sigma\ac \cap \Mono $ is an extended Grothendieck topology.
Thus, $\Sigma \gtop \subseteq \Sigma \ac\cap \Mono \subseteq \Sigma \ac$, and $(\Sigma \gtop)\ac \subseteq  \Sigma \ac$.
\end{proof}

\subsection{Univalent monomorphisms}

The purpose of this section is to prove that any extended Grothendieck topology is of small generation (\cref{Gtoparesmall}) and that the poset of extended Grothendieck topologies is small (\cref{cor:GTop-is-small}).
We do this by introducing the notion of a univalent monomorphism (\cref{def:univ-mono}).

\medskip

Let $\cE$ be a topos and $A$ an object of $\cE$.
The class of subobjects of $A$ is the class $\powerset A$ of isomorphism classes of monomorphisms $A'\to A$ in the slice category $\cE\slice A$.
We shall make the classical abuse to identify subobjects $S\subseteq A$ and monomorphisms $S\to A$.
This is fine because the space of monomorphisms representing a given subobject is contractible.

\begin{lemma}[Well-poweredness of topoi]
\label{lem:well-power}
The class of subobjects of $A$ is a set.
\end{lemma}
\begin{proof}
We prove this first in the topos $\cS$.
There, the bijection $\powerset A = \powerset {\pi_0(A)}$ shows that $\powerset A$ is a set.
The argument is similar in any topos $\cS^I$, where $I$ is a set.
Let $C$ be a small category, equipped with a set of objects $C_0$, that is a surjective map $i:C_0\to C$.
The functor $i^*:\PSh C \to \PSh {C_0}$ reflects subobjects.
Thus, for any object $A\in \PSh C$, $\powerset A$ is a subset of the set $\powerset {i^*A}$.
Finally if $\cE$ is an arbitrary topos, we use a presentation as a left-exact localization of some $\PSh C$.
The canonical inclusion $\iota:\cE\to \PSh C$ also reflects subobjects.
Thus, for any object $A\in \cE$, $\powerset A$ is a subset of the set $\powerset {\iota A}$.
\end{proof}

If $f:A\to B$ is a map in $\cE$ then the inverse image by $f$ of a subobject $S\subseteq B$ is a subobject $f^{-1}(S)\subseteq A$.
This defines the {\it inverse image map} 
$f^{-1}: \powerset B\to \powerset A$.
The resulting functor has values in sets by \cref{lem:well-power}
\[
\calP:\cE\op\stto \Set\,.
\]
This functor is called the {\it contravariant subobject functor}.

\begin{theorem}[Lawvere object/subobject classifier]
\label{mono-classifier}
The functor $\calP:\cE\op\to \Set$ is representable by an object $\Omega\in \cE$ equipped with a monomorphism $t:1\to \Omega$.
\end{theorem}
\begin{proof}
We saw in \cref{monoarelocal} that the class of monomorphisms is local, and \cref{lem:well-power} shows that the category of monomorphism over a given object is small.
Then, the result follows from \cite[Proposition 6.1.6.3]{Lurie:HTT}.
\end{proof}

The object $\Omega\in \cE$ is the {\it Lawvere object}, or the {\it subobject classifier} of the topos $\cE$.
The monomorphism $t:1\to \Omega$ is the {\it universal subobject}.
By definition, for every object $A\in \cE$ and every subobject $S\subseteq A$ there exists a unique map $\chi:A\to \Omega$, such that the following square is cartesian
\[
\begin{tikzcd}
S \ar[d]   \ar[r] & 1  \ar[d, "t"]    \\
A  \ar[r, "\chi"] &\Omega\,.
\end{tikzcd}
\]
The map $\chi$ is said to be the {\it characteristic map}, or the {\it classifying map}, of the subobject $S\subseteq A$ and we shall also denote it by $\chi_S$.
    
\medskip  

We shall denote the full subcategory of discrete objects in a topos $\cE$ by $\cE\truncated 0$. 
The subcategory $\cE\truncated 0$ is reflective and it is a 1-topos. 
Recall that an object $X$ in $\cE$ is discrete if and and only if the space $\Map A X= C(A,X)$ is discrete for every object $A\in \cE$.
Hence the object $\Omega\in \cE$ is discrete, since the space $\Map A \Omega =\powerset A$ is a set for every object $A\in \cE$.
 
\medskip
 
The following definition is inspired by Homotopy Type Theory.
 
\begin{definition}[Univalent monomorphism]
\label{def:univ-mono}
We say that a monomorphism $v:T\to  V$ in a topos $\cE$ is {\it univalent} 
if its classifying map $\chi_T: V\to \Omega$ is a monomorphism.
\end{definition}

The codomain $V$ of a univalent monomorphism $v:T\to V$ defines a subobject of $\Omega$, since the map $\chi_T: V\to \Omega$ is monic.
Conversely, every subobject $i:V\subseteq \Omega$ is the codomain of the univalent monomorphism $v:T\to  V$ defined by the following pullback square:
\[
\begin{tikzcd}
T \ar[d,"v"']   \ar[r] & 1  \ar[d, "t"]    \\
V  \ar[r, "i"] &\Omega
\end{tikzcd}
\]

\medskip

\begin{remark}
Every map $1\to \Omega$ is a monomorphism, since $\Omega$ is discrete. 
More generally, if $V\subseteq 1$ is a subterminal object, then every map $V\to \Omega$ is a monomorphism.
It follows that any inclusion of subterminal objects $U\subseteq V$ is univalent.
Geometrically, the subtopos corresponding to the localization by such a map is the union of $\sf U\cup \complement \sf V$, 
where $\sf U$ is the open subtopos corresponding to $U$, and 
$\complement \sf V$ is the closed complement of the open subtopos corresponding to $V$.
\end{remark}

If $\Sigma$ is a class of maps in a topos $\cE$, we shall denote by $\Sigma\bc$ the smallest class of maps which contains $\Sigma$ and is closed under base change.

\begin{lemma}
If $v:T\to V$ is a univalent monomorphism in a topos $\cE$, then the class $\{v\}\bc$ is local.
\end{lemma}

\begin{proof}
If a monomorphism $u:S\subseteq B$ 
is locally in $\{v\}\bc$, let us show that $u\in \{v\}\bc$.
There exists a surjective family of maps 
$\{g_i : A_i \to B \}_{i\in I}$ such that 
 the inclusion $g_i^{-1}(S)\subseteq A_i$
belongs to $\{v\}\bc$ for every $i\in I$,
since $u:S\subseteq B$ is locally 
in $\{v\}\bc$. 
Thus, for every $i\in I$ there exists a map $f_i:A_i\to V$ such that $g_i^{-1}(S)=f_i^{-1}(T)$. 
Let $\chi_T:V\to \Omega$ be the characteristic map
of the inclusion $v:T\subseteq V$
and $\chi_S:B\to \Omega$ be the characteristic map
of an inclusion $S\subseteq B$. 
We have $\chi_S g_i=\chi_T f_i$,
since $g_i^{-1}(S)=f_i^{-1}(T)$.
Hence the following square commutes for every $i\in I$.
\[
\begin{tikzcd}
A_i \ar[d, "{g_i}"']   \ar[r,"{f_i}"] & V  \ar[d, "{\chi_T}"]    \\
B  \ar[r, "{\chi_S}"] &\Omega
\end{tikzcd}
\]
It follows that the following square commutes,
where $g=(g_i\ | \ i\in I)$ and $f=(f_i\ | \ i\in I)$.
\begin{equation}
\label{descmonosq2}
\begin{tikzcd}
\bigsqcup_{i} A_i \ar[d, "g"']   \ar[r,"f"] & V  \ar[d, "{\chi_T}"]    \\
B  \ar[r, "{\chi_S}"] &\Omega
\end{tikzcd}
\end{equation}
But the map $g$ is surjective
since the family of maps $\{g_i : A_i \to B \}_{i\in I}$ 
is surjective. Moreover, the map $\chi_T$
is a monomorphism, since 
the map $v:T\to V$ is univalent. Hence the square
 (\ref{descmonosq2}) has a diagonal
filler $d:B\to V$.
We then have $d^{-1}(T)=S$, since $\chi_T d=\chi_S$.
This shows that the inclusion $S\subseteq B$
belongs to $\{v\}\bc$. 
Hence the class $\{v\}\bc$
is local. 
\end{proof}

We shall say that a univalent morphism $v:T\to V$ is a {\it univalent generator} of the local class $\{v\}\bc$.

\medskip

Let $\cM$ be a local class of monomorphisms in a topos $\cE$. 
For every object $B\in \cE$, let us denote by $\calP_\cM(B)\subseteq \powerset A$ the subset of subobjects $S\in \powerset B$ such that the inclusion $S\subseteq B$ belongs to $\cM$.
Since the set $\powerset A$ is small by \cref{lem:well-power}, so is $\calP_\cM(A)$.
The inverse image of a subobject $S\in \calP_\cM(B)$ by a map $f:A\to B$ is a subobject $f^{-1}(S)\in \calP_\cM(A)$, since the class $\cM$ is closed under base change.
This defines a functor
\[
\calP_\cM:\cE\op \stto \Set\ .
\]

\begin{lemma}
\label{replocalclass}
Let $\cM$ be a local class of monomorphisms in a topos $\cE$.
Then the functor $\calP_\cM:\cE\op \to \Set$ is represented by an element $T\in \calP_\cM(V)$  if and only if the inclusion $v:T\to V$ is univalent and $\cM=\{v\}\bc$.
\end{lemma}

\begin{proof}
($\Rightarrow$) Recall that the functor $\calP$ is represented by the universal subobject $t:1\to \Omega$ viewed as an element  $t\in \powerset \Omega$.
If the functor $\calP_\cM$ is represented by an element $T\in \calP_\cM(V)$ then, by Yoneda Lemma, the natural transformation $\calP_\cM \to \calP$ is represented by a map $h:V\to \Omega$ such that $h^{-1}(t)=T$.
In which case, we have $h=\chi_T$ by uniqueness of the classifying map of $v:T\subseteq V$.
The map $\chi_T:V\to \Omega$ is a monomorphism, since the natural transformation $\calP_\cM \subseteq \calP$ is an inclusion. 
Hence the map $v:T\subseteq V$ is univalent. 
Moreover, for every object $A\in \cE$ and every $S\in \calP_\cM(A)$ there exists a unique map $f:A\to V$ such that $f^{-1}(T)=S$, since the functor $\calP_\cM$ is represented by $T\in \calP_\cM(V)$.
Thus, $\cM=\{v\}\bc$.
The converse ($\Leftarrow$) is left to the reader. 
\end{proof}

The following result is again a special case of \cite[Proposition 6.1.6.3]{Lurie:HTT}

\begin{lemma}
\label{genlocalclass}
Every local class of monomorphisms $\cM$ in a topos $\cE$ is of the form $\{v\}\bc$ for a unique univalent monomorphism $v:T\to V$.
\end{lemma}

\begin{proof} 
Let us first show that every monomorphism $u:A\subseteq B$ in $\cM$ is the base change of a univalent morphism $k:C\subseteq D$ in $\cM$.
By \cref{mono-classifier}, there exists a map $\chi:B\to \Omega$ such that the following square is cartesian.
\begin{equation}
\label{reflectionsquare10}
\begin{tikzcd}
A \ar[d, "u"']   \ar[r] & 1  \ar[d, "t"]    \\
B  \ar[r, "\chi"] &\Omega
\end{tikzcd}
\end{equation}
The map $\chi:B\to \Omega$ admits a factorization $\chi=mp:B\to D\to \Omega$, with $p:B\to D$ a surjection and $m:D\to \Omega$ a monomorphism.
The square \eqref{reflectionsquare10} is then  the composite of two cartesian squares
\begin{equation}
\label{reflectionsquare11}
\begin{tikzcd}
A \ar[d, "u"']   \ar[r] & C  \ar[d, "k"] \ar[r]  & 1  \ar[d, "t"]    \\
B  \ar[r, "p"] &D  \ar[r, "m"] &\Omega
\end{tikzcd}
\end{equation}
where $C=m^{-1}(t)$.
The map $k:C\to D$ is a monomorphism, since the map $t:1\to \Omega$ is a monomorphism and the right hand square in (\ref{reflectionsquare11}) is cartesian.
Moreover, $k:C\subseteq D$ is univalent since the map $m$ is a monomorphism.
And we have $u\in \{k\}\bc$ since the left hand square in (\ref{reflectionsquare11}) is cartesian.  
Moreover, the map $k:C\subseteq D$ belongs to $\cM$, since $u:A\subseteq B$ belongs to $\cM$, since  the class $\cM$ is local and the map $p$ is surjective.
Let us now construct a univalent generator of $\cM$.
For this, observe that the collection of univalent morphisms in $\cM$ is a set, since the collection of subobjects of $\Omega$ is a set.
Let us denote by $\alpha:E\to F$ the coproduct of all univalent morphisms in $\cM$.
Observe that every univalent morphism in $\cM$ is a base change of $\alpha$.
The morphism $\alpha$ belongs to $\cM$, since the class $\cM$ is local.
Thus, $\alpha:E\to F$ is the base change of a univalent morphism $v:T\to V$ in $\cM$ by the first part of the proof.
Every univalent morphism in $\cM$ is a base change of $v:T\to V$, since every univalent morphism in $\cM$ is a base change of $\alpha$.
Thus, every morphism in $\cM$ is a base change of $v:T\to V$, since every morphism in $\cM$
is a base change of a univalent morphism in $\cM$ by the first part of the proof.
The proof of the uniqueness of $v$ is left to the reader.
\end{proof}

\begin{proposition}
\label{Gtoparesmall} 
Every extended Grothendieck topology $\cG$ on a topos $\cE$ is generated by a unique univalent monomorphism.
\end{proposition}

\begin{proof}
This follows from \cref{genlocalclass}, since an extended Grothendieck topology is a local class of monomorphisms by \cref{definitionGrothtop}.
\end{proof} 

\begin{corollary}
\label{cor:GTop-is-small} 
The poset of extended Grothendieck topologies on a topos $\cE$ is small.
\end{corollary}
\begin{proof}
By \cref{Gtoparesmall}, the size of this poset is bounded by the size of the poset of subobjects of $\Omega$, which is small.
\end{proof}

\subsection{Covering topologies}

In practice, Grothendieck topologies are often defined in terms of covering families instead of covering sieves.
The covering families are the families that are meant to become surjective in the localization.
This localization is then generated by the collection of images of the families, which are the covering sieves of a Grothendieck topology.

This suggests an axiomatization of the classes of maps which are the inverse image of the class $\Surj$ of surjections by some algebraic morphism of topoi.
The class $\Surj$ is an example of an acyclic class and if $\phi:\cE\to \cF$ is an algebraic morphism of topoi $\phi^{-1}(\Surj(\cF))$ is always an acyclic class in $\cE$ containing the class $\Surj(\cE)$.
This is our definition of a covering topology.
We prove in \cref{thm:covering-top} that covering topologies are in bijection with extended Grothendieck topologies.
The proof relies on \cref{prop:FromGrothtoCov} which provides an explicit description of the covering topology associated to an extended Grothendieck topology.
The main result of the section is \cref{thm:tripleadj-acyclic} which explains why covering topologies are natural objects to consider in the relation between extended Grothendieck topologies and acyclic classes.

\begin{definition}[Covering topology]
\label{def:covering-top}
A {\it covering topology} on a topos $\cE$ is an acyclic class $\cC$ containing the class of surjections $\Surj(\cE)$.
The extended Grothendieck topology associated to a covering topology is $\cC\cap \Mono$ (which we know is an extended Grothendieck topology by \cref{fromacyclictoGroth}).
\end{definition}

\begin{examples}
\label{ex:CTop}
We give some examples of covering topologies.
\begin{examplenum}
\item\label{ex:CTop:0}
    The class $\Surj$ is the smallest covering topology, and the class $\All$ of all maps is the largest.

\item\label{ex:CTop:1}
    If $\phi:\cE\to \cF$ is an algebraic morphism of topoi, the class $\cC_\phi = \phi^{-1}(\Surj)$ is a covering topology on $\cE$.
    We shall see in \cref{cor:cov=invsurj} that all covering topologies can be produced in this way.

\item\label{ex:CTop:2}
    As a particular case, let $\S X = \fun \Fin \cS$ be the algebraically free topos on one generator of \cref{def:freetopos}. 
    The evaluation at $1\in \Fin$ provides an algebraic morphism $\S X \to \cS$.
    The class $\cC$ of all maps $F\to G$ such that $F(1)\to G(1)$ is surjective is a covering topology.

    The extended Grothendieck topology $\cC\cap \Mono$ associated to $\cC$ by \cref{fromacyclictoGroth} is the one of \cref{ex:GTop:2}.

\item\label{ex:CTop:3}

Recall from \cite[II.1.3]{SGA41} 
that a pretopology on a small 1-category $ C$ is the data, for each object of $X\in  C$, of a set $Cov(X)$ of families $X_i\to X$ called covering families and satisfying some axioms that we shall not recall.
To any covering family $X_i\to X$ we associate the map $\coprod X_i \to X$ in $\PSh C$.
Let $\Sigma$ be the class of all the maps $\coprod_i X_i \to X$ associated to the pretopology.
Then, the axioms of a pretopology are such that $\im \Sigma$ is a Grothendieck topology on $ C$.
Let $\Sigma\cov$ be the smallest acyclic class containing $\Sigma$ and all surjections.
Then $\Sigma\cov$ is a covering topology such that the associated extended Grothendieck topology $\Sigma\cov\cap \Mono$ of \cref{fromacyclictoGroth} is the extended Grothendieck topology $\im \Sigma\loc$ of \cref{lem:explicitGtop}.
Moreover, we have the equivalence of forcing conditions
\[
\Forcing\Sigma\Surj
\ =\ 
\Forcing{\im\Sigma}\Iso
\ =\ 
\Forcing{\Sigma\cov}\Surj\,.
\]

\item\label{ex:CTop:5}
    Any intersection of covering topologies is a covering topology.

\end{examplenum}

\end{examples}

\begin{definition}[Covering map]
\label{def:cover}
Let $\cG$ be an extended Grothendieck topology.
A map $f$ is called a {\it $\cG$-covering} if $\im f$ is in $\cG$.
We denote by $\Cover \cG$ the class of $\cG$-coverings.
\end{definition}

The following lemma is a very useful property of covering topologies.

\begin{lemma}
\label{lem:image-acyclic}
Let $\cC$ be a covering topology, then
\[
f \in \cC \quad\iff\quad \im f \in \cC\ .
\]
In other words, we have $\cC = \Cover{(\cC\cap \Mono)} = \Cover{\im\cC}$.
\end{lemma}
\begin{proof}
We have $\coim f\in \cC$, since $\coim f\in \Surj \subseteq \cC$ by definition of covering topologies.
By \cref{lem:acyclic-image}, if $f \in \cC$, then $\im f \in \cC$. Conversely,  if $\im f \in \cC$ then $f \in \cC$
since $\coim f\in \cC$ and $\cC$ is closed under composition.
\end{proof}

\begin{lemma}
\label{lem:covGmono=G}
For any extended Grothendieck topology we have $\Cover\cG\cap \Mono = \cG$.
\end{lemma}
\begin{proof}
By definition of $\Cover\cG $, a monomorphism is in $\Cover\cG$ if and only if it is in $\cG$.
Hence $\Cover\cG \cap \Mono = \cG$.    
\end{proof}

\begin{proposition}
\label{prop:FromGrothtoCov}
Let $\cG$ be an extended Grothendieck topology on a topos $\cE$.
\begin{enumerate}
\item\label{prop:FromGrothtoCov:1} The class of $\cG$-coverings $\cG^{\mathsf{cov}}$ is a covering topology, and it is the smallest covering topology containing $\cG$.
\item\label{prop:FromGrothtoCov:2} We have $\cG^{\mathsf{cov}}=\cG\ac \vee \Surj$ where the supremum $\vee$ is taken in the poset $\mathsf{Acy}(\cE)$. 
	In particular, we have $\cG\ac \subseteq \Cover\cG$.
\item\label{prop:FromGrothtoCov:3} The class $\cG^{\mathsf{cov}}$ is of small generation as an acyclic class.
\end{enumerate}
\end{proposition}

\begin{proof}
\eqref{prop:FromGrothtoCov:1}
We need to show that that $\Cover \cG$ is acyclic and contains all surjections.
Since $\cG$ contains all isomorphisms, it is clear that $\Cover \cG$ contains all surjections (and therefore all isomorphisms).
To see that $\cG^{\mathsf{cov}}$ is acyclic, we are left to show that it is closed under base change, composition and colimits.

Let us show that the class $\cG^{\mathsf{cov}}$ is closed under base change.
Suppose that a map $f':A'\to B'$ is a base change of a map $f:A\to B$ in $\cG^{\mathsf{cov}}$.
We have $\im f\in \cG$, since $f\in \cG^{\mathsf{cov}}$.
The map  $\im {f'}$ is a base change of the map $\im f$, since the factorisation $f=\coim f \im f$ is stable under base change. 
Since the class $\cG$ is closed under base change, we have $\im {f'}\in \cG$, and thus $f'\in\cG^{\mathsf{cov}}$. 

We prove now that $\cG^{\mathsf{cov}}$ is closed under composition.
For this, we shall first prove the composite $f=vu:A\to B\to C$ of a morphism $u\in \cG$ with a surjection $v$ belongs to $\cG^{\mathsf{cov}}$.
Consider the factorisation $f=\coim f \im f:A\to I\to B$ and the following commutative diagram with a pullback square
\[
\begin{tikzcd}
A \ar[dr, "w"] \ar[drr, "u", bend left] \ar[ddr, "\coim f "', bend right, two heads] && \\
& J \ar[r, "p_2", tail] \ar[d, "p_1"'] \pbmark & A\ar[d, "v"] \\
& I \ar[r, "\im f"', tail]  & C\,.
\end{tikzcd}
\]
The map $p_2$ is a monomorphism since it is a base change of the map $\im f$.
The map $w$ is also a monomorphism, since $p_2w=u$ is a monomorphism.
Moreover, $p_2\in \cG$, since $p_2w=u\in \cG$ by the condition \ref{definitionGrothtop:3} defining an extended Grothendieck topology.
Thus,  $\im f\in \cG$ since the map $v$ is surjective and the class $\cG$ is local. 
This proves that $f\in \cG^{\mathsf{cov}}$. 
More generally, let us show that if two maps $u:A\to B$ and $v:B\to C$ belongs to $\cG^{\mathsf{cov}}$.
then so is the map $f=vu$.
We have $f=\coim v\circ  \im v \circ \coim u\circ\im u$ and the map $w:=\im v\circ \coim u$ belongs to $\cG^{\mathsf{cov}}$ by what we  have proved before.
The decomposition $f=\coim v\circ  w\circ \im u=\coim v\circ \coim w\circ \im w\circ \im u$ shows that $\im f=\im w\circ \im u$.
Thus  $\im f \in \cG$, since $\cG$ is closed under composition.
Thus, $f\in  \cG^{\mathsf{cov}}$ and this proves that $\cG^{\mathsf{cov}}$ is closed under composition.

Let us see now that the class $\cG^{\mathsf{cov}}$ is closed under colimits.
We prove first that the class is closed under coproducts (indexed by sets).
The class $\cG$ is local, hence closed under coproducts by \cref{localclasscoproduct}.
Let $I$ be a set and $(f_i,i\in I)$ a family of maps in $\Cover\cG$.
If $f=\bigsqcup f_i$, then $\im f =\bigsqcup \im {f_i}$ since the class of monomorphisms and the class of surjections are closed under coproducts.
Thus, if the map $ \im {f_i}$ belongs to $\cG$ for every $i\in I$, then the map $\im f$ belongs to $\cG$.
This shows that the class $\cG^{\mathsf{cov}}$ is closed under coproducts.
We can now prove that the class $\cG^{\mathsf{cov}}$ is closed under colimits.
Let $f':A'\to B'$ be the colimit of a diagram of maps
$f_i:A_i\to B_i$ in $\cG^{\mathsf{cov}}$, indexed by a small category $I$.
We shall prove that $f'$ is in $\cG^{\mathsf{cov}}$.
We have a commutative square of canonical maps 
\[
\begin{tikzcd}
A_i \ar[d,"f_i"']   \ar[r,"{a_i}"] & A'  \ar[d, "{f'}"]    \\
B_i  \ar[r, "b_i"] & B'
\end{tikzcd}
\]
for every $i\in I$.
If $f:A\to B$ is the coproduct of the maps $f_i:A_i\to B_i$ over $i\in I$, then we have a commutative square of canonical maps 
\[
\begin{tikzcd}
A \ar[d,"f"']   \ar[r,"{a}"] & A'  \ar[d, "{f'}"]    \\
B  \ar[r, "b"] & B'
\end{tikzcd}
\]
in which the maps $a$ and $b$ are surjective.
The map $f$ belongs to $\cG^{\mathsf{cov}}$, since the class $\cG^{\mathsf{cov}}$ is closed under coproducts.
Moreover, $b\in \cG^{\mathsf{cov}}$ since $b$ is a surjection. 
Thus, $f'a = bf\in \cG^{\mathsf{cov}}$, since the class $\cG^{\mathsf{cov}}$ is closed under composition. 
Thus, $\im {f'a} \in \cG^{\mathsf{cov}}$. 
But we have $\im {f'a} =\im {f'}  $, since the map $a$ is surjective.
Thus $\im {f'} \in \cG$ and hence $f'\in \cG^{\mathsf{cov}}$.
This finishes to show that the class $\cG^{\mathsf{cov}}$ is closed under colimits.

We have proved that $\Cover \cG$ is a covering topology.
Let us see that it is the smallest one containing $\cG$.
Let $\cC$ be a covering topology.
If $\cG\subseteq \cC$, then we have $\cG \subseteq \cC\cap \Mono$ and $\Cover \cG \subseteq (\cC\cap \Mono)\cov$.
By \cref{lem:image-acyclic}, we have $\cC = (\cC\cap \Mono)\cov$.
This shows that $\cG\subseteq \cC$ implies $\Cover \cG \subseteq \cC$.
Conversely, if $\Cover \cG \subseteq \cC$, we get $\cG = \Cover \cG\cap \Mono \subseteq \cC$ using \cref{lem:covGmono=G}.
This proves that $\cG\subseteq \cC$ if and only if $\Cover \cG \subseteq \cC$.

\smallskip
\noindent\eqref{prop:FromGrothtoCov:2}
We have always, $\Surj \subseteq \Cover\cG$.
Moreover, since the class $\Cover \cG$ is in particular acyclic, we have also $\cG\ac\subseteq \Cover \cG$.
This proves that $\cG\ac \vee \Surj \subseteq \Cover \cG$.
Conversely, $\cG\ac \vee \Surj$ is a covering topology containing $\cG$ and the minimality property proved above gives the reverse inclusion.

\smallskip
\noindent\eqref{prop:FromGrothtoCov:3}
We have seen in \Cref{Gtoparesmall} that any extended Grothendieck topology $\cG$ is of small generation.
Using \cref{lem:sg-gtop-to-sg-ac}, we deduce that the acyclic class $\cG\ac$ is of small generation.
The acyclic class $\Surj$ is generated by the single map $s^0:S^0\to 1$ by \cref{lem:gen-conn-trunc}.
Thus, using that $\Cover \cG = \cG\ac \vee \Surj$, the acyclic class $\Cover \cG$ is of small generation.
\end{proof}

\begin{remark}
\cref{prop:FromGrothtoCov}~\eqref{prop:FromGrothtoCov:1} shows in particular that $\Cover \cG$ is an acyclic class.
Moreover \cref{prop:FromGrothtoCov}~\eqref{prop:FromGrothtoCov:3} and \cref{rem:acyclic2modality} shows that it is the left class of a modality.
We shall identify the corresponding right class in \cref{cor:ptopmodality}.
\end{remark}

\begin{corollary}
\label{cor:GacMono=G}
For any extended Grothendieck topology $\cG$, we have $\cG = \cG\ac \cap \Mono$.
\end{corollary}
\begin{proof}
We have always $\cG\subseteq \cG\ac \cap \Mono$.
Conversely, using that $\cG\ac \subseteq \Cover\cG$ from \cref{prop:FromGrothtoCov}~\eqref{prop:FromGrothtoCov:2}, and using that $\Cover\cG\cap \Mono = \cG$ from \cref{lem:covGmono=G}, we get that $\cG\ac \cap \Mono \subseteq \cG$.
\end{proof}

\begin{corollary}
\label{cor:freeGtop}
If $\Sigma$ is a set of maps in a topos $\cE$,
then ${\im\Sigma}\gtop = \im {\Sigma\ac} = \Sigma\ac\cap\Mono$.
\end{corollary}

\begin{proof}
The class $\im {\Sigma\ac}={\Sigma\ac}\cap \Mono$ is an extended Grothendieck topology by \cref{fromacyclictoGroth}.
Thus, ${\im\Sigma}\gtop \subseteq \im {\Sigma\ac}$ since ${\im\Sigma} \subseteq \im {\Sigma\ac}$.
Conversely, the class $\mathsf{im}^{-1} (\im {\Sigma}\gtop) = \Cover{\big(\im {\Sigma}\gtop\big)}$ is acyclic by \cref{prop:FromGrothtoCov}~\eqref{prop:FromGrothtoCov:2}.
Hence we have 
$\Sigma\ac \subseteq \mathsf{im}^{-1} (\im \Sigma \gtop)$,
since $\Sigma \subseteq \mathsf{im}^{-1}(\im \Sigma)\subseteq  
\mathsf{im}^{-1} (\im \Sigma \gtop)$.
Thus, $\im {\Sigma\ac} \subseteq \im \Sigma \gtop$
and this proves that $\im {\Sigma\ac}= \im \Sigma \gtop$.
\end{proof}

\begin{theorem}[Equivalence Grothendieck/covering topologies]
\label{thm:covering-top}
Let $\cE$ be a topos.
The maps $\cG\mapsto \Cover \cG$ and $\cC\mapsto \im \cC = \cC\cap \Mono$ define inverse isomorphisms between the poset $\CTop(\cE)$ of covering topologies and the poset $\GTop(\cE)$ of extended Grothendieck topologies.
\[
\begin{tikzcd}
\GTop(\cE) \ar[rr,shift left = 1.6, "\Cover{(-)}"] \ar[from=rr,shift left = 1.6, "\im-", "\simeq"']
&&\CTop(\cE)
\end{tikzcd}
\]
\end{theorem}
\begin{proof}
The map $\cC\mapsto \cC\cap \Mono$ defines a morphism of posets $\CTop(\cE) \to \GTop(\cE)$ by \cref{fromacyclictoGroth}.
And the map $\cG\mapsto \Cover \cG$ defines a morphism of posets $\GTop(\cE)\to \CTop(\cE)$ by \cref{prop:FromGrothtoCov}~\eqref{prop:FromGrothtoCov:1}.
The equality $\Cover{(\cC\cap \Mono)} = \cC$ is the statement of \cref{lem:image-acyclic}.
The relation $\Cover\cG\cap \Mono = \cG$ is \cref{lem:covGmono=G}.
\end{proof}

The notion of covering topology provides the following enhancement of \cref{prop:adjGtopAc}.

\begin{theorem}[Adjunctions between acyclic classes and extended Grothendieck topologies]
\label{thm:tripleadj-acyclic}
The morphism of posets $\im - =\Mono\cap -:\Acyclic(\cE) \to \GTop(\cE)$ admits 
\begin{enumerate}
\item a fully faithful left adjoint $(-)\ac$ whose image is the subposet $\MAcyclic(\cE)$ of monogenic congruences, and
\item a fully faithful right adjoint $\Cover{(-)} = (-)\ac\vee \Surj$ whose image is the subposet $\CTop(\cE)$ of covering topologies.
\[
\begin{tikzcd}
\Acyclic(\cE)
\ar[from=rr,shift left = 3, "\Cover{(-)}", hook']
\ar[rr,"\im-"' description ]
\ar[from=rr,shift right = 3, "(-)\ac"', hook']
&&\GTop(\cE)
\end{tikzcd}
\]
\end{enumerate}
\end{theorem}

\begin{proof}
The first statement is \cref{prop:adjGtopAc}.
We prove the second one.
Let $\cG$ be an extended Grothendieck topology and $\cA$ an acyclic class.
By definition of $\Cover \cG$, we have 
$\cA \subseteq \Cover \cG \iff \im \cA \subseteq \cG$.
This proves that $\cG\mapsto \Cover\cG$ is right adjoint to $\im-$.
The fact that it is fully faithful can be seen from \cref{cor:GacMono=G}, or using the triple adjunction and the fact that the morphism $(-)\ac$ is fully faithful.
\end{proof}

\begin{remark}
\label{rem:tripleadjacyclic}
We have seen in \cref{prop:adjGtopAc} that the image of the morphism $(-)\ac:\GTop(\cE)\to \Acyclic(\cE)$ is the poset of monogenic acyclic classes.
We can deduce from \cref{thm:covering-top,thm:tripleadj-acyclic} that the image of the morphism $\Cover{(-)}:\GTop(\cE)\to \Acyclic(\cE)$ is the poset of covering topologies.
Altogether, this provide the following isomorphisms
\[
\begin{tikzcd}
\MAcyclic(\cE) \ar[rr,hook]&& \Acyclic(\cE) \ar[dd] && \CTop(\cE) \ar[ll,hook']
\\
\\
&&\GTop(\cE)
\ar[rruu,equal,"\mathrm{\cref{thm:covering-top}}"',"\Cover{(-)}"]
\ar[lluu,equal,"\mathrm{\cref{prop:adjGtopAc}}","(-)\ac"']
\end{tikzcd}\,.
\]

Using these equivalences, the triple adjunction of \cref{thm:tripleadj-acyclic} can be presented in other ways, more suited for some applications.
\[
\begin{tikzcd}
\Acyclic(\cE)
\ar[rrr,"{(-\cap \Mono)\ac}" description]
\ar[from=rrr, shift left = 4,"{(-)\vee \Surj}",hook']
\ar[from=rrr, shift right = 4,"can."', hook']
&&& \MAcyclic(\cE)
\end{tikzcd}
\qquad
\begin{tikzcd}
\Acyclic(\cE)
\ar[rr,"{(-)\vee \Surj}" description]
\ar[from=rr, shift left = 4,"can.",hook']
\ar[from=rr, shift right = 4,"{(-\cap \Mono)\ac}"', hook']
&& \CTop(\cE)
\end{tikzcd}
\]
\end{remark}

\begin{remark}
\label{rem:all-phi-classes}
Let $\phi:\cE\to \cF$ be a left-exact localization of topoi.
The transport bijections of \cref{prop:transport-acyclic} provide the following correspondence between acyclic classes in $\cE$ and $\cF$.
(The last two formulas use implicitly \cref{ex:lexcore:4}.)
\begin{align*}
\cK_\phi & =\phi^{-1}(\Iso) &&\textrm{is a congruence by \cref{exmpcongruence4}}\\
\cG_\phi & =\phi^{-1}(\Iso)\cap \Mono &&\textrm{is an extended Grothendieck topology by \cref{ex:GTop:1}}\\
\cC_\phi & =\phi^{-1}(\Surj) &&\textrm{is a covering topology by \cref{ex:CTop:1}}\\
\decn {n+1} {\cC_\phi} & =\phi^{-1}(\Conn n) && \textrm{(where $\dec-$ is  defined in \cref{sec:acvcong})}\\
\decinfty {\cC_\phi} & =\phi^{-1}(\Conn \infty) &&\textrm{is a hypercomplete congruence (see \cref{sec:hypercoverings})}
\end{align*}
We shall see in \cref{lem:hyper}~\eqref{lem:hyper:2} that $\cC_\phi = \Cover{(\cG_\phi)}$.
\end{remark}

\subsection{Lawvere--Tierney topologies}

In this section, we define the notion of a Lawvere--Tierney topology on a topos $\cE$ (\cref{LawvereTierney}).
The definition is merely a transposition of the classical notion for 1-topoi, and we shall see in \cref{prop:bijLTLT} that the notion depends only on the underlying 1-topos of the topos $\cE$.
The main results of the section are \cref{LTtop}, where the factorization system on monomorphisms encoded by a Lawvere--Tierney topology is constructed, 
and \cref{LT-topoGroth}, where the bijection with extended Grothendieck topologies is proved.

\medskip

Let $\cE$ be a topos. 
For every object $A\in \cE$, the set $\powerset A$ of subobjects of $A$ is partially ordered by the inclusion relation.
The presheaf $\calP:\cE\op\to \Set$ can be enhanced into a presheaf with values in the category of posets.
It follows that the Lawvere object $\Omega$, which is representing $\calP$ by \cref{mono-classifier}, is naturally partially ordered by a binary relation $\lsem\leq\rsem\subseteq \Omega\times \Omega$.

\begin{definition}[Lawvere--Tierney topology]
\label{LawvereTierney}
We shall say that an endomorphism
$j:\Omega\to \Omega$ of the Lawvere object
$\Omega$ of a topos $\cE$
is a {\it Lawvere--Tierney topology on $\cE$}
if it is a closure operator, that is,
if the following three conditions hold: 
\begin{enumerate}[label=\roman*)]
\item\label{LawvereTierney:1} $j$ is monotonic: $\ x\leq y \Rightarrow jx \leq jy$, for any $A\in \cE$ and any maps $x,y:A\to \Omega$;
\item\label{LawvereTierney:2} $j$ is inflating: $\ x\leq jx$, for any $A\in \cE$ and any map $x:A\to \Omega$;
\item\label{LawvereTierney:3} $j$ is idempotent: $jj=j$.
\end{enumerate}
\end{definition}

A closure operator $j:\Omega\to \Omega$
is the same thing as a closure operator 
$j:\calP\to \calP$ on the presheaf represented by $\Omega$.
The map $j_A:\powerset A\to
\powerset A$ is a closure operator on the poset $\powerset A$
for every object $A\in \cE$.
We have $u^{-1}j_B(S)=j_A(u^{-1}(S))$
for every map $u:A\to B$ in $\cE$
and every $S\in \powerset B$ 
since the operator $j:\calP\to \calP$
is a natural transformation.
In particular, 
\begin{equation}
\label{closureop}
j_B(S)\cap T \ =\ j_T(S\cap T)
\end{equation}
for every $S,T\subseteq B$.

\begin{definition} \label{j-closed-dense} 
If $j:\Omega\to \Omega$ is a closure operator, we shall say that a monomorphism $S\to A$, or a subobject $S\subseteq A$, is $j$-{\it dense} (resp. $j$-closed) if $j_A(S)=A$ (resp. $j_A(S)=S$).
We shall denote class of $j$-dense (resp. $j$-closed) monomorphisms by $\Dense j$ (resp. $\Close j$). 
\end{definition}

For example, if $S\in \powerset A$ then the 
subobject $ j_A(S)\subseteq A$ is $j$-closed, since
we have 
$ j_A j_A(S)=j_A(S)$ by the idempotence of $j_A$.
We have $S\subseteq j_A(S)$ by inflation.
Let us show that $S$ is $j$-dense in  $C:=j_A(S)$.
For this, we have to show that $j_C(S)=C$.
But $j_C(S)=j_C(S\cap C)=j_A(S)\cap C=C$
by \eqref{closureop}. Hence the inclusion $S\subseteq j_A(S)$ is $j$-dense.
It follows from these observations 
that the inclusion
$S\subseteq A$ is the composite of a $j$-dense
inclusion  $S\subseteq j_A(S)$ followed 
by a $j$-closed inclusion $j_A(S)\subseteq A$.
We shall see below that the closure operator $j$
is defining a factorization system
in the (non-full) subcategory of monomorphisms 
$\Mono=\Mono(\cE)$.

\begin{proposition}[Factorization system of a topology]
\label{LTtop}
Let $j:\Omega\to \Omega$ be a Lawvere--Tierney topology
on a topos $\cE$. Then,
\begin{enumerate}
\item \label{LTtop:1} $\Close j=\rorth{\Dense j} \cap\Mono$
and $\Dense j= \lorth {\Close j} \cap\Mono$;
\item \label{LTtop:2} Every monomorphism $w:A\to C$ in $\cE$ is the composite
of a monomorphism $u:A\to B$ in $\Dense j$ 
followed by a monomorphism $v:B\to C$ in  $\Close j$, and this decomposition is unique;
\item \label{LTtop:3} The classes $\Dense j$ and
$\Close j$ contain the isomorphisms
and they are closed under composition;
\item \label{LTtop:4} The classes $\Dense j$ and
$\Close j$ are local;
\item \label{LTtop:5}  the class $\Dense j$ is an extended Grothendieck topology.
\end{enumerate}
\end{proposition}

\begin{remark}[Orthogonality on monomorphisms]
\label{rem:orthomono}
The orthogonality $\perp$ of maps and classes of maps restricts to an orthogonality relation $\perp_m$ on monomorphisms and classes of monomorphisms.
If $\cM$ is a class of monomorphisms its left orthogonal is ${}^{\perp_m}\cM = \lorth {\cM} \cap\Mono$ and its right orthogonal is $\cM^{\perp_m}=\rorth{\cM} \cap\Mono$.
This is the meaning of the formula in \cref{LTtop}~\eqref{LTtop:1}.
The pair $(\Dense j,\Close j)$ is a factorization system on monomorphisms relative to the orthogonality relation $\perp_m$.
\end{remark}

\begin{proof}
\eqref{LTtop:1}
Let us first show that $\Dense j\perp \Close j$.
If a monomorphism $u:A\subseteq B$ is $j$-dense
and a monomorphism $z: Y\subseteq Z$ is $j$-closed
let us show 
that every
commutative square has a unique
diagonal filler
\begin{equation}
  \label{diagfillerunique}
    \begin{tikzcd}
      A \ar[d,"u"']   \ar[r,"f"] & Y  \ar[d, "z"]    \\
    B \ar[r, "g"] & Z
    \end{tikzcd}
    \end{equation}
By hypothesis, we have $B=j_B(A)$ and $Y=j_Z(Y)$.
Moreover, we have $A\subseteq g^{-1}(Y)$,
since the square (\ref{diagfillerunique}) commutes.
Thus, 
\[
B=j_B(A)\subseteq j_B(g^{-1}(Y))=g^{-1}(j_Z(Y))= g^{-1}(Y)
\]
It follows that $g$ induces a diagonal filler $d:B\to Y$ for the square
 (\ref{diagfillerunique}).
The uniqueness of the diagonal filler is clear, since $z$ is a monomorphism.
This proves the relation $\Dense j\perp \Close j$.
  Hence we have $\Close j\subseteq \rorth{\Dense j}$
and $\Dense j\subseteq \lorth{\Close j}$.

Let us now show that $\rorth{\Dense j}  \cap\Mono\subseteq \Close j$.
If the inclusion $z: Y\subseteq Z$ belongs to
$\rorth{\Dense j}$, consider 
the following square of inclusions
 \begin{equation}
  \label{diagfillerunique2}
    \begin{tikzcd}
      Y \ar[d,"u"']   \ar[r,"{1_Y}"] & Y  \ar[d, "z"]    \\
      j_Z(Y)\ar[r, "v"] & Z
    \end{tikzcd}
    \end{equation}
We saw above that the inclusion $u:Y\subseteq j_Z(Y)$
is $j$-dense. Hence the square \eqref{diagfillerunique2}
has a diagonal filler $ j_Z(Y)\to Y$.
It follows that $j_Z(Y)=Y$ and this shows that the inclusion $z:Y\subseteq Z$ is $j$-closed. 
This proves the equality $\rorth{\Dense j} \cap\Mono=\Close j$.

Dually, let us show that $\lorth{\Close j}\cap\Mono \subseteq \Dense j$.
If the inclusion $z: Y\subseteq Z$ belongs to $\lorth{\Close j}$, consider 
the following square of inclusions
 \begin{equation}
  \label{diagfillerunique3}
    \begin{tikzcd}
      Y \ar[d,"z"']   \ar[r,"{u}"] & j_Z(Y)  \ar[d, "v"]    \\
     Z \ar[r, "{1_Z}"] & Z
    \end{tikzcd}
    \end{equation}
We saw above that the inclusion $v: j_Z(Y)\to Z$
is $j$-closed. Hence the square \eqref{diagfillerunique3}
has a diagonal filler.   
It follows that  $Z= j_Z(Y)$
and this shows that the inclusion $z:Y\subseteq Z$
is $j$-dense.
This proves the equality $\lorth{\Close j}\cap\Mono =\Dense j$.
 
\smallskip
\noindent 
\eqref{LTtop:2} We saw above 
that every inclusion
$S\subseteq A$ is the composite of a $j$-dense
inclusion  $S\subseteq j_A(S)$ followed 
by a $j$-closed inclusion $j_A(S)\subseteq A$.
The unicity of this decomposition follows from the 
orthogonality $\Dense j\perp \Close j$ 
proved in \eqref{LTtop:1}.

  \smallskip
\noindent  
\eqref{LTtop:3} The identity map $1_A:A\to A$ is both $j$-closed and $j$-dense for every object $A\in \cE$, 
since $j_A(A)=A$. It follows that every isomorphism belongs to $\Dense j$ and $ \Close j$. 
The closure under composition of the classes $\Dense j$ 
and $ \Close j$ 
follows from \eqref{LTtop:1}.

\smallskip
\noindent \eqref{LTtop:4}
We have $g^{-1}j_B(S)=j_A(g^{-1}(S))$ for every map $g:A\to B$ in $\cE$ and every $S\in \powerset B$. 
If $j_B(S)=B$, then $j_A(g^{-1}(S))=A$ and this shows that the class $\Dense j$ is closed under base change.
Moreover,  if $j_B(S)=S$, then $j_A(g^{-1}(S))=g^{-1}(S)$. Thus the class $\Close j$ is also closed under base change.
Let us show that the class $\Dense j$ is local.
If an inclusion $u:S\subseteq B$ is locally $j$-dense, let us show that $u$ is $j$-dense.
By the hypothesis, there exists a surjective family of maps  $\{g_i:A_i\to B\}_{i\in I}$ such the inclusion $g_i^{-1}(S)\subseteq A_i$ is $j$-dense for every $i\in I$.
We then have $g_i^{-1}(j_{B}S)=j_{A_i}g_i^{-1}(S)=A_i$ for every $i\in I$. 
Thus, $j_{B}S=B$, since the family of maps $\{g_i:A_i\to B\}_{i\in I}$ is surjective.
Hence the inclusion $u:S\subseteq B$ is $j$-dense and this shows that the class $\Dense j$ is local.
It remains to show that the class $\Close j$ is local.
If an inclusion $u:S\subseteq B$ is locally $j$-closed, let us show that $u$ is $j$-closed. 
By the hypothesis, there exists a surjective family of maps  $\{g_i:A_i\to B\}_{i\in I}$ such the inclusion $g_i^{-1}(S)\subseteq A_i$ is $j$-closed for every $i\in I$.
We then have $g_i^{-1}(j_{B}S)=j_{A_i}g_i^{-1}(S)= g_i^{-1}(S)$ for every $i\in I$. 
Thus, $j_{B}S=S$, since the family of maps $\{g_i:A_i\to B\}_{i\in I}$ is surjective.
Hence the inclusion $u:S\subseteq B$ is $j$-closed and this shows that the class $\Close j$ is local.

\smallskip
\noindent \eqref{LTtop:5}
It remains to show that the third condition of \cref{definitionGrothtop} holds for the class $\cG:=\Dense j$.
For this,  we have to show that 
if the composite of two
inclusion $u:A\subseteq B$ and $v:B\subseteq C$
is $j$-dense, then the inclusion $v:B\subseteq C$ is $j$-dense. We have $j_C(A)=C$, since $vu$ is $j$-dense.
But $j_C(A)\subseteq j_C(B)$, since $A\subseteq B$.
It follows that $j_C(B)=C$ and hence that $v$ is $j$-dense.
\end{proof}

We shall prove in \cref{LT-topoGroth}  that 
for any extended Grothendieck topology $\cG$ 
in a topos $\cE$,
there
exists a unique Lawvere--Tierney topology
$j:\Omega\to \Omega$ such that $\cG=\Dense j$.

\medskip

Let $\cM:=\{v\}\bc$ be the local class of monomorphisms generated
by a univalent monomorphism $v:T\to V$.
The classifying map $\chi_T:V\to \Omega$
has itself a classifying map $j:\Omega\to \Omega$:
\begin{equation}
  \label{reflectionsquare23}
      \begin{tikzcd}
      T \ar[d, "v"']   \ar[r] &1  \ar[d, "t"]    \\
      V  \ar[r, "{\chi_T}"] &\Omega
    \end{tikzcd} \quad \quad \quad 
    \begin{tikzcd}
      V \ar[d, "\chi_T"']   \ar[r] &1  \ar[d, "t"]    \\
      \Omega  \ar[r, "{j}"] &\Omega
    \end{tikzcd}
    \end{equation}
    Let us denote by $j:\calP\to \calP$
    the natural transformation defined
    by the operator $j:\Omega \to \Omega$.

\begin{lemma} \label{keylemmaclosure}
A monomorphism $u:S\subseteq A$ 
belongs to $\cM$ if and
only if $j_A(S)=A$.
In general, $j_A(S)\subseteq A$
is the largest subobject $S'\subseteq A$
such that the inclusion $S'\cap S\subseteq S'$
belongs to $\cM$.
\end{lemma}

\begin{proof} 
For every object $A\in \cE$,
let us denote by $ \calP_\cM(A)$
the set of subobjects $S\in  \powerset A$
whose inclusion $S\subseteq A$ belongs to the class $\cM$.
An inclusion $S\subseteq A$ belongs
to $\cM:=\{v\}\bc$ if and only if the map $\chi_S:A\to \Omega$ factors through the map $\chi_T:V\to \Omega$,
since the left hand square in \eqref{reflectionsquare23} is a pullback.    
Hence the following square is a pullback, since the 
right hand square in \eqref{reflectionsquare23} is a pullback.    
\[
\begin{tikzcd}
\calP_\cM(A) \ar[d]   \ar[r] & 1  \ar[d, "1_A"]    \\
\powerset A  \ar[r, "j_A"] & \powerset A
\end{tikzcd}
\]
Thus, an inclusion $S\subseteq A$ belongs to the class $\cM$
if and only if $j_A(S)=A $.
If $f:B\to A$, then $f^{-1}( j_A(S))=j_B(f^{-1}(S))$
by naturality of $j:  \calP\to  \calP$. Thus, $f^{-1}( j_A(S))=B$
if and only if the inclusion $f^{-1}(S)\subseteq B$
belongs to $\cM$. But $f^{-1}( j_A(S))=B$ if and only if $\Im f\subseteq  j_A(S)$.
Thus, $\Im f\subseteq  j_A(S)$ if 
and only if the inclusion $f^{-1}(S)\subseteq B$
belongs to $\cM$. In particular, if the map $f:B\to A$
is the inclusion $S'\subseteq A$ of a subobject
$S'\in \powerset A$, then $S'\subseteq  j_A(S)$ if 
and only if the inclusion $S'\cap S\subseteq S'$
belongs to $\cM$. 
\end{proof}

We saw in \cref{LTtop} that if $j:\Omega\to \Omega$ is a Lawvere--Tierney topology
in a topos $\cE$, then the class $\Dense j$ of $j$-dense monomorphisms is an extended Grothendieck topology.
Let $\LTTop(\cE)$ be the poset of Lawvere--Tierney topologies ordered by inclusion of the classes $\Dense j$.

\begin{theorem}[Equivalence Grothendieck/Lawvere--Tierney topologies]
\label{LT-topoGroth}
The map $j\mapsto \Dense j$ provides an isomorphism between the poset of Lawvere--Tierney topologies and that of extended Grothendieck topologies.
\[
\begin{tikzcd}
\LTTop(\cE) \ar[rr, "\Dense-","\simeq"'] &&\GTop(\cE)
\end{tikzcd}
\]    
\end{theorem}

\begin{proof}
If $j_1$ and $j_2$ are closure operators $\Omega\to \Omega$, let us show that $j_1\leq j_2\Leftrightarrow \Dense {j_1}\subseteq \Dense {j_2}$.
If $j_1\leq j_2$, then $(j_1)_A(S)\leq (j_2)_A(S)$ for every object  $A\in \cE$ and every subobject $S\subseteq A$.
Thus, $(j_1)_A(S)=A\Rightarrow (j_2)_A(S)=A $ and this shows that $\Dense {j_1}\subseteq \Dense {j_2}$.
Conversely, if $\Dense {j_1}\subseteq \Dense {j_2}$ let us show that $(j_1)_A(S)\leq (j_2)_A(S)$ for every object  $A\in \cE$ and every subobject $S\subseteq A$.
If $C:=(j_1)_A(S)$, then the inclusion $S\subseteq C$ is $j_1$-dense, since $j_1)_C(S)=(j_1)_C(S\cap C)=(j_1)_A(S)\cap C=C$ by formula \eqref{closureop}.
It follows that the inclusion $S\subseteq C$ is $j_2$-dense,  since $\Dense {j_1}\subseteq \Dense {j_2}$.
Thus, $(j_2)_C(S)=C$ and hence $(j_2)_A(S)\cap C=(j_2)_C(S\cap C)=(j_2)_C(S)=C$ by formula \eqref{closureop}.
It follows that $(j_1)_A(S)\subseteq (j_2)_A(S)$ for every object $A\in \cE$ and every subobject $S\subseteq A$.

It remains to show that for any extended Grothendieck topology $\cG$ there exists a closure operator $j:\Omega \to \Omega$ such that $\cG=\Dense j$.
The local class $\cG$ is generated by a univalent monomorphism $v:T\to V$ by \cref{Gtoparesmall}.
Let us first verify that the map $j:\Omega \to \Omega$ which classifies the monomorphism $\chi_T:V\to \Omega$.
is a closure operator.
For this it suffices to show that the resulting map $j_A:\powerset A\to \powerset A$ is a closure operator for every object $A\in \cE$.
We shall use \cref{keylemmaclosure}.

Let us see first that the map $j_A$ is monotonic.
For this, we must show that $S\subseteq T \Rightarrow j_A(S)\subseteq j_A(T)$ for every $S,T \in \powerset A$.
The inclusion $ j_A(S)\cap S\subseteq j_A(S)$ belongs to $\cG$ by \cref{keylemmaclosure} with $S'= j_A(S)$.
Hence the inclusion $ j_A(S)\cap T\subseteq j_A(S)$ belongs to $\cG$ by the condition \cref{definitionGrothtop}.\ref{definitionGrothtop:3}, since $j_A(S)\cap S\subseteq j_A(S)\cap T \subseteq j_A(S)$.
Thus, $j_A(S)\subseteq j_A(T)$ by \cref{keylemmaclosure}.
Let us now show that the map $j_A$ is inflating.
The inclusion $1_S:S=S$ belongs to $\cG$ for every $S\subseteq A\in \cE$, since $\cG$ contains the isomorphisms.
Thus, $S\subseteq  j_A(S)$ by \cref{keylemmaclosure}.

Let us show now that $j_A$ is idempotent.
Obviously, we have $j_A(S)\subseteq j^2_A(S)$ for every $S\subseteq A$, since we have $S\subseteq j_A(S)$ and $j_A$ is monotonic.
It remains to show that $j^2_A(S)\subseteq j_A(S)$.
The inclusion $j_A(S)\cap S\subseteq j_A(S)$ belongs to $\cG$ for every $S\subseteq A$ by \cref{keylemmaclosure}.
But we have $S\subseteq j_A(S)$ since $j_A$ is inflating.
Hence the inclusion $S\subseteq j_A(S)$ belongs to $\cG$.
It we apply this result to the subobject $j_A(S)\subseteq A$, instead of the subobject $S\subseteq A$, we obtain that the inclusion $j_A(S)\subseteq j_A^2(S)$ belongs to $\cG$.
It follows by composing $S\subseteq j_A(S)\subseteq j_A^2(S)$ that the inclusion $S\subseteq j^2_A(S)$ belongs to $\cG$, since $\cG$ is closed under composition.
Thus, $j^2_A(S)\subseteq j_A(S)$ by \cref{keylemmaclosure}, since the inclusion $S\cap j^2_A(S)\subseteq  j^2_A(S)$ belongs to $\cG$.
We have proved that $j_A$ is idempotent.

This completes the proof that $j$ is a closure operator.
By definition, a monomorphism $S\subseteq A$ is $j$-dense if and only if $j_A(S)=A$ if and only if the inclusion $S\subseteq A$ belongs to $\cG$ by \cref{keylemmaclosure}.
Thus, $\Dense j=\cG$. 
The existence of the closure operator $j:\Omega\to \Omega$ is proved. 
\end{proof}

\begin{remark}
\label{rem:orthomono2}
Recall the orthogonality relation $\perp_m$ of \cref{rem:orthomono}.
Then a consequence of \cref{LT-topoGroth} is that any extended Grothendieck topology $\cG$ is saturated as a class of monomorphisms in the sense that $\cG = {}^{\perp_m}(\cG^{\perp_m})$.
Another way to say this is that any extended Grothendieck topology is always the left class of a factorization system on monomorphisms.
\end{remark}

\medskip

We finish this section with the description of the modality associated to the acyclic class $\Cover \cG$.
For an extended Grothendieck topology $\cG$, we define $\Close \cG := \cG^\perp \cap \Mono$.
If $j$ is the Lawvere--Tierney topology associated to $\cG$ by \cref{LT-topoGroth}, then $\Close \cG = \Close j$.
We also put $\Dense\cG := \Dense j$.
The following result identifies the modality associated to the acyclic class $\Cover \cG$ (\cref{def:cover}).

\begin{corollary}
\label{cor:ptopmodality}
For any extended Grothendieck topology $\cG$ on a topos $\cE$, 
the pair $(\Cover \cG, \Close \cG)$ is a modality on $\cE$.
\end{corollary}
\begin{proof}
Using the factorization system on monomorphisms of \cref{LTtop,rem:orthomono},  any map $f:A\to B$ can be factored uniquely as
\[
\begin{tikzcd}
A \ar[rr,"\coim f"] \ar[rrrr, "{\in \Cover\cG}"', bend right=15]
&& A' \ar[rr, "\Dense {\im f}"]
&& B' \ar[rr, "\Close  {\im f}"]
&& B
\end{tikzcd}
\]
where $\Dense {\im f}$ and $\Close  {\im f}$ are respectively the dense part and the close part of the monomorphism $\im f$.
The map $\Dense {\im f}\circ \coim f: A\to B'$ is in $\Cover\cG$ since $\Dense {\im f} \in \Dense \cG = \cG$.
This proves the existence of the $(\Cover \cG, \Close \cG)$--factorization.
The orthogonality $\Cover \cG \perp \Close \cG$ can be deduced from the orthogonality $\Surj\perp \Mono$ and the orthogonality $\Dense \cG \perp \Close \cG$ of \cref{LTtop}.
We leave the details to the reader.
The class $\Surj$ is stable by base change, the classes $\Dense \cG$ and $\Close \cG$ also by \cref{LTtop}.
This proves that the $(\Cover \cG, \Close \cG)$--factorization is stable by base change, hence a modality.
\end{proof}

\subsection{Topologies and 1-topoi}

The definitions of extended Grothendieck topologies, covering topologies, and Lawvere--Tierney topologies do not use the full strength of the topos axioms and make sense in a 1-topos.
We claim that the equivalence results of \cref{Gtop=Gtop,thm:covering-top,LT-topoGroth} are still true for 1-topoi and we shall use this fact implicitly in what follows.

Recall that if $\cE$ is a topos, then the category $\cE\truncated 0$ of 0-truncated objects in $\cE$ is a 1-topos \cite[Theorem 6.4.1.5]{Lurie:HTT}.
The main result of this short section is to show that a topology on a topos $\cE$ is equivalent to a topology on the associated 1-topos $\cE\truncated 0$.

\begin{lemma}
\label{lem:omega-is-omega}
The subobject classifier of the topos $\cE$ belongs to $\cE\truncated 0$, 
where it is the subobject classifier of the 1-topos $\cE\truncated 0$.
\end{lemma}
\begin{proof}
By definition, the functor of points of $\Omega$ takes values in sets.
This proves the first statement.
To prove the second statement we need to verify that any subobject of a 0-truncated object is 0-truncated.
Let $X$ be a 0-truncated object and $Y\to X$ a subobject.
The monomorphisms are the $(-1)$-truncated maps, and therefore they are 0-truncated.
The composite map $Y\to X\to 1$ is therefore 0-truncated.
\end{proof}

\begin{proposition}
\label{prop:bijLTLT} 
The poset of Lawvere--Tierney topologies on a topos $\cE$ is isomorphic to the poset of Lawvere--Tierney topologies of the 1-topos $\cE\truncated 0$.
Consequently, it has the structure of a frame.
\end{proposition}
\begin{proof}
We saw in \cref{lem:omega-is-omega} that the same object $\Omega$ is the subobject classifier  in both contexts.
Thus, the poset of closure operators on the Lawvere object $\Omega$ of $\cE$ 
is the same as the poset of closure operators on the Lawvere object 
$\Omega$ of the 1-topos $\cE\truncated 0$.
The statement about the frame structure is a classical result about 1-topoi \cite[Example 4.5.14 (f)]{Johnstone:Elephant}.
\end{proof} 

The following statement is a direct consequence of the bijections between Lawvere--Tierney topologies and extended Grothendieck topologies, both for topoi and 1-topoi.

\begin{corollary}
\label{cor:bijGTGT} 
The poset of extended Grothendieck topologies on a topos $\cE$ is isomorphic to the poset of extended Grothendieck topologies of the 1-topos $\cE\truncated 0$.
\end{corollary}

\begin{remark}
This bijection can be described explicitly as follows.
Given an extended Grothendieck topology $\cG$ on $\cE$, the intersection $\cG\cap \cE\truncated 0$ is an extended Grothendieck topology on $\cE\truncated 0$.
Conversely, given an extended Grothendieck topology $\cG_0$ on $\cE\truncated 0$, it generates an extended Grothendieck topology $\cG_0\gtop$ on $\cE$.
One can show that $\cG_0\gtop$ is simply the closure under base change $\cG_0\bc$ of $\cG_0$ in $\cE$.

A similar bijection exists for covering topologies.
Given a covering topology $\cC$ on $\cE$, the intersection $\cC\cap \cE\truncated 0$ is a covering topology on $\cE\truncated 0$.
Conversely, given any covering topology $\cC_0$ on $\cE\truncated 0$ generates a covering topology $\Cover{\cC_0}$ on $\cE$.
\end{remark}

If $ C$ is a category, we shall denote by $\ho C$ its homotopy 1-category.
The canonical functor $h: C\to \ho C$ reflects the category $ C$ into the category of 1-categories.
Hence the functor
$h^\star:\fun {\ho C\op} \Set \to \fun { C\op} \Set$
is an equivalence of categories.
A presheaf $F: C\op\to \cS$ is 0-truncated in the topos $\mathrm{PSh}( C)$ if and only the functor $F$ takes its values on the category of sets $Set=\cS\truncated 0$. 
Thus, 
\[
{\mathrm{PSh}( C)}\truncated 0
\ =\ 
\fun { C\op} \Set
\ =\ 
\fun {\ho C\op} \Set\,.
\]

\Cref{cor:bijGTGT} can be seen as generalization of the fact that an extended Grothendieck topology on a category $ C$ is equivalent to an extended Grothendieck topology on the 1-category $\ho C$ \cite[Remark 6.2.2.3]{Lurie:HTT}.
The connection is made more precise by the following result.

\begin{corollary}
\label{cor:TVGtop}
If $ C$ is a small category, then there exists bijections between
the poset of extended Grothendieck topologies on the topos ${\mathrm{PSh}( C)}$,
the poset of extended Grothendieck topologies on the 1-topos $[\ho C\op,\Set]$, 
and the poset of Grothendieck topologies on the 1-category $\ho C$.
\end{corollary}

\begin{proof}
\Cref{cor:bijGTGT} provide the first bijection, since $\PSh C\truncated 0=\fun {\ho C\op} \Set$.
The second one is given by the analogue of \cref{Gtop=Gtop} for 1-topoi.
\end{proof}

\section{Topological and cotopological congruences}
\label{sec:top-cotop-loc}

\subsection{Topological congruences}

This section is devoted to the study of topological localizations and topological congruences.
Our first result is to establish an equivalence between extended Grothendieck topologies, topological congruences, and topological localizations on any topos $\cE$ (\cref{thm:equivTcongGtop}).
We use this equivalence to define the topological part $\cK\topo$ of a congruence $\cK$ in \cref{def:topopart}.
We explain how to characterize it in terms of generators of $\cK$ in \cref{prop:topological-part}.
And we provide a universal property for the corresponding topological localization in \cref{thm:topopart,thm:meaningtopcotopfacto}.

\medskip

\begin{proposition}
\label{thm:adjGtopCong}
The morphism $(-)\ac:\GTop(\cE)\to \Acyclic(\cE)$ takes values in congruences $\Cong(\cE)\subseteq \Acyclic(\cE)$ and the adjunction of \cref{prop:adjGtopAc} restricts to an adjunction
\begin{equation}
\label{adj:cong:gtop}
\begin{tikzcd}
\Cong(\cE)
\ar[rr,shift right = 1.6, "\im-"']
\ar[from=rr,shift right = 1.6, "(-)\ac"', hook']
&&\GTop(\cE)\ .
\end{tikzcd}
\end{equation}
\end{proposition}
\begin{proof}
By \cref{congruencegenerated17}, we have $\cG\ac=\cG\cong$.
The rest follows from \cref{prop:adjGtopAc}.
\end{proof}

\begin{definition}
\label{deftopcong}
We say that a congruence $\cK$ in a topos $\cE$ is {\it topological} if it is generated by a class of monomorphisms $\Sigma\subseteq \cE$ ($\cK=\Sigma\cong$).
By \cref{congruencegenerated17}, a congruence is topological if and only if it is monogenic as an acyclic class (\cref{def:monogenic}).
We denote $\TCong(\cE)$ the poset of topological congruences.

We shall say that a left-exact localization $\phi:\cE\to \cE'$  is a {\it topological localization} if the congruence $\cK_\phi$ is topological.
We denote by $\TLoclex(\cE)$ the poset of topological localizations.
\end{definition}

\begin{examples}
We give some examples of topological congruences.
Examples of topological localizations will be given below.
\label{ex:topcong}
\begin{examplenum}
\item\label{ex:topcong:1}
    When $\cG$ is an extended Grothendieck topology, the congruence $\cG\ac=\cG\cong$ is topological.

\item\label{ex:topcong:2}
    The class of isomorphisms $\Iso$ and the class of all maps $\All$ in a topos $\cE$ are topological congruences. 
    The former is generated by the empty class of maps, while the later by the map $\emptyset \to 1$.
    These examples are associated to the minimal and maximal extended Grothendieck topologies of \cref{ex:GTop:0}, by the construction of \cref{ex:topcong:1}.

\end{examplenum}

\end{examples}

The following proposition is \cite[Proposition 6.2.1.5]{Lurie:HTT} but we provide a proof taking advantage of univalent monomorphisms.
Recall the subposet $\Cong_\textsf{sg}(\cE)\subseteq \Cong(\cE)$ of congruences of small generation from \cref{thm:bij-congruence-lexloc}.

\begin{proposition}
\label{topcongaresmall} 
Every topological congruence is generated by a univalent monomorphism $v:T\to V$.
In particular it is of small generation and $\TCong(\cE)\subseteq \Cong_\mathsf{sg}(\cE)$.
\end{proposition}
\begin{proof} 
A topological congruence $\cK$ is generated by its intersection $\cK\cap \Mono$.
But the intersection $\cK\cap \Mono$ is an extended Grothendieck topology by  \cref{fromacyclictoGroth}.
Hence we have $\cK\cap \Mono=\{v\}\bc$ for a univalent monomorphism $v:T\to V$ by \cref{Gtoparesmall}.
It follows that $\cK=(\{v\}\bc)\cong=\{v\}\cong$.
\end{proof}

\begin{corollary}$\quad$
\label{cor:topcongexists}
\begin{enumerate}
\item\label{cor:topcongexists:1} The localization of a topos $\cE$ associated to a topological congruence $\cK$ always exists and
\[
\cE\Forcing\cK\Iso
\ =\ 
\Loc\cE \cK\,.
\]
\item\label{cor:topcongexists:2} For any class $\Sigma$, the forcing condition $\Forcing\Sigma\Surj=\Forcing{\Sigma\ac}\Surj$ is representable. 
\item\label{cor:topcongexists:3} For any class $\Sigma$ and any $-2\leq n\leq \infty$, the forcing condition $\Forcing\Sigma{\Conn n}=\Forcing{\Sigma\ac}{\Conn n}$ is representable.
\end{enumerate}
\end{corollary}
\begin{proof}
\eqref{cor:topcongexists:1}
Every topological congruence is of small generation by \cref{topcongaresmall}.
Then the result follows from \cref{Luriethm1}.

\smallskip
\noindent\eqref{cor:topcongexists:2}
By \cref{thm:forcing} we have the equivalence of forcing conditions 
\[
\Forcing{\Sigma\ac}\Surj
\ =\ 
\Forcing\Sigma\Surj
\ =\ 
\Forcing{\im\Sigma}\Iso
\ =\ 
\Forcing{\im\Sigma\ac}\Iso\ .
\]
Then the result follows from \eqref{cor:topcongexists:1} applied to the topological congruence $(\im\Sigma)\ac$.

\smallskip
\noindent\eqref{cor:topcongexists:3}
Proof similar to \eqref{cor:topcongexists:2}.
\end{proof}

\begin{examples}
\label{ex:toploc}
We give some examples of topological localizations.
More will be given below.

Recall the algebraically free topos $\S X = \fun \Fin \cS$ from \cref{def:freetopos}.
\begin{examplenum}
\item\label{ex:toploc:trivial}
The trivial algebraic morphisms $\cE\xto {id} \cE$ and $\cE\to 1$ are the topological localizations corresponding the topological congruences $\Iso$ and $\All$ of \cref{ex:topcong:2}.

\item\label{ex:toploc:surj}
    An object $X$ is a topos is called {\it inhabited} if $X\to 1$ is a surjection (i.e $(-1)$-connected).
    The topos $\S {X\connected {-1}}$ (denoted $\S {X^\circ}$ in \cite{Anel-Joyal:topo-logie}) freely generated by an inhabited object $X\connected {-1}$ is a topological localization of the algebraically free topos $\S X$, since 
    \[
    \S {X\connected {-1}}\ :=\ \S X\Forcing{X\to 1}\Surj\,.
    \]
    We proved in \cite[Section 5.4]{ABFJ:HS} that $\S {X\connected {-1}} = \fun {\Fin\connected {-1}} \cS$ where $\Fin\connected {-1}$ is the category of nonempty finite spaces.

\item\label{ex:toploc:nconn}
    The topos $\S {X\connected n}$ freely generated by a $n$-connected object $X\connected n$ is a topological localization of the algebraically free topos $\S X$, since 
    \[
    \S{X\connected n}\ :=\ \S X\Forcing{\Delta^{\leq n+1}(X\to 1)}\Surj\,.
    \]
    We proved in \cite[Section 5.4]{ABFJ:HS} that $\S {X\connected n} = \fun {\Fin\connected n} \cS$ where $\Fin\connected n$ is the category of $n$-connected finite spaces.

\item\label{ex:toploc:ooconn}
    The topos $\S {X\connected \infty}$ freely generated by a \oo connected object $X\connected \infty$ is a topological localization of the algebraically free topos $\S X$, since 
    \[
    \S{X\connected \infty}\ :=\ \S X\Forcing{(X\to 1)\diag}\Surj
    \]
    where $(X\to 1)\diag=\{\Delta^nX\,|\, n\geq 0\}$.
    Contrary to the case where $n<\infty$, the topos $\S{X\connected \infty}$ is not a presheaf topos.

    We shall see in \cref{ex:toppart:ooconn} that $\S X \to \S{X\connected \infty}$ is the topological part of the left-exact localization $\S X \to \cS$ given by the evaluation at $1\in \Fin$ and that the corresponding extended Grothendieck topology is that of \cref{ex:GTop:2}.

\end{examplenum}
\end{examples}

\medskip

We can now prove the result that justifies the name of covering topologies.
\begin{corollary}
\label{cor:cov=invsurj}
A class of maps $\cA$ in a topos $\cE$ is a covering topology if and only if it is the inverse image of the class $\Surj$ by some algebraic morphism of topoi.
\end{corollary}
\begin{proof}
If $\cA = \phi^{-1}(\Surj)$ for some algebraic morphism of topoi $\cE\to \cF$, then 
$\cA$ is acyclic by \cref{prop:transport-acyclic}, and contains surjections since, we have always $\phi(\Surj)\subseteq \Surj$. 
Hence it is always a covering topology.

Reciprocally, let $\cC$ be a covering topology of a topos $\cE$.
By \cref{cor:topcongexists}~\eqref{cor:topcongexists:2}, the forcing condition
$\Forcing\cC\Surj$ is representable by $\phi:\cE\to \cE\Forcing{\im\cC}\Iso$
where $\im \cC = \cC\cap \Mono$ is the extended Grothendieck topology associated to $\cC$ by \cref{fromacyclictoGroth}.
The congruence associated to $\phi$ is $\im \cC\ac$, and by \cref{cor:GacMono=G}, we know that $\im \cC\ac\cap \Mono = \im \cC$.
Let $f$ be a map in $\cE$. 
The maps $\phi(f)$ is surjective if and only if $\phi(\im f)$ is invertible, if and only if $\im f\in \im \cC$, if and only if $f\in \cC$ (by \cref{lem:image-acyclic}).
\end{proof}

\begin{corollary}
\label{cor:TLoc=TCong}
The map $\phi\mapsto \cK_\phi$ induces an isomorphism of posets
\[
\begin{tikzcd}
\TLoclex(\cE)
\ar[rrr,shift left = 1.6, "\phi\mapsto \cK_\phi"]
\ar[from=rrr,shift left = 1.6, "{\cK\mapsto \cE\Forcing\cK\Iso}", "\simeq"']
&&&\TCong(\cE)\ .
\end{tikzcd}
\]
\end{corollary}
\begin{proof}
This is a restriction of the isomorphisms of \cref{thm:bij-congruence-lexloc}.
The map $\phi\mapsto \cK_\phi$ sends topological localizations to topological congruences by definition of topological localizations.
Using \cref{topcongaresmall}, we know that $\TCong(\cE)\subseteq \mathsf{Cong_{sg}}(\cE)$.
Then the inverse map $\cK\mapsto \cE\Forcing\cK\Iso$ also restricts to $\TCong(\cE)$ and $\TLoclex(\cE)$.
\end{proof}

The following result is our version of \cite[Proposition 6.2.2.17]{Lurie:HTT}.

\begin{theorem}[Generalized Lurie bijection]
\label{generalL}
\label{thm:equivTcongGtop}
\label{thmmonogenicmod17} 
The adjunction \eqref{adj:cong:gtop} restricts into an isomorphism of posets
\[
\begin{tikzcd}
\TCong(\cE)
\ar[rr,shift right = 1.6, "-\cap \Mono"']
\ar[from=rr,shift right = 1.6, "(-)\ac"', "\simeq"]
&&\GTop(\cE)\ .
\end{tikzcd}
\]
In particular, the poset $\TCong(\cE)$ is small.
\end{theorem}
\begin{proof}
It is enough to prove that the topological congruences coincide with the image of the morphism $(-)\ac$.
By definition $\cG\ac$ is a topological congruence.
Conversely, let $\Sigma$ be a class of monomorphisms, we want to show that $\Sigma\ac$ is generated by an extended Grothendieck topology.
We consider $\Sigma\gtop$, the extended Grothendieck topology generated by $\Sigma$.
Recall from \cref{cor:freeGtop} that $\Sigma\gtop = \Sigma\ac\cap \Mono$.
Thus, we have inclusions $\Sigma\subseteq \Sigma\gtop \subseteq \Sigma\ac$ and therefore $(\Sigma\gtop)\ac = \Sigma\ac$.
Finally, the smallness assertion is a consequence of that of $\GTop(\cE)$ (see \cref{cor:GTop-is-small}).
\end{proof}

We can now prove the analogue of \cref{cor:cov=invsurj} for extended Grothendieck topologies mentioned in \cref{ex:GTop:1}.

\begin{corollary}
\label{cor:GTop=invmono}
A class of maps $\cG$ in a topos $\cE$ is an extended Grothendieck topology if and only if $\cG = \phi^{-1}(\Iso)\cap \Mono$ for some algebraic morphism of topoi $\phi:\cE\to \cF$.
\end{corollary}
\begin{proof}
Direct from the composition of the isomorphisms of \cref{cor:TLoc=TCong} and \cref{thm:equivTcongGtop}.
\end{proof}

The following result is a convenient reformulation of \cref{thm:adjGtopCong} using \cref{thm:equivTcongGtop}.
\begin{corollary}
\label{topcore}
\label{criterionfortopcong}
The map $\cK\mapsto(\cK\cap \Mono)\ac$ defines a right adjoint to the inclusion of posets $\TCong(\cE) \to \Cong(\cE)$.
\[
\begin{tikzcd}
\Cong(\cE)
\ar[rr,shift right = 1.6, "(-\cap \Mono)\ac"']
\ar[from=rr,"j"',shift right = 1.6, hook']
&&\TCong(\cE)
\end{tikzcd}
\]
\end{corollary}

Recall that for an acyclic class $\cA$ we called $\im\cA\ac = (\cA\cap \Mono)\ac$ the monogenic part of $\cA$.
When $\cK$ is a congruence, we shall keep the usual terminology (see also \cref{def:cotopcong}).

\begin{definition}[Topological part]
\label{def:topopart}
For a congruence $\cK$, we shall say that the topological congruence $\cK\topo:=(\cK\cap \Mono)\ac = \im\cK \ac$ is the {\it topological part} of $\cK$.
When the localization $\phi:\cE\to \cE\Forcing\cK\Iso$ exists, we shall say that the topological localization $\phi\topo:\cE\to \cE\Forcing{\cK\topo}\Iso$ is the topological part of $\phi$.

\end{definition}

By \cref{topcore}, we have always $\cK\topo \subseteq\cK$ and $\cK\topo$ is the largest topological congruence within $\cK$.
We have also $(\cK\topo)\topo = \cK\topo$.

\begin{lemma}
\label{lem:topofix}
A congruence $\cK$ is topological if and only if it $\cK\topo = \cK$.
\end{lemma}
\begin{proof}
If $\cK$ is topological, there exists a class $\Sigma\subseteq \cK\cap \Mono$ such that $\Sigma\ac=\cK$.
Using that $\Sigma\ac\subseteq (\cK\cap \Mono)\ac \subseteq \cK$, this proves $\cK=\cK\topo$.
Conversely, if $\cK=\cK\topo$ it is generated by $\cK\cap \Mono$ hence topological.
\end{proof}

The following result provides generators for the topological part of a congruence $\cK$ in terms of generators of $\cK$.
This is quite useful in practice.

\begin{proposition}[Computation of monogenic/topological parts]
\label{prop:topological-part}
Let $\Sigma$ be a class of maps in a topos $\cE$.
The following formulas hold:
\begin{enumerate}
\item\label{prop:topological-part:1} $\im{\Sigma\ac}\ac = {\im \Sigma}\ac\,$;
\item\label{prop:topological-part:2} $(\Sigma\cong)\topo = {\im{\Sigma\diag}}\ac\,$.
\end{enumerate}
\end{proposition}
\begin{proof}
\eqref{prop:topological-part:1}
By \cref{thm:forcing} we have the following equivalences of forcing conditions
\[
\Forcing{\im{\Sigma\ac}\ac}\Iso
\ =\ 
\Forcing{\im{\Sigma\ac}}\Iso
\ =\ 
\Forcing{\Sigma\ac}\Surj
\ =\ 
\Forcing{\Sigma}\Surj
\ =\ 
\Forcing{\im\Sigma}\Iso
\ =\ 
\Forcing{\im\Sigma\ac}\Iso\ .
\]
By \cref{cor:topcongexists}, the localization
$\rho:\cE \to \cE\Forcing{\im\Sigma\ac}\Iso = \cE\Forcing{\im{\Sigma\ac}\ac}\Iso$
exists.
Since $\im{\Sigma\ac}\ac$ and $\im\Sigma\ac$ are congruences by \cref{congruencegenerated17}, they are exactly the class of maps inverted by $\rho$ (\cref{Luriethmloc}), and therefore equal.

\smallskip
\noindent\eqref{prop:topological-part:2}
We have $\Sigma\cong = (\Sigma\diag)\ac$ by \cref{congruencegenerated17}.
Thus, $(\Sigma\cong)\topo = \im{\Sigma\cong}\ac =  {\im{\Sigma\diag}}\ac$ by \eqref{prop:topological-part:1}.
\end{proof}

\begin{remark}
There is no analogue of the formula $\im{\Sigma\ac}\ac = {\im \Sigma}\ac$ for the epigenic part of an acyclic class $\Sigma\ac$.
We have always an inclusion $\coim \Sigma \ac \subseteq \coim{\Sigma\ac}$ but it can be strict.
For example, if $\Sigma$ is a class of monomorphisms, then we have $\coim\Sigma\ac = \Iso$ .
And if not all maps in $\Sigma$ are isomorphisms we know that $\Sigma\ac\cap \Surj \not= \Iso$ from \cref{lem:nogo}.
However, it is still true that $\Sigma\ac = \im\Sigma\ac \vee \coim\Sigma\ac$ in the poset of acyclic classes.
\end{remark}

If $\cK$ is an arbitrary congruence (not necessarily of small generation) its topological part $\cK\topo$ is always a congruence of small generation.
Therefore the localization $\cE\Forcing{\cK\topo}\Iso$ always exists.
The following theorem gives a meaning to this localization.

\begin{proposition}
\label{thm:topopart}
For any congruence, we have an equivalence of forcing conditions
\[
\Forcing{\cK\topo}\Iso
\ =\ 
\Forcing{\cK}{\Conn\infty}\ .
\]
\end{proposition}
\begin{proof}
\begin{align*}
\Forcing{\cK\topo}\Iso
& = \Forcing{(\cK\cap \Mono)\ac}\Iso    \\
& = \Forcing{\cK\cap \Mono}\Iso    &&\text{because $\Iso$ is acyclic}\\
& = \Forcing{\im\cK}\Iso    &&\text{$\cK\cap \Mono = \im \cK$ by \cref{lem:imagepretopo}}\\
& = \Forcing{\im{\cK\diag}}\Iso    &&\text{$\cK\diag = \cK$ by \cref{lem:caraccong}}\\
& = \Forcing{\cK\diag}{\Surj}    &&\text{by \cref{thm:forcing}}\\
& = \Forcing{\cK}{\Conn\infty}    &&\text{by \cref{thm:forcing}\,.}
\qedhere
\end{align*}
\end{proof}

\begin{theorem}[Interpretation of topological localizations]
\label{topopartloc}
\label{thm:meaningtopcotopfacto}
Let $\Sigma$ be a set of maps in a topos $\cE$.
Then the topological part of the localization 
$\cE\to \cE\Forcing\Sigma\Iso$
is the localization $\cE\to \cE\Forcing\Sigma{\Conn\infty}$
universally forcing the maps in $\Sigma$ to be \oo connected.
\end{theorem}
\begin{proof}
We use \cref{thm:topopart} with $\cK = \Sigma\cong$.
The topological part $\cE\Forcing{\cK\topo}\Iso$ of the localization $\cE\to \cE\Forcing\cK\Iso$ exists by \cref{cor:topcongexists}~\eqref{cor:topcongexists:1}.
Then by \cref{thm:topopart}, we get
\[
\Forcing{\cK\topo}\Iso
\ =\ 
\Forcing\cK{\Conn\infty}
\ =\ 
\Forcing{\Sigma\cong}{\Conn\infty}
\ =\ 
\Forcing\Sigma{\Conn\infty}
\]
where the last equality is from \cref{thm:forcing}.
\end{proof}

The following result provides a nice characterization of topological localizations.
\begin{corollary}
\label{cor:carctoploc}
A localization $\phi:\cE\to \cF$ is topological if and only it can be presented as forcing a class of maps to be \oo connected.
\end{corollary}
\begin{proof}
If $\phi:\cE\to \cF$ is a topological localization, then $\cK_\phi = \cK_\phi\topo$ by \cref{lem:topofix}.
Then, by \cref{thm:topopart}, we have $\cF = \cE\Forcing{\cK_\phi\topo}\Iso = \cE\Forcing{\cK_\phi}{\Conn\infty}$.
This proves the condition is necessary.
Conversely, any forcing $\phi:\cE\to \cE\Forcing\cK{\Conn\infty}$ is a topological localization by \cref{thm:forcing}~\eqref{thm:forcing:4}.
\end{proof}

\begin{examples}
\label{ex:toppart}
We give examples of computations of topological parts.
Recall the algebraically free topos $\S X$ from \cref{def:freetopos}.
\begin{examplenum}
\item\label{ex:toppart:ooconn}
The evaluation at $1\in \Fin$ defines a left-exact localization of topoi $ev_1:\S X= \fun \Fin \cS \to \cS$.
Since the functor represented by $1\in \Fin$ is the canonical inclusion $X:\Fin \to \Set$, the congruence $\cK$ associated to the localization is generated by the map $X\to 1$ ($\cK = \{X\to 1\}\cong$)  \cite[Section~5.2]{ABFJ:HS}.
By \cref{topopartloc}, the topological part of $\S X \to \cS$ is the localization $\S X\to \S{X\connected \infty}$ forcing the universal object $X$ to be \oo connected.

Let us see this a bit more explicity using \cref{prop:topological-part}~\eqref{prop:topological-part:2}. 
We get some concrete generators for the topological part of $\cK$:
\[
(\{X\to 1\}\cong)\topo
\ =\ 
\big\{\im{\Delta^nX}\,|\, n\geq 0\big\}\ac \,.
\]
This recovers the presentation of \cref{ex:toploc:ooconn}. 
This congruence is forcing all diagonals of $X$ to be surjective, and thus $X$ to be \oo connected.

Using \cref{thm:equivTcongGtop} the extended Grothendieck topology corresponding to this topological localization is the one of \cref{ex:GTop:2}.

\item\label{ex:toppart:ooconn2}
    We provide another presentation of the topological part of $ev_1:\S X \to \cS$ of \cref{ex:toppart:ooconn}.
    Let $X^K:\Fin\to \cS$ be the functor represented by $K\in \Fin$, and let $\Sigma$ be the set of all maps $X^K\to X^J$ between representable functors in $\S X = \fun \Fin \cS$.
    All representable functors are sent to 1 by $ev_1$.
    This shows that $\Sigma$ is another generating set for the left-exact localization $ev_1:\S X \to \cS$.
    Then using \cref{prop:topological-part}~\eqref{prop:topological-part:2}, we get that the topological part is generated by the class $\im \Sigma$ of all images of the maps $X^K\to X^J$ between representable functors.
    Equivalently, this shows that the topological part forces universally all maps $X^K\to X^J$ to be surjective.

\item\label{ex:toppart:ntrunc}
We consider now the localization $\S X\to \S{X\truncated n}$ forcing the universal $X$ to be $n$-truncated.
We proved in \cite[Section~5.3]{ABFJ:HS} that this localization is generated 
either by the map $\Delta^{n+2}:X\to X^{S^{n+1}}$, or by the map $X\to P_n(X)$ where $P_n(X)$ is the $n$-truncation of the object $X$ in $\S X$.
By \cref{thm:topopart}, the topological part of $\S X\to \S{X\truncated n}$ is the localization of $\S X$ forcing the map $\Delta^{n+2}:X\to X^{S^{n+1}}$ or $X\to P_n(X)$ to be \oo connected.
In this localization, the image $X'$ of $X$ is not $n$-truncated but its Postnikov tower is nonetheless constant after the stage $n$ ($P_{n+k}(X') = P_n(X')$).

\end{examplenum}
\end{examples}

\medskip

We now turn to the characterization of other generators for topological congruences than monomorphisms.

\begin{definition}
We shall say that a map $f$ in a topos $\cE$ is {\it topological} if the congruence $\{f\}\cong$ generated by $f$ is topological.
\end{definition}

\begin{proposition}
\label{conggentopmap} 
If a congruence $\cK$ is generated by topological maps, then $\cK$ is topological.
\end{proposition}

\begin{proof}
Let $\cK$ be a congruence generated by topological maps in a topos $\cE$.
By hypothesis, we have $\cK=\Sigma\cong$, for a class of topological maps $\Sigma\subseteq \cE$.
For every $f\in \Sigma$ there exists a class of monomorphisms $\cM\subseteq \cE$ such that $\{f\}\cong =\cM(f)\cong$, since the map $f$ is topological.
Let us put $\cM:=\bigcup_{f\in \Sigma} \cM(f)$.
Then we have $\Sigma\cong =\cM\cong$, since we have $\{f\}\cong =\cM(f)\cong$ for every $f\in \Sigma$.
This shows $\cK=\cM\cong$ and hence that $\cK$ is a topological congruence.
\end{proof}

Recall that a map $u:A\to B$ in a topos is said to be {\it truncated} if it is $n$-truncated for some $n\geq -1$.

\begin{lemma}
\label{truncatedaretop}
Any truncated map in a topos is a topological map.
\end{lemma} 
\begin{proof} 
A $n$-truncated map $u:A\to B$ in a topos $\cE$ is invertible if and only it is $n$-connected
if and only if the diagonal $\Delta^k(u)$ is surjective for every $0\leq k\leq n+1$. 
This provides an equivalence of forcing conditions $\Forcing u\Iso =\Forcing u{\Conn n}$.
By \cref{thm:forcing}, we have also an equivalence $\Forcing u{\Conn n} = \Forcing\Sigma\Iso$
where $\Sigma=\{\im{\Delta^k(u)}\ | \ 0\leq k\leq n+1\}$.
Since $\Sigma$ is a class of monomorphisms, this proves that the congruence $\{u\}\cong$ is topological.
\end{proof}

\begin{proposition}
\label{congtruncated}
A congruence is topological if and only if it is generated by a class of truncated maps.    
\end{proposition}
\begin{proof}
By definition, a topological congruence is generated by monomorphism, that is by $(-1)$-truncated maps.
The converse is given by \cref{truncatedaretop}.
\end{proof}

\begin{remark}
Not all topological maps are truncated since any coproduct of topological maps is topological and not all coproducts of truncated maps are truncated.
We do not know if a colimit of topological maps is a topological map, nor if every map in a topological congruence is a topological map.
\end{remark}

\subsection{Cotopological congruences}
\label{sec:cotopcong}

In this section, we introduce the notion of a cotopological congruence (\cref{def:cotopcong}) and provide a number of characterizations (\cref{prop:caractopcong}).

\bigskip
Recall from \cref{def:epigenic} that an acyclic class is epigenic if it is generated by a class of surjections.

\begin{definition}
\label{def:cotopcong}
We shall say that a congruence is {\it cotopological} if it is epigenic as an acyclic class.
\end{definition}

Recall also from \cref{def:monogenic} that any acyclic class $\cA$ has a monogenic part $\im\cA\ac = (\cA\cap \Mono)\ac$ and an epigenic part $\coim\cA=\cA\cap \Surj$.
And recall from \cref{def:topopart} that we defined the topological part of a congruence $\cK$ to be $\cK\topo = \im\cK\ac$.
The following result is \cref{lem:epigenic}.

\begin{proposition}[Characterization of cotopological congruences]
\label{prop:caractopcong}
The following conditions on a congruence $\cK$ are equivalent:
\begin{enumerate}
\item\label{prop:caractopcong:1} $\cK$ is cotopological;
\item\label{prop:caractopcong:2} $\cK = \coim\cK = \cK\cap\Surj$;
\item\label{prop:caractopcong:3} $\cK\cap \Mono = \Iso$;
\item\label{prop:caractopcong:4} $\cK\topo = \Iso$;
\item\label{prop:caractopcong:5} $\cK \subseteq \Surj$;
\item\label{prop:caractopcong:6} $\cK \subseteq \Conn\infty$.
\end{enumerate}
\end{proposition}

\begin{remark}
\label{rem:maxcotopcong}
\Cref{prop:caractopcong} shows that a congruence $\cK$ is cotopological if and only if it is contained in \oo connected maps.
We have seen in \cref{exmpcongruence3} that the class $\Conn\infty$ is a congruence, hence it is the maximal cotopological congruence.
The poset of cotopological congruences is then the slice poset $\Cong(\cE)/\Conn\infty$.
\end{remark}

\begin{remark}
The reader interested in Homotopy Type Theory can compare \cref{prop:caractopcong} with 
\cite[Theorem~6.5]{Christensen-Rijke} and \cite[Theorem 3.22]{RSS}.
\end{remark}

\smallskip

\begin{proposition}[Topological--cotopological decomposition]
\label{prop:gen-cotop}
Let $\phi:\cE\to \cE\Forcing{\cK\topo}\Iso$ be the topological localization associated to $\cK$.
Then the class $\phi(\cK)$ is a cotopological congruence on $\cE\Forcing{\cK\topo}\Iso$ and we have
\[
\phi(\cK\cap\Surj)
\ =\ 
\phi(\cK)
\ =\ 
\phi(\cK)\cap \Surj\,.
\]
\end{proposition}
\begin{proof}
The class $\phi(\cK)$ is a congruence by \cref{prop:transport-cong}.
By the equivalence of forcing conditions $\Forcing{\cK\topo}\Iso = \Forcing{\cK}{\Conn\infty}$ of \cref{thm:topopart} we know that $\phi(\cK)\subseteq\Conn\infty$.
This proves that $\phi(\cK)$ is cotopological by \cref{prop:caractopcong}.
The equality $\phi(\cK) = \phi(\cK)\cap\Surj$ follows from \cref{prop:caractopcong}.
We are left to show that $\phi(\cK\cap\Surj) = \phi(\cK)$.
We have always $\phi(\cK\cap\Surj) \subseteq \phi(\cK)$.
Conversely, if $f$ is a map in $\cK$ then we have 
$\phi(f) = \phi(\im f) \circ \phi(\coim f) = \phi(\coim f)$ since $\phi(\im f)$ is invertible by definition of $\phi$.
This proves that the inclusion $\phi(\cK\cap\Surj) \subseteq \phi(\cK)$ is surjective, hence $\phi(\cK\cap\Surj) = \phi(\cK)$.
\end{proof}

\subsection{Cotopological morphisms}
\label{sec:cotopmor}

In this section, we introduce the notion of a cotopological morphism of topoi (\cref{cotopmorphism}), generalizing the notion of cotopological localization introduced in \cite[Definition 6.5.2.17]{Lurie:HTT}.
We characterize them as the morphisms reflecting \oo connected objets in \cref{morphismvslemmacotop}.
Then, we prove that together with the class of topological localization, they define a factorization system on the category $\Toposalg$ (\cref{facttopcotop}).

\begin{definition}[Cotopological morphism]
\label{cotopmorphism} 
We shall say that an algebraic morphism of topoi $\phi:\cE\to \cF$ is {\it cotopological}
if its congruence $\cK_\phi$ is cotopological.
When $\phi$ is a left-exact localization which is cotopological, we shall say that $\phi$ is a {\it cotopological localization}.
\end{definition}

For any $-1\leq n\leq \infty$, we shall say that an algebraic morphism of topoi $\phi:\cE\to \cF$ {\it reflects $n$-connected maps} if, for a map $f$ in $\cE$, 
$f$ is $n$-connected if and only if $\phi(f)$ is $n$-connected in $\cF$.
Since $f\in \Conn n\Rightarrow \phi(f)\in \Conn n$ is always true, $\phi$ reflects $n$-connected maps if and only if $\phi^{-1}(\Conn n(\cF))\subseteq \Conn n(\cE)$.

\begin{proposition}[Characterization of cotopological morphisms]
\label{morphismvslemmacotop}
Let $\phi:\cE\to \cF$ be an algebraic morphism of topoi.
The following condition are equivalent:
\begin{enumerate}
\item\label{morphismvslemmacotop:0} $\phi$ is cotopological;
\item\label{morphismvslemmacotop:1} $\phi$ inverts no monomorphism (i.e. $\cK_\phi\cap \Mono=\Iso$);
\item\label{morphismvslemmacotop:2} $\phi$ reflects surjective maps;
\item\label{morphismvslemmacotop:3} $\phi$ reflects $n$-connected maps for all $-1\leq n\leq \infty$;
\item\label{morphismvslemmacotop:4} $\phi$ reflects \oo connected maps.
\end{enumerate}
\end{proposition}
\begin{proof}
\eqref{morphismvslemmacotop:0}$\Leftrightarrow$\eqref{morphismvslemmacotop:1}
This is \cref{prop:caractopcong}~\eqref{prop:caractopcong:3} applied to $\cK_\phi$.

\smallskip
\noindent\eqref{morphismvslemmacotop:1}$\Rightarrow$\eqref{morphismvslemmacotop:2}
For $f\in \cE$ such that $\phi(f)$ is surjective, let us show that $f$ is surjective.
The functor $\phi$ preserves the factorization $f={\im f}{\coim f}$ and hence $\phi({\im f})=\im {\phi(f)}$.
But the map $\im {\phi(f)}$ is invertible, since $\phi(f)$ is surjective. 
Thus, $\im f\in \cK_\phi\cap\Mono$ and hence $\im f$ is invertible, since $\cK_\phi $ is cotopological.
This shows that $f$ is surjective.

\smallskip
\noindent\eqref{morphismvslemmacotop:2}$\Rightarrow$\eqref{morphismvslemmacotop:3}
For $f\in \cE$ such that $\phi(f)$ is $n$-connected, let us show that $f$ is $n$-connected.
The map $\phi(\Delta^k f)=\Delta^k\phi(f)$ is surjective for every  $-1\leq k\leq n+1$, since $\phi(f)$ is $n$-connected.
Hence the map $\Delta^k f$  is surjective for every $-1\leq k\leq n+1$, since the functor $\phi$ reflects surjective maps.
This shows that $f$ is $n$-connected.

\smallskip
\noindent\eqref{morphismvslemmacotop:3}$\Rightarrow$\eqref{morphismvslemmacotop:4} is obvious.

\smallskip
\noindent\eqref{morphismvslemmacotop:4}$\Rightarrow$\eqref{morphismvslemmacotop:1}
Let us show that $\cK_\phi\subseteq \Surj$.
If $f\in \cK_\phi$ then $\phi(f)$ is \oo connected, since any isomorphism is \oo connected.
Thus, $f$ is \oo connected, since the functor $\phi$ reflects \oo connected maps.
It follows that $f$ is surjective, since every \oo connected map is surjective.
This shows that $\cK_\phi\subseteq \Surj$ and hence that the congruence $\cK_\phi$ is cotopological by 
\cref{prop:caractopcong}~\eqref{prop:caractopcong:5}.
\end{proof}

\begin{remark}
\label{rem:Lurie-surjection}
A geometric morphism dual to a cotopological morphism is called a {\it surjective morphism} of topoi in \cite[Definition A.4.2.2]{Lurie:SAG} and the previous characterization is essentially \cite[Proposition A.4.2.1]{Lurie:SAG}.
Cotopological morphisms are a weakening of conservative algebraic morphisms since they reflect \oo connected maps rather than isomorphisms.
In keeping with Lurie's terminology, the geometric morphisms dual to conservative algebraic morphisms define a subclass of ``strong'' surjections of topoi.
\end{remark}

\begin{proposition}
\label{prop:caraccotoploc}
A left-exact localization $\phi:\cE\to \cF$ is cotopological if and only if the morphism $\phi^{-1}$ induces an isomorphism of posets of covering topologies
\[
\CTop(\cF)
\ =\ 
\CTop(\cE)\,.
\]
\end{proposition}
\begin{proof}
We use the notations of \cref{prop:transport-acyclic}.
For any algebraic morphism $\phi:\cE\to \cF$, we get a map $\phi^{-1}: \Acyclic(\cF)\upslice{\Surj} \to \Acyclic(\cE)\upslice{\phi^{-1}(\Surj)}$.
This map takes values in $\Acyclic(\cE)\upslice{\Surj} $ since $\Surj\subseteq \phi^{-1}(\Surj)$.
This defines a map $\phi^{-1}_{\CTop}:\CTop(\cF)\to \CTop(\cE)$.
When $\phi$ is a localization, \cref{prop:transport-acyclic} says that the map $\phi^{-1}_{\CTop}$ is injective with image $\Acyclic(\cE)\upslice{\phi^{-1}(\Surj)}\subseteq \CTop(\cE)$.
If $\phi$ is cotopological, we know by \cref{morphismvslemmacotop} that it reflects surjections, that is $\phi^{-1}(\Surj(\cF))=\Surj(\cE)$.
This proves that the image is the whole $\CTop(\cE)$ and that $\CTop(\cF) = \CTop(\cE)$.
Conversely, if the map $\phi^{-1}_{\CTop}$ is an isomorphism, then the minimal element is sent to the minimal element, hence $\phi^{-1}(\Surj(\cF))=\Surj(\cE)$ and $\phi$ is cotopological.
\end{proof}

\begin{remark}
Using the equivalence of covering topologies with extended Grothendieck topologies, topological congruences, and hypercomplete congruences (to be defined in \cref{sec:hypercoverings}), the isomorphism of \cref{prop:caraccotoploc} can be formulated in terms of these other objects.
\end{remark}

\begin{corollary}
\label{cor:cons}
Any conservative algebraic morphism of topoi is cotopological.
\end{corollary}
\begin{proof}
Let $\phi:\cE\to \cF$ be a conservative algebraic morphism.
Then, for a monomorphism $m$ in $\cE$, $\phi(m)$ is invertible if and only if $m$ is invertible.
This proves condition \eqref{morphismvslemmacotop:1} of \cref{morphismvslemmacotop}.
\end{proof}

\begin{lemma}
\label{compositecotop}
Cotopological morphisms are stable by composition.
\end{lemma}
\begin{proof}
Immediate from the conditions of \cref{morphismvslemmacotop}.
\end{proof}

\medskip

We are now going to prove that any algebraic morphism can be factored into a topological localization followed by a cotopological morphism.
For purposes of comparison, we recall first the factorization of any algebraic morphism into a localization followed by a conservative morphism.

Given any algebraic morphism of topoi $\phi:\cE\to \cF$, the congruence $\cK_\phi$ is always of small generation by \cite[Lemma 4.2.7]{ABFJ:HS}.
Therefore the localization $\phi\loc:\cE\to \cE\Forcing{\cK_\phi}\Iso$ exists by \cref{Luriethmloc} and we have a factorization
\[
\begin{tikzcd}
\cE \ar[rr,"\phi"] \ar[rd,"\phi\loc"']
&&\cF \\
& \cE\Forcing{\cK_\phi}\Iso \ar[ru,"\phi\cons"']
\end{tikzcd}
\]

\begin{lemma}
\label{lem:cons}
The functor $\phi\cons$ is conservative.
\end{lemma}
\begin{proof}
For $f\in \cE\Forcing{\cK_\phi}\Iso$ such that $\phi\cons(f)$ is invertible, let us see that $f$ is invertible.
By \cref{Luriethmloc}, the localization $\phi\loc:\cE\to \cE\Forcing{\cK_\phi}\Iso$ is reflective and therefore we can assume that $f = \phi\loc(g)$ for some $g\in \cE$.
The map $\phi(g) = \phi\cons(\phi\loc(g)) = \phi\cons(f)$ is invertible by assumption on $f$, hence $g\in \cK_\phi$.
Therefore $f=\phi(g)$ is invertible in $\cE\Forcing{\cK_\phi}\Iso$.
\end{proof}

The following proposition is folkloric but we have not been able to find a reference.

\begin{proposition}[Localization--conservative factorization]
\label{factloccons} 
Every algebraic morphism of topoi $\phi:\cE\to \cF$ admits a unique factorization in the category of $\Toposalg$
\begin{equation}
\begin{tikzcd}
\cE \ar[rr,"\phi"] \ar[rd,"\phi\loc"']&& \cF \\
& \cE' \ar[ru,"\phi\cons"']
\end{tikzcd}
\end{equation}
where $\phi\loc$ is a left-exact localization and $\phi\cons$ is a conservative morphism.
By construction, $\cK_{\phi\loc}=\cK_\phi$.
\end{proposition}
\begin{proof}
The construction above and \cref{lem:cons} prove that the factorization exists.
Let us see that it is unique.
Using the universal property of localizations, it is sufficient to prove that 
$\cK_{\phi\loc} = \cK_\phi$ for any such factorization $\phi = \phi \cons \circ \phi \loc$.
But, using that $\phi\cons$ is conservative we have
\[
f\in \cK_{\phi\loc} 
\quad\iff\quad
\phi\loc(f)\in\Iso
\quad\iff\quad
\phi(f)=\phi\cons(\phi\loc(f))\in\Iso
\quad\iff\quad
f\in\cK_\phi\,.
\qedhere
\]
\end{proof}

For any algebraic morphism of topoi $\phi:\cE\to \cF$, we proved in \cref{topcongaresmall} that the congruence $\cK_\phi\topo$ is of small generation.
Therefore the localization $\phi\loc:\cE\to \cE\Forcing{\cK_\phi\topo}\Iso$ exists by \cref{Luriethmloc} and using the universal property of the localization $\phi\loc$,
we have a factorization
\[
\begin{tikzcd}
\cE \ar[rr,"\phi"] \ar[rd,"\phi\topo"']
&&\cF \\
& \cE\Forcing{\cK_\phi\topo}\Iso \ar[ru,"\phi\cotop"']
\end{tikzcd}
\]

\begin{lemma}
\label{lem:cotop}
The functor $\phi\cotop$ is cotopological.
\end{lemma}
\begin{proof}
We use condition \eqref{morphismvslemmacotop:1} of \cref{morphismvslemmacotop}.
Let $m$ be a monomorphism in $\cE\Forcing{\cK_\phi\topo}\Iso$ such that $\phi\cotop(m)$ is invertible.
We prove that $m$ is invertible.
By \cref{Luriethmloc}, the localization $\phi\loc:\cE\to \cE\Forcing{\cK_\phi\topo}\Iso$ is reflective.
Let $\phi_*$ be the right adjoint to $\phi$, then the map $m'=\phi_*(m)$ is a monomorphism in $\cE$ such that $\phi(m') = m$.
The map $\phi(m') = \phi\cotop(\phi\topo(m')) = \phi\cotop(m)$ is invertible by assumption on $m$, hence $m'\in \cK_\phi\cap \Mono \subseteq \cK_\phi\topo$.
Therefore $m=\phi(m')$ is invertible in $\cE\Forcing{\cK_\phi\topo}\Iso$.
\end{proof}

The following result is a mild generalization of \cite[Proposition 6.5.2.19]{Lurie:HTT}.

\begin{proposition}[Topological--cotopological factorization of a morphism]
\label{facttopcotop} 
Every algebraic morphism of topoi $\phi:\cE\to \cF$ admits a factorization in the category $\Toposalg$
\begin{equation}
\label{factorizationtopcotopeq}
\begin{tikzcd}
\cE \ar[rr,"\phi"] \ar[rd,"\phi\topo"']&& \cF \\
& \cE' \ar[ru,"\phi\cotop"']
\end{tikzcd}
\end{equation}
where $\phi\topo$ is a topological localization and $\phi\cotop$ is a cotopological morphism.
By construction, $\cK_{\phi\topo}=(\cK_\phi)\topo$.
The factorization is {\it unique} (up to unique isomorphism).
\end{proposition}
\begin{proof}
The construction above and \cref{lem:cotop} prove that the factorization exists.
We need to show that it is unique.
Using the universal property of localizations, it is sufficient to prove that 
$\cK_{\phi\topo}=(\cK_\phi)\topo$ for any such factorization $\phi = \phi \cotop \circ \phi \topo$.
By commutation of the triangle \cref{factorizationtopcotopeq}, we have $\cK_{\phi\topo}\subseteq \cK_\phi$.
Since $\cK_{\phi\topo}$ is topological by assumption, we have
$\cK_{\phi\topo}\subseteq \cK_\phi\topo$ by \cref{topcore}.
Conversely, we need to prove $\cK_\phi\topo\subseteq\cK_{\phi\topo}$.
Since $\cK_\phi\topo=(\cK\cap\Mono)\ac$ and $\cK_{\phi\topo}$ is a congruence thus acyclic, 
it is enough to show $\cK_\phi\cap\Mono\subseteq\cK_{\phi\topo}$.
Let $m$ be a monomorphism in $\cE$ such that $\phi(m)$ is invertible in $\cF$.
Then $\phi\topo(m)$ is a monomorphism inverted by $\phi\cotop:\cE'\to \cF$.
Since $\phi\cotop$ is assumed to be cotopological, we have that $\phi\topo(m)$ is invertible in $\cE'$ by \cref{morphismvslemmacotop}~\eqref{morphismvslemmacotop:1}.
Hence $\cK_\phi\cap\Mono\subseteq\cK_{\phi\topo}$ and the unicity of the factorization.
\end{proof}

\begin{remark}[Topological--cotopological factorization of a localization]
\label{rem:meaningtopcotopfacto}

When $\phi$ is a localization, the factorization of \cref{facttopcotop} recovers the factorization of \cite[Proposition 6.5.2.19]{Lurie:HTT}.
\[
\begin{tikzcd}
\cE \ar[rr,"\phi"] \ar[rd,"\phi\topo"']&&\cE\Forcing\cK\Iso \\
& \cE\Forcing{\cK\topo}\Iso \ar[ru,"\phi\cotop"']
\end{tikzcd}
\]
We can now provide an interpretation of this factorization.
We saw in \cref{thm:meaningtopcotopfacto} that the topological part of the localization, which is the localization along $\cK\topo$, forces the maps in $\cK$ to be \oo connected, which is weaker than forcing them to be invertible.
Then, the cotopological part, which is the localization with respect to $\phi\topo(\cK)=\phi(\cK)$ by \cref{prop:gen-cotop}, inverts the resulting \oo connected maps, which fully inverts the maps in $\cK$.
\end{remark}

\begin{examples}
\label{ex:topcotop}
The computations of topological parts of \cref{ex:toppart} provide examples of topological--cotopological factorizations.
\begin{examplenum}
\item\label{ex:topcotop:ooconn}
\[
\begin{tikzcd}
\S X \ar[rr,"ev_1"] \ar[rd,"ev_1\topo"']&&\cS = \S X\Forcing{X\to 1}\Iso \\
& \S X\Forcing{X\to 1}{\Conn\infty} \ar[ru,"ev_1\cotop"']
\end{tikzcd}
\]

\item\label{ex:topcotop:ntrunc}
\[
\begin{tikzcd}
\S X \ar[rr,"\phi_n"] \ar[rd,"\phi_n\topo"']&&\S {X\truncated n} = \S X\Forcing{X\to P_n}\Iso \\
& \S X\Forcing{X\to P_nX}{\Conn\infty} \ar[ru,"\phi_n\cotop"']
\end{tikzcd}
\]

\end{examplenum}
\end{examples}

\medskip

More generally, the factorization of \cref{facttopcotop} applied to the localization part $\phi\loc:\cE\to \cE\Forcing\cK\Iso$ of an arbitrary algebraic morphism $\phi:\cE\to\cF$ from \cref{factloccons} provides a triple factorization
\begin{equation}
\label{eq:triplefact}
\begin{tikzcd}
\cE
\ar[rr,"\phi"] \ar[dd,"\phi\topo"']
\ar[rrdd,"\phi\loc"', near start]
&&\cF
\\
\\
\cE\Forcing{\cK\topo}\Iso
\ar[rr,"(\phi\loc)\cotop"']
\ar[rruu,"\phi\cotop"', crossing over, near end]
&& \cE\Forcing\cK\Iso \ar[uu,"{\phi\cons}"']\,.
\end{tikzcd}
\end{equation}

\begin{proposition}
We have $\phi\cons \circ (\phi\loc)\cotop = \phi\cotop$ and the square \eqref{eq:triplefact} commutes.
\end{proposition}
\begin{proof}
By \cref{cor:cons}, the morphism $\phi\cons$ is cotopological.
Hence $\phi\cons \circ (\phi\loc)\cotop$ is cotopological by \cref{compositecotop}.
Then it must coincide with $\phi\cotop$ by unicity of the factorization of \cref{facttopcotop}.
The same argument of unicity proves that the whole diagram commutes.
\end{proof}

\begin{remark}
\label{rem:image}
The triple factorization of the square \eqref{eq:triplefact} is an instance of a ternary factorization system.
Geometrically, this factorization provides a refinement of the image factorization.
The factorization $\cE \to \cE\Forcing{\cK}\Iso \to \cF$, into a localization followed by a conservative morphism, is the naive analogue of the image factorization of 1-topoi. 
The factorization $\cE \to \cE\Forcing{\cK\topo}\Iso \to \cF$, into a topological localization followed by a cotopological morphism, correspond to a factorization into a stricter notion of image and a looser notion of surjective map.
\end{remark}

\section{Hypercomplete congruences}
\label{sec:hypercongruence}

\subsection{Hypercoverings and hypercomplete congruences} 
\label{sec:hypercoverings}

This section studies hypercomplete localizations.
We introduce the notion of a hypercomplete congruence \cref{def:hypercongruence}, which is a congruence closed under the construction of hypercoverings \cref{def:hypercovering}.
We prove in \cref{prop:recog-hyper} that a localization is hypercomplete if and only if the corresponding congruence is hypercomplete.
Any congruence $\cK$ can be completed into a hypercomplete one $\HCover\cK$ (\cref{thm:hypercompletion}) and the localization $\cE\Forcing{\HCover\cK}\Iso$ is the hypercompletion of the localization $\cE\Forcing\cK\Iso$ (\cref{prop:hypercompletion}).
Another important result of the section is \cref{thm:equivPtopHcong} in which a bijection between hypercomplete congruences and covering topologies is constructed.

\bigskip
Recall the definition of surjective maps (\cref{def:surjection}), 
$n$-connected maps (\cref{def:connected}),
and \oo connected maps (\cref{def:oo-connected}).

\begin{definition}[Coverings and hypercoverings]
\label{def:hypercovering}
\noindent Let $\cK$ be a congruence on a topos $\cE$.
\begin{enumerate}

\item\label{def:hypercovering:covsieve}
    A monomorphism $m:X\to Y$ is called a {\it $\cK$-covering sieve} 
    (or a {\it $\cK$-dense monomorphism})
    if $m \in \cK$.
The class of $\cK$-covering sieves is then $\cK\cap \Mono$.
By \cref{thm:adjGtopCong}, this is the largest extended Grothendieck topology contained in $\cK$.

\item\label{def:hypercovering:cov}
    A map $f:X\to Y$ is a {\it $\cK$-covering} if its image $\im f$ is a $\cK$-covering sieve (or equivalently if $\im f \in \cK$).
We denote by $\Cover \cK$ the class of all $\cK$-coverings.

\item\label{def:hypercovering:hypercov}
    A map $f:X\to Y$ is a {\it $\cK$-hypercovering} if all diagonal $\Delta^n f$ ($n\geq 0$) are $\cK$-covering (or, equivalently, if $\im {\Delta^nf} \in \cK$ for all $n\geq 0$).
This notion is the natural generalization of the notion of {\it morphisme bicouvrant} from \cite[II.5.2]{SGA41}.
Using the notation from \cref{sec:acvcong}, the class of $\cK$-hypercoverings is $\decinfty {\Cover\cK}$.
We shall denote it simply by $\HCover\cK$.
When the congruence is associated to an extended Grothendieck topology $\cG$, we shall write $\HCover\cG$ for $\HCover{(\cG\ac)}$.
Since a congruence is closed under diagonals, we have always $\cK\subseteq \HCover\cK$.

\end{enumerate}

When $\cK_\phi$ is the congruence associated to a left-exact localization $\phi:\cE\to \cF$,
we shall simply say $\phi$-covering sieves, $\phi$-coverings, and $\phi$-hypercoverings, 
instead of $\cK_\phi$-covering sieves, $\cK_\phi$-coverings, and $\cK_\phi$-hypercoverings.
\end{definition}

\begin{examples}
\label{ex:HCov}
In a topos $\cE$, 
the class of $\Iso$-coverings is the class $\Surj$ of surjections, 
and the class of $\Iso$-hypercoverings is the class $\Conn \infty$ of \oo connected maps.
Since $\Iso$ is the smallest congruence, for any congruence $\cK$, we have always
\[
\Surj\ \subseteq\ \Cover\cK
\qquad\textrm{and}\qquad
\Conn \infty\ \subseteq\ \HCover\cK\,.
\]
\end{examples}

\begin{remark}
Summarizing our various constructions, we can associate four objects to a congruence $\cK$:
\[
\cK\cap \Mono
\quad\subseteq\quad
\cK\topo
\quad\subseteq\quad
\cK
\quad\subseteq\quad
\HCover\cK
\quad\subseteq\quad
\Cover\cK\,.
\]
The class $\cK\cap \Mono$ is an extended Grothendieck topology (\cref{definitionGrothtop}),
$\cK\topo = (\cK\cap \Mono)\ac$ is a topological congruence (\cref{deftopcong}),
we will see below that $\HCover\cK$ is a hypercomplete congruence (see \cref{def:hypercongruence} below)
and that $\Cover\cK$ is a covering topology (\cref{def:covering-top}).
The four objects associated to $\cK$ determine each other (but they do not determine $\cK$).    
\end{remark}

\begin{lemma}
\label{lem:hyper}
Let $\phi:\cE\to \cF$ be a left-exact localization.
\begin{enumerate}
\item\label{lem:hyper:1} A monomorphism in $\cE$ is a $\phi$-covering sieve if and only if $\phi(f)$ is invertible.
\item\label{lem:hyper:2} A map $f$ in $\cE$ is a $\phi$-covering if and only if $\phi(f)$ is a surjection in $\cF$.
\item\label{lem:hyper:3} A map $f$ in $\cE$ is $\phi$-hypercovering if and only if $\phi(f)$ is \oo connected in $\cF$.
\end{enumerate}    
\end{lemma}
\begin{proof}
\eqref{lem:hyper:1} Direct.

\smallskip
\noindent \eqref{lem:hyper:2}
Let $\im f\circ \coim f$ be the image factorization of a map $f:X\to Y$ into a surjection followed by a monomorphism.
Then, $f$ is a surjection if and only if $\im f\in \Iso$.
Any left-exact cocontinuous functor $\cE\to \cF$ between topoi preserves monomorphisms and surjections and therefore image factorizations (see \cref{sec:connected}).
Now, from $\phi(\im f) = \im{\phi(f)}$, we get that $f$ is a $\phi$-covering if and only if $\phi(\im f)\in \Iso$ if and only if $\im f\in \cK_\phi$.

\smallskip
\noindent \eqref{lem:hyper:3}
A map $f$ is \oo connected if and only if all its iterated diagonals $\Delta^nf$ are surjective (see \cref{sec:connected}).
Then, \eqref{lem:hyper:3} is a consequence of \eqref{lem:hyper:2}.
\end{proof}

\begin{proposition}
\label{lem:fromCongtoTop}
Let $\cK$ be a congruence on a topos $\cE$.
\begin{enumerate}
\item\label{lem:fromCongtoTop:1} The class of $\cK$-covering sieves is the extended Grothendieck topology $\cK\cap\Mono = \im \cK$ associated to $\cK$ in \cref{fromacyclictoGroth}.
\item\label{lem:fromCongtoTop:2} The class $\Cover\cK$ of $\cK$-coverings is the covering topology $\Cover{(\cK\cap \Mono)}$ associated to $\cK$ in \cref{prop:FromGrothtoCov}~\eqref{prop:FromGrothtoCov:1}.
\item\label{lem:fromCongtoTop:3} The class $\HCover\cK$ of $\cK$-hypercoverings is a congruence.
\end{enumerate}
\end{proposition}
\begin{proof}
\eqref{lem:fromCongtoTop:1} Clear by definition.

\smallskip
\noindent\eqref{lem:fromCongtoTop:2} 
For a map $f$, we have $\im f\in \cK \Leftrightarrow \im f \in \cK\cap \Mono$.
This proves that the classes $\Cover \cK$ and $\Cover{(\cK\cap \Mono)}$ (see \cref{def:cover}) are the same.

\smallskip
\noindent\eqref{lem:fromCongtoTop:3}
Using $\HCover\cK = \decinfty {\Cover\cK}$, the result follows from \cref{lem:decalage}~\eqref{lem:decalage:2}.
\end{proof}

\begin{definition}[Hypercomplete congruences]
\label{def:hypercongruence}
A congruence $\cK$ is called {\it hypercomplete} 
if it contains all its hypercoverings, that is if $\cK = \HCover\cK$.
For a topos $\cE$, we denote by $\HCong(\cE)\subseteq \Cong(\cE)$ the subposet spanned by hypercomplete congruences of $\cE$.
\end{definition}

\begin{examples}
\label{ex:HCong}
We give some examples of hypercomplete congruences.
\begin{examplenum}
\item\label{ex:HCong:0}
    In a topos $\cE$, the class $\Conn \infty$ is the smallest hypercomplete congruence, and the class $\All$ of all maps is the largest.

\item\label{ex:HCong:1}
If $\phi:\cE\to \cF$ is an algebraic morphism of topoi, and $\cK$ is a hypercomplete congruence on $\cF$, then $\phi^{-1}(\cK)$ is a hypercomplete congruence on $\cF$.
In particular, the class $\phi^{-1}(\Conn \infty)$ is a hypercomplete congruence on $\cE$.

\item\label{ex:HCong:2}
We shall see in \cref{prop:recog-hyper} that a congruence of small generation $\cK$ is hypercomplete if and only if the corresponding topos $\cE\Forcing\cK\Iso$ is hypercomplete  (\cref{def:oo-connected}).

\item\label{ex:HCong:3}
    Any intersection of hypercomplete congruence is a hypercomplete congruence.

\end{examplenum}
\end{examples}

\begin{lemma}
\label{lem:thcov=iso}
Let $f$ be a truncated map in $\cE$, and $\cK$ be a congruence on $\cE$.
Then $f$ is a $\cK$-hypercovering if and only if $f$ is in $\cK$.
In other words, for any $n$, we have
\[
\cK \cap \Trunc n
\  =\ 
\HCover\cK \cap\Trunc n
\]
In particular, we have $\cK\cap \Mono = \HCover\cK \cap \Mono$.
\end{lemma}
\begin{proof}
Recall from \cref{sec:truncated} that a map $f$ is $n$-truncated if and only if the iterated diagonal $\Delta^{n+2}f$ is invertible, and that it is $n$-connected if and only if the diagonal maps $\Delta^kf$ are all surjective for $0\leq k\leq n+1$.
Recall also that a map which is $n$-truncated and $n$-connected is invertible.
Let $f$ be an $n$-truncated $\cK$-hypercovering and let $\cK_f$ be the smallest congruence containing all the maps $\im{\Delta^k f}$, for $0\leq k\leq n+1$.
By hypothesis on $f$, we have $\cK_f \subseteq \cK$.
The congruence $\cK_f$ is of small generation and the forcing condition $\Forcing{\cK_f}\Iso$ is representable.
The image of $f$ is $n$-truncated and \oo connected in $\cE\Forcing{\cK_f}\Iso$, thus invertible.
Therefore $f \in \cK_f$ and this proves that $f\in \cK$.
\end{proof}

\begin{lemma}
\label{lem:hypercompletion}
For any congruence $\cK$, the congruence $\HCover\cK$ is hypercomplete, that is
\[
\HCover{(\HCover \cK)}
\ =\ 
\HCover\cK\ .
\]
\end{lemma}
\begin{proof}
$\HCover\cK$ is a congruence by \cref{lem:fromCongtoTop}~\eqref{lem:fromCongtoTop:3}.
Let us see that it is hypercomplete.
We have
\[
\HCover \cK 
\  =\  
\big\{ f \,|\, \forall n\geq 0,\ \im{\Delta^n f}\in \cK \big\}
\qquad\text{and}
\]
\[
\HCover{(\HCover \cK)}
\  =\  
\big\{ f \,|\, \forall n\geq 0,\ \im{\Delta^n f}\in \HCover\cK\big\}\,.
\]
Using \cref{lem:thcov=iso} for $n=-1$, we get $\HCover\cK\cap \Mono = \cK\cap \Mono$, and the equality.
\end{proof}

\begin{proposition}
\label{thm:hypercompletion}
The map $\cK\mapsto\HCover\cK$ provides a left adjoint for the inclusion $\HCong(\cE)\subseteq \Cong(\cE)$.
\[
\begin{tikzcd}
\HCong(\cE)
\ar[rr, "i"',shift right = 1.6, hook]
\ar[from =rr,"{{\HCover{(-)}}}"', shift right = 1.6]
&&\Cong(\cE)
\end{tikzcd}
\]
\end{proposition}
\begin{proof}
The morphism $\HCover{(-)}$ is well defined by \cref{lem:hypercompletion}.
Let $\cV$ be a hypercomplete congruence and $\cK$ be an arbitrary congruence.
We have
\[
\cK\ \subseteq\ \cV 
\quad\Rightarrow\quad
\HCover\cK\ \subseteq\ \HCover\cV = \cV
\]
Conversely, if $\HCover\cK \subseteq \cV$ then we get $\cK \subseteq \cV$ using that $\cK\subseteq \HCover\cK$.
\end{proof}

Recall from \cite[Section~6.2.5]{Lurie:HTT} the following facts.
A topos $\cE$ is {\it hypercomplete} if it has no nontrivial \oo connected map, that is if $\Conn\infty(\cE)=\Iso$.
The category of hypercomplete topoi is reflective in the category of all topoi, and the reflection of a topos $\cE$ is given by the localization $\cE\to \cE\Forcing{\Conn\infty}\Iso$.
The following result justifies the name for the notion of hypercomplete congruence.

\begin{proposition}
\label{prop:hypercompletion}
For any congruence of small generation $\cK$ in a topos $\cE$, the localization $\cE\Forcing{\HCover\cK}\Iso$ is the hypercompletion of $\cE\Forcing\cK\Iso$.
\end{proposition}
\begin{proof}
Let $\cK$ be a congruence of small generation in a topos $\cE$.
Then $\cE\Forcing\cK\Iso$ exists and its hypercompletion is 
$\big(\cE\Forcing\cK\Iso\big)\Forcing{\Conn\infty}\Iso$.
By \cref{lem:hyper}~\eqref{lem:hyper:3}, the congruence associated to the composition of localizations 
\[
\cE
\nstto \phi
\cE\Forcing\cK\Iso
\nstto \psi
\big(\cE\Forcing\cK\Iso\big)\Forcing{\Conn\infty}\Iso
\]
is 
\[
\phi^{-1}\big(\psi^{-1}(\Iso)\big)
\ =\ 
\phi^{-1}(\Conn\infty) = \HCover\cK\,.
\]
This proves that 
\[
\big(\cE\Forcing\cK\Iso\big)\Forcing{\Conn\infty}\Iso
\ =\ 
\cE\Forcing{\HCover\cK}\Iso\,.
\]
and that
\[
\cE\Forcing\cK\Iso
\nstto \psi
\cE\Forcing{\HCover\cK}\Iso\,.
\]
is the hypercompletion of $\cE\Forcing\cK\Iso$.
\end{proof}

\begin{lemma}
\label{lem:closurerelation1}
For any congruence $\cK$, we have the relation
\[
\HCover{(\cK\topo)}
\ =\ 
\HCover\cK\,.
\]
\end{lemma}
\begin{proof}
By \cref{cor:GacMono=G}, $\cK\topo\cap\Mono = \cK\cap \Mono$, then
\begin{align*}
\HCover{(\cK\topo)}
&\ =\  \{f\,|\, \forall n\geq 0,\im{\Delta^nf} \in \cK\topo\}\\
&\ =\  \{f\,|\, \forall n\geq 0,\im{\Delta^nf} \in \cK\cap \Mono\}\\
&\ =\  \{f\,|\, \forall n\geq 0,\im{\Delta^nf} \in \cK\}\\
&\ =\  \HCover\cK\,.
\qedhere
\end{align*}
\end{proof}

\begin{corollary}
\label{lem:hyperloctopopart}
For any congruence $\cK$ (not necessarily of small generation) the hypercompletion of $\cE\Forcing{\cK\topo}\Iso$ is $\cE\Forcing{\HCover\cK}\Iso$.
In particular, when $\cK$ is of small generation, the canonical morphism 
$\cE\Forcing{\cK\topo}\Iso\to \cE\Forcing\cK\Iso$ induces an equivalence between the hypercompletions.
\end{corollary}
\begin{proof}
We apply \cref{prop:hypercompletion} to the congruence $\cK\topo$, which is always of small generation by \cref{topcongaresmall}.
Then the first assertion follows from \cref{lem:closurerelation1}.
The second assertion is a direct consequence.
\end{proof}

\begin{remark}
The localization $\cE\Forcing{\cK\topo}\Iso\to \cE\Forcing\cK\Iso$ is cotopological by \cref{facttopcotop}.
The fact that it induces an equivalence between the hypercompletions is a general fact about cotopological localizations.
It is possible to prove that the category of hypercomplete topoi is the localization of the category of topoi inverting the cotopological localizations.
\end{remark}

The following result can be proved by showing that any hypercomplete congruence is of small generation (for presheaves categories this is what is done in \cite{TV:hag1}).
We deduce it from \cref{lem:hyperloctopopart} and use it to show that any hypercomplete congruence is of small generation.

\begin{proposition}
\label{lem:existhyperloc}
The localization with respect to any hypercomplete congruence exists.
\end{proposition}
\begin{proof}
Let $\cK$ be a hypercomplete congruence, then $\cK = \HCover{(\cK\topo)}$ by \cref{lem:closurerelation1}.
Then, the localization $\cE\Forcing\cK\Iso$ exists by \cref{lem:hyperloctopopart}.
\end{proof}

Recall the notion of a hypercomplete topos from \cref{def:oo-connected}.
\begin{corollary}
\label{prop:recog-hyper}
A congruence $\cK$ is hypercomplete if and only if 
it is of small generation and 
the corresponding localization $\cE\Forcing\cK\Iso$ is a hypercomplete topos.
\end{corollary}
\begin{proof}
Let $\cK$ be a hypercomplete congruence. 
The localization $\cE\Forcing \cK\Iso$ exists by \cref{lem:existhyperloc}.
We have seen in \cref{Luriethm2} that the congruence $\cK_\phi$ associated to an algebraic morphism $\phi:\cE\to \cF$ is always of small generation.
This shows that $\cK$ is always of small generation.
Using \cref{prop:hypercompletion}, the hypercompletion of $\cE\Forcing\cK\Iso$ is $\cE\Forcing{\HCover\cK}\Iso$ and the two coincide if and only if $\cK= \HCover\cK$.
\end{proof}

Recall from \cref{prop:largestcong=racine}, that any acyclic class $\cA$ contains a largest congruence
\[
\decinfty \cA
\ :=\ 
\{f\in \cE\,|\, \forall n\geq 0\,, \Delta^n f \in \cA\}\,.
\]

\begin{lemma}$\quad$
\label{lem:equivPtopHcong}
\begin{enumerate}
\item\label{lem:equivPtopHcong:1} For an acyclic class $\cA$, we have
\[
\Cover\cA
\ =\ 
\Cover{\decinfty\cA}\ .
\]

\item\label{lem:equivPtopHcong:2} For a congruence $\cK$, we have
\[
\HCover\cK
\ =\ 
\decinfty{\Cover \cK}\ .
\]
\end{enumerate}
\end{lemma}
\begin{proof}
\eqref{lem:equivPtopHcong:1}
By \cref{prop:racinetopo}, we know that $\decinfty\cA\cap \Mono = \cA\cap \Mono$. Hence,
\[
\Cover{\decinfty\cA}
\ =\ 
\{f\,|\,\im f\in\decinfty\cA \}
\ =\ 
\{f\,|\,\im f\in\cA \}
\ =\ 
\Cover\cA\ .
\]

\noindent\eqref{lem:equivPtopHcong:2}
By definition, $g\in \Cover\cK \Leftrightarrow \im g\in \cK$. 
Hence,
\[
\decinfty{\Cover \cK}
\ =\ 
\{f\,|\,\forall k\geq 0, \Delta^kf\in\Cover\cK\}
\ =\ 
\{f\,|\,\forall k\geq 0, \im{\Delta^kf}\in\cK\}
\ =\ 
\HCover\cK\ .
\qedhere
\]
\end{proof}

\begin{remark}
\Cref{lem:equivPtopHcong}~\eqref{lem:equivPtopHcong:2} provides a direct proof of 
\cref{lem:fromCongtoTop}~\eqref{lem:fromCongtoTop:3} (not using congruences of small generation).
\end{remark}

Recall from \cref{thm:adjCongAcyclic} that the inclusion of congruences in acyclic classes has a right adjoint $\cA\mapsto \decinfty\cA$.

\begin{theorem}[Equivalence covering topologies/hypercomplete congruences]
\label{thm:equivPtopHcong}
The maps $\cC \mapsto \decinfty\cC$ and $\cK\mapsto \Cover \cK$ define inverse isomorphisms between the poset $\HCong(\cE)$ of hypercomplete congruences and the poset $\CTop(\cE)$ of covering topologies.
\[
\begin{tikzcd}
\HCong(\cE)
\ar[rr, "\Cover{(-)}",shift left = 1.6]
\ar[from =rr,"{\decinfty-}", shift left = 1.6, "\simeq"']
&&\CTop(\cE)
\end{tikzcd}
\]

\end{theorem}
\begin{proof}
For any congruence $\cK$, the class $\Cover \cK$ is a covering topology by \cref{lem:fromCongtoTop}.
This proves that the top map is well defined.
For any acyclic class $\cC$ the class $\decinfty\cC$ is a congruence by \cref{thm:adjCongAcyclic}.
We need to see that it is a hypercomplete congruence when $\cC$ is a covering topology, that is when $\Surj\subseteq \cC$.
By definition, we have
\[
\decinfty\cC
\quad:=\quad
\big\{f\in \cE \ |\ \forall k\geq 0,\, \Delta^k f \in \cC\big\}\,.
\]
A map $f$ is a $\decinfty\cC$-hypercovering if and only if all the maps $\im{\Delta^kf}$ (for $k\geq 0$) are in $\decinfty\cC$.
By \cref{prop:racinetopo}, we know that $\decinfty\cC\cap \Mono = \cC\cap \Mono$.
Thus, $f$ is a $\decinfty\cC$-hypercovering if and only if all the maps $\im{\Delta^kf}$ are in $\cC$.
By hypothesis on $\cC$, all the maps $\coim{\Delta^kf}$ are also in $\cC$.
Acyclic classes are stable by composition, this shows that $f$ is a $\decinfty\cC$-hypercovering if and only if all the maps $\Delta^kf$ are in $\cC$ if and only if $f\in \decinfty\cC$.
This proves that the second morphism is well defined.
Let us see now that they are inverse to each other.
Using \cref{lem:equivPtopHcong}~\eqref{lem:equivPtopHcong:1} for a covering topology, we have 
\[
\Cover{\decinfty\cC}
\ =\ 
\Cover\cC
\ =\ 
\cC\ .
\]
Conversely, using \cref{lem:equivPtopHcong}~\eqref{lem:equivPtopHcong:2} for a hypercomplete congruence, we have 
\[
\decinfty{\Cover \cK}
\ =\ 
\HCover\cK
\ =\ 
\cK\ .
\qedhere
\]
\end{proof}

Composing the isomorphism of \cref{thm:equivPtopHcong} with the one of \cref{thm:covering-top} between covering topologies and extended Grothendieck topologies, we get the following generalization of \cite[Theorem 3.8.3]{TV:hag1}.

\begin{corollary}[Generalized To\"en--Vezzosi bijection]
\label{generalTV}
The adjunction
\[
\begin{tikzcd}
\GTop(\cE) \ar[rr,shift left = 1.6, "\HCover{((-)\ac)}"] \ar[from=rr,shift left = 1.6, "-\cap \Mono"]
&&\HCong(\cE)
\end{tikzcd}
\]    
is an isomorphism of posets.
\end{corollary}

\subsection{Topological v. Hypercomplete congruences}
\label{sec:mainthm}

In this section we put together our results to construct the bijection between topological congruences and hypercomplete congruences in \cref{thm:adjTCongHCong}.
All the equivalence results proved in this paper are summarized in \cref{rem:maindiag} and Diagram~\eqref{maindiag}.

\bigskip

For a topos $\cE$, we have shown in \cref{topcore,thm:hypercompletion} the existence of two adjunctions
\[
\begin{tikzcd}
\TCong(\cE) \ar[rr, "j", hook] \ar[from =rr,"(-)\topo", shift left = 3]
&&\Cong(\cE)\ar[from=rr, "i", hook'] \ar[rr,"{\HCover{(-)}}", shift left = 3]
&& \HCong(\cE)\,.
\end{tikzcd}
\]

\begin{lemma}
\label{lem:closurerelation}
For any congruence $\cK$, we have always the relations
\[
\HCover{(\cK\topo)} = \HCover\cK
\qquad\text{and}\qquad
(\HCover\cK)\topo = \cK\topo\,.
\]
\end{lemma}
\begin{proof}
The first relation is \cref{lem:closurerelation1}.
To get the second one, recall that $\HCover\cK \cap \Mono = \cK \cap \Mono$ by \cref{lem:thcov=iso} (for $n=-1$), then
\[
(\HCover\cK)\topo
\ =\ 
(\HCover\cK \cap \Mono)\ac
\ =\ 
(\cK \cap \Mono)\ac
\ =\ 
\cK\topo\,.
\qedhere
\]
\end{proof}

\begin{theorem}
\label{thm:adjTCongHCong}
The composite adjunction
\[
\begin{tikzcd}
\TCong(\cE)
\ar[rr,shift left = 1.6, "\HCover{j(-)}", "\simeq"']
\ar[from=rr,shift left = 1.6, "(i(-))\topo"]
&&\HCong(\cE)
\end{tikzcd}
\]    
is an isomorphism of posets.
In particular the poset $\HCong(\cE)$ is small.
\end{theorem}
\begin{proof}
Using \cref{lem:closurerelation}, we have
\begin{equation}
\label{adj1}
\HCover{j(i(\cK)\topo)}
\ =\ 
\HCover{(\cK\topo)}
\ =\ 
\HCover\cK
\ =\ 
\cK
\end{equation}
for any hypercomplete congruence $\cK$.
And we have
\begin{equation}
\label{adj2}
i(\HCover{j(\cK)})\topo
\ =\ 
(\HCover\cK)\topo
\ =\ 
\cK\topo
\ =\ 
\cK    
\end{equation}
for any topological congruence $\cK$.
Finally, the smallness assertion is a consequence of that of $\TCong(\cE)$ (see \cref{thm:equivTcongGtop}).
\end{proof}

\begin{remark}
\label{rem:maindiag}
Summarizing our results, we have proven the existence of the following diagram (to be compared with the one of \cref{rem:tripleadjacyclic}).
\begin{equation}
\label{maindiag}
\begin{tikzcd}
\TCong(\cE)
\ar[rr, "j", hook]
\ar[from =rr,"(-)\topo", shift left = 3]
&&
\Cong(\cE)
\ar[from=rr, "i", hook']
\ar[rr,"{\HCover{(-)}}", shift left = 3]
\ar[ddl,"-\cap \Mono" description]
\ar[ddr,"\Cover{(-)}" description]
&&
\HCong(\cE)
\\
\\
&
\GTop(\cE)
\ar[rr,equal,"\text{\cref{thm:covering-top}}"']
\ar[luu,equal,"\text{\cref{thm:equivTcongGtop}}","{\cG\mapsto\cG\ac}"']
&&\CTop(\cE)
\ar[ruu,equal,"\text{\cref{thm:equivPtopHcong}}"',"{\cC\mapsto\decinfty{\cC}}"]
\end{tikzcd}    
\end{equation}

Notice that we have two different ways to prove the bijection between $\TCong(\cE)$ and $\HCong(\cE)$: the upper path of \cref{thm:adjTCongHCong}, or the lower path of composing the bijections of \cref{thm:equivTcongGtop,thm:covering-top,thm:equivPtopHcong}.
Let us see that the two paths produce the same bijection $\TCong(\cE)\simeq\HCong(\cE)$.
From the left to the right, the lower path gives
\[
\cK
\quad\mapsto\quad
\cK\cap\Mono
\quad\mapsto\quad
\Cover{(\cK\cap \Mono)}
\quad\mapsto\quad
\decinfty{\Cover{(\cK\cap \Mono)}}
\]
and we need to check that this is $\cK\mapsto\HCover \cK$.
But this is \cref{lem:equivPtopHcong}~\eqref{lem:equivPtopHcong:2}.

And from the right to the left
\[
\cK
\quad\mapsto\quad
\Cover\cK
\quad\mapsto\quad
\Cover\cK\cap \Mono
\quad\mapsto\quad
(\Cover\cK\cap \Mono)\ac
\]
which we need to see is $\cK\mapsto \cK\topo= (\cK\cap \Mono)\ac$.
This follows from $\Cover\cK\cap \Mono = \cK\cap \Mono$:
by \cref{lem:fromCongtoTop}~\eqref{lem:fromCongtoTop:2}, we have $\Cover\cK = \Cover{(\cK\cap \Mono)}$,
and by \cref{thm:covering-top}, we have $\Cover{(\cK\cap \Mono)}\cap \Mono = \cK\cap \Mono$.

\end{remark}

We now state our interpretation of the equivalence between topological and hypercomplete congruences.
This result is to be compared with \cref{thm:tripleadj-acyclic}, see the proof.

\begin{theorem}
\label{thm:tripleadj}
The morphism of posets $t:=\im- = \Mono\cap -:\Cong(\cE) \to \GTop(\cE)$ admits 
\begin{enumerate}
\item a fully faithful left adjoint $j'$ whose image is the subposet $\TCong(\cE)$ of topological congruences, and
\item a fully faithful right adjoint $i'$ whose image is the subposet $\HCong(\cE)$ of hypercomplete congruences.
\end{enumerate}
\[
\begin{tikzcd}
\Cong(\cE) \ar[rr,"t" description] \ar[from=rr, shift left = 3,"i'",hook'] \ar[from=rr, shift right = 3,"j'"', hook']
&& \GTop(\cE)
\end{tikzcd}
\]
\end{theorem}
\begin{proof}
The triple adjunction of the statement is defined as the composition of the triple adjunctions of \cref{thm:adjCongAcyclic,thm:tripleadj-acyclic}
\[
\begin{tikzcd}
\Cong(\cE)
\ar[rr, hook]
\ar[from=rr, shift right = 3,"(-)\cong"']
\ar[from=rr, shift left = 3,"\decinfty-"]
&& \Acyclic(\cE)
\ar[rr, "\im-" description]
\ar[from=rr, shift right = 3,"(-)\ac"', hook']
\ar[from=rr, shift left = 3,"\Cover{(-)}", hook']
&&\GTop(\cE)
\end{tikzcd}\,.
\]
The composition of the middle morphisms is clearly $t$.
The composition of the left adjoint is $\cG\mapsto j'(\cG) := \cG\cong = \cG\ac$, and its image is all topological congruences by \cref{thm:equivTcongGtop}.
For an extended Grothendieck topology, we have $\cG\ac\cap \Mono = \cG$ by \cref{cor:GacMono=G} and this proves $i'$ is fully faithful.
The composition of the right adjoints is $\cG\mapsto i'(\cG):=\decinfty{\Cover\cG} = \HCover{(\cG\ac)}$ by \cref{thm:equivPtopHcong} and its image is all hypercomplete congruence by the same result.
The morphism $j'$ must be fully faithful since $i'$ is. 
This can also be seen directly using that $\HCover{(\cG\ac)}\cap \Mono = \cG$ by 
\cref{lem:thcov=iso} and \cref{cor:GacMono=G}.
\end{proof}

\begin{remark}
\label{rem:tripleadj:1}
The triple adjunction of \cref{thm:tripleadj} shows that an extended Grothendieck topology $\cG$ is associated to a whole spectrum of congruences, which is the fiber of the map $\Cong(\cE)\to \GTop(\cE)$ at $\cG$.
This fiber has a minimal element which is given by the topological congruence $\cG\ac$ and a maximal element given by the hypercomplete congruence $(\cG\ac)\hcov$.

\end{remark}

\begin{remark}
\label{rem:tripleadj:2}
Using the equivalences $\TCong(\cE) = \GTop(\cE) = \CTop(\cE) = \HCong(\cE)$
and Diagram~\eqref{maindiag}, the triple adjunction of \cref{thm:tripleadj} can be presented in other ways, more suited for some applications.
\[
\begin{tikzcd}
\Cong(\cE)
\ar[rr,"{(-)\topo}" description]
\ar[from=rr, shift left = 4,"{\HCover{(-)}}",hook']
\ar[from=rr, shift right = 4,"can."', hook']
&& \TCong(\cE)
\end{tikzcd}
\qquad
\begin{tikzcd}
\Cong(\cE)
\ar[rr,"{\HCover{(-)}}" description]
\ar[from=rr, shift left = 4,"can.",hook']
\ar[from=rr, shift right = 4,"{(-)\topo}"', hook']
&& \HCong(\cE)
\end{tikzcd}
\]
\[
\begin{tikzcd}
\Cong(\cE) \ar[rr,"t" description] \ar[from=rr, shift left = 3,"\HCover{((-)\ac)}",hook'] \ar[from=rr, shift right = 3,"(-)\ac"', hook']
&& \GTop(\cE)
\end{tikzcd}
\qquad
\begin{tikzcd}
\Cong(\cE)
\ar[rrr,"{\Cover{(-)}}" description]
\ar[from=rrr, shift left = 4,"{\decinfty-}",hook']
\ar[from=rrr, shift right = 4,"{(-\cap\Mono)\ac}"', hook']
&&& \CTop(\cE)
\end{tikzcd}
\]
\end{remark}

\subsection{Sheaves and hypersheaves}
\label{sec:GTsheaf}

In this last section, we define a notion of sheaf and hypersheaf for an extended Grothendieck topology (\cref{def:Gsheaf,def:GHsheaf}).
We prove that the category of sheaves is the left-exact localization generated by the topology (\cref{prop:Gsheaf}).
and that the category of hypersheaves is the hypercompletion of this left-exact localization (\cref{prop:GHsheaf}).

\bigskip

Let $\cG$ be an extended Grothendieck topology on a topos $\cE$.
We have associated to $\cG$ several objects:
\begin{enumerate}[label=--]
\item the covering topology $\Cover\cG$ (\cref{def:cover}).
\item the topological congruence $\cG\ac$ (\cref{ex:topcong:1});
\item the hypercomplete congruence $\HCover\cG$ (\cref{def:hypercovering}~\eqref{def:hypercovering:hypercov}).
\end{enumerate}
We recall some of the relations between these objects:
\[
\cG
\quad\subseteq\quad
\cG\ac
\quad\subseteq\quad
\HCover\cG
\quad\subseteq\quad
\Cover\cG\,;
\]
\begin{enumerate}[label=--]
\item $\im{\cG\ac} = \cG\ac \cap \Mono = \cG$ (\cref{cor:GacMono=G});
\item $\im{\Cover\cG} = \Cover\cG \cap \Mono = \cG$ (\cref{lem:covGmono=G});
\item $\im{\HCover\cG} = \HCover\cG \cap \Mono = \cG$ (\cref{lem:thcov=iso});
\item $\decinfty{\Cover\cG} = \HCover\cG$ (\cref{lem:equivPtopHcong});
\item $(\HCover\cG)\topo = \cG\ac$ (\cref{lem:closurerelation}).
\end{enumerate}

\medskip

\begin{definition}[Sheaf for a topology]
\label{def:Gsheaf}

Let $\cE$ be a topos and $\cG$ an extended Grothendieck topology on $\cE$.
We say that an object $X\in \cE$ is a {\it $\cG$-sheaf} if it is local for the class $\cG$.
The category of sheaves is defined as
\[
\Sh\cE\cG
\ :=\ 
\Loc \cE\cG\,.
\]
A sheaf for a Lawvere--Tierney topology or a covering topology can be defined as a sheaf for the associated extended Grothendieck topology (but we shall not need these notions).
\end{definition}

Every object of $\cE$ is a sheaf for the minimal extended Grothendieck topology $\cG=\Iso$,
and only the terminal object is a sheaf for the maximal extended Grothendieck topology $\cG=\Mono$.

\begin{remark}[Sheaf = $\Sigma$-sheaf]
\label{rem:Gsh=Sigmash}
The notion of sheaf for an extended Grothendieck topology $\cG$ is compatible with the notion of $\Sigma$-sheaf of \cref{def:ABFJmain}, since for an extended Grothendieck topology, we have $(\cG\diag)\bc = \cG$.
\end{remark}

\begin{proposition}[Universal property of sheaves]
\label{prop:Gsheaf}
The subcategory $\Sh\cE\cG \subseteq \cE$ is reflective, and the reflector $\rho:\cE\to \Sh\cE\cG$ is the topological localization with the following universal properties:
\[
\Sh\cE\cG
\ =\ 
\cE\Forcing \cG\Iso
\ =\ 
\cE\Forcing {\cG\ac}\Iso
\ =\ 
\cE\Forcing {\Cover\cG}\Surj
\ =\ 
\cE\Forcing {\HCover\cG}{\Conn\infty}\,.
\]
In particular, if $\cG=\cC\cap \Mono$ is the topology associated to a covering topology $\cC$, we have
$\Sh\cE{\cC\cap \Mono} = \cE\Forcing\cC\Surj$.
\end{proposition}
\begin{proof}
Let $\cG\ac$ be the topological congruence associated to $\cG$.
By \cref{cor:topcongexists}, we have $\cE\Forcing \cG\Iso=\Loc \cE{\cG\ac}$.
We need to prove that $\Loc \cE{\cG\ac} = \Loc \cE\cG$.
This is consequence of the description of $\Sigma\ac$ in terms of saturated classes of \cite[Corollary 3.2.19]{ABFJ:HS}.
This proves that $\cE\Forcing \cG\Iso=\Loc \cE\cG = \Sh\cE\cG$.
The equivalences of forcing conditions 
$
\cE\Forcing \cG\Iso
=
\cE\Forcing {\cG\ac}\Iso
=
\cE\Forcing {\Cover\cG}\Surj
$
are a consequence of \cref{thm:forcing} using $\im{\Cover{(\cG\ac)}}= \cG$.
The equivalence
$
\cE\Forcing {\cG\ac}\Iso
=
\cE\Forcing {\HCover\cG}{\Conn\infty}
$
is \cref{thm:topopart} using $(\HCover\cG)\topo = \cG\ac$.
Finally, the last assertion is a consequence of \cref{thm:covering-top}.
\end{proof}

\begin{remark}
It is also possible to prove the equivalence of forcing conditions
$\Forcing {\HCover\cG}{\Conn\infty} = \Forcing {\HCover\cG}\Surj$
using that the iterated diagonals of a hypercovering are hypercoverings.
\end{remark}

\medskip

The following notion of hypersheaf generalizes the idea of hyperdescent of \cite{TV:hag1}.

\begin{definition}[Hypersheaf for a topology]
\label{def:GHsheaf}
Let $\cE$ be a topos and $\cG$ an extended Grothendieck topology on $\cE$.
We say that an object $X\in \cE$ is a {\it $\cG$-hypersheaf} if it is local for all 
$\cG$-hypercoverings.
The category of hypersheaves is defined as
\[
\HSh\cE\cG
\ :=\ 
\Loc \cE{\HCover\cG}\,.
\]
\end{definition}

Any hypersheaf is a sheaf since $\cG\subseteq \HCover\cG$.

\smallskip

\begin{proposition}[Universal property of hypersheaves]
\label{prop:GHsheaf}
The subcategory $\HSh\cE\cG \subseteq \cE$ is reflective, and the reflector $\rho:\cE\to \HSh\cE\cG$ is the hypercomplete localization with the following universal property:
\[
\HSh\cE\cG
\ =\ 
\cE\Forcing {\HCover\cG}\Iso \,.
\]
Moreover, $\HSh\cE\cG$ is the hypercompletion of $\Sh\cE\cG$.
\end{proposition}
\begin{proof}
Localizations with respect to hypercomplete congruences exist by \cref{lem:existhyperloc}.
Therefore any hypercomplete congruence is of small generation by \cref{Luriethm2}.
Then the result is \cref{Luriethm1}.
The last statement is \cref{prop:hypercompletion}.
\end{proof}


\begin{thebibliography}{CORS20}

\bibitem[ABFJ18]{ABFJ:GC}
M.~Anel, G.~Biedermann, E.~Finster, and A.~Joyal, \emph{Goodwillie's calculus of functors and higher topos theory},
  Journal of Topology \textbf{11} (2018), no.~4, 1100--1132.

\bibitem[ABFJ20]{ABFJ:GBM}
\bysame, \emph{{A generalized Blakers--Massey theorem}}, Journal of Topology
  \textbf{13} (2020), no.~4, 1521--1553.

\bibitem[ABFJ22]{ABFJ:HS}
\bysame, \emph{{Left-exact localizations of $\infty$-topoi I: Higher sheaves}},
  Advances in Mathematics \textbf{400} (2022), 108268.

\bibitem[ABFJ23]{ABFJ:box}
\bysame, \emph{{Left-exact
  localizations of $\infty$-topoi III: The algebra of congruences and Goodwillie calculus}}, In
  preparation.

\bibitem[AGV72]{SGA41}
M.~Artin, A.~Grothendieck, and J.~L. Verdier, \emph{{(SGA4-1) Th\'eorie des
  topos et cohomologie \'etale des sch\'emas. Tome 1: Th\'eorie des topos,
  S\'eminaire de G\'eom\'etrie Alg\'ebrique du Bois-Marie 1963--1964 (SGA 4)}},
  {Lecture Notes in Mathematics}, no. 269, Springer, Berlin, Heidelberg, 1972.

\bibitem[AJ21]{Anel-Joyal:topo-logie}
M.~Anel and A.~Joyal, \emph{Topo-logie}, New Spaces in Mathematics: Formal and
  Conceptual Reflections (M.~Anel and G.~Catren, eds.), Cambridge University
  Press, 2021, pp.~155--257.

\bibitem[AL19]{Anel-Lejay:topos-exp}
M.~Anel and D.~Lejay, \emph{Exponentiable $\infty$-topoi (version 2)}, Preprint
  (2019),
  \href{http://mathieu.anel.free.fr/mat/doc/Anel-Lejay-Exponentiable-topoi.pdf}{on
  Anel's homepage}.

\bibitem[Cis19]{Cisinski}
D.-C. Cisinski, \emph{Higher categories and homotopical algebra}, Cambridge
  Studies in Advanced Mathematics, Cambridge University Press, 2019.

\bibitem[CORS20]{CORS}
J.~D. Christensen, M.~Opie, E.~Rijke, and L.~Scoccola, \emph{Localization in
  homotopy type theory}, Higher Structures \textbf{4} (2020), no.~1, 1--32.

\bibitem[CR22]{Christensen-Rijke}
J.~D. Christensen and E.~Rijke, \emph{Characterizations of modalities and lex
  modalities}, Journal of Pure and Applied Algebra \textbf{226} (2022), no.~3,
  106848.

\bibitem[Hoy19]{Hoyois:acyclic}
M.~Hoyois, \emph{On Quillen's plus construction},
  \href{https://hoyois.app.uni-regensburg.de/papers/acyclic.pdf}{Note}, 2019.

\bibitem[Joh02]{Johnstone:Elephant}
P.T. Johnstone, \emph{Sketches of an elephant: a topos theory compendium.
  {V}ol. 1 and 2}, Oxford Logic Guides, Oxford University Press, Oxford, 2002.

\bibitem[Lur09]{Lurie:HTT}
J.~Lurie, \emph{Higher topos theory}, Annals of Mathematics Studies, vol. 170,
  Princeton University Press, Princeton, NJ, 2009.

\bibitem[Lur17]{Lurie:SAG}
\bysame, \emph{Spectral algebraic geometry}, version of February 3, 2017,
  Online book \url{https://www.math.ias.edu/~lurie/papers/SAG-rootfile.pdf}.

\bibitem[Rap19]{Raptis:acyclic}
G.~Raptis, \emph{Some characterizations of acyclic maps}, Journal of Homotopy
  and Related Structures \textbf{14} (2019), 773 -- 785.

\bibitem[Rez19]{Rezk:Leeds}
C.~Rezk, \emph{{Lectures on Higher Topos theory (Leeds)}},
  \href{https://faculty.math.illinois.edu/~rezk/leeds-lectures-2019.pdf}{on
  Rezk's homepage}.

\bibitem[RSS19]{RSS}
E.~Rijke, M.~Shulman, and B.~Spitters, \emph{Modalities in homotopy type
  theory}, Logical Methods in Computer Science \textbf{16} (2019), no.~1.

\bibitem[RV21]{Riehl-Verity:EICT}
E.~Riehl and D.~Verity, \emph{Elements of $\infty$-category theory},
  \href{https://emilyriehl.github.io/files/elements.pdf}{on Riehl's homepage}.

\bibitem[TV05]{TV:hag1}
B.~To{\"e}n and G.~Vezzosi, \emph{{Homotopical algebraic geometry I: topos
  theory}}, Advances in Mathematics \textbf{193} (2005), no.~2, 257--372.

\end{thebibliography}

\providecommand{\bysame}{\leavevmode\hbox to3em{\hrulefill}\thinspace}
\providecommand{\MR}{\relax\ifhmode\unskip\space\fi MR }

\end{document}